\documentclass[11pt]{amsart}

\usepackage{amsmath,amssymb,amsthm,amscd}
\usepackage[a4paper,text={140mm,240mm},centering,headsep=5mm,footskip=10mm]{geometry}
\usepackage[matrix,arrow]{xy}
\usepackage{pst-all}

\numberwithin{equation}{section}

\theoremstyle{plain}
\newtheorem{theo}{Theorem}[section]
\newtheorem{lem}[theo]{Lemma}

\newtheorem{prop}[theo]{Proposition}

\newtheorem{coro}[theo]{Corollary}
\newtheorem{claim}[theo]{Claim}

\newcounter{zaehler} 
\newtheorem*{theointro}{Theorem \Roman{zaehler}\stepcounter{zaehler}}
\newtheorem*{corointro}{Corollary \Roman{zaehler}\stepcounter{zaehler}}
\newtheorem*{quest}{Question \Roman{zaehler}\stepcounter{zaehler}}

\theoremstyle{definition}
\newtheorem{defi}[theo]{Definition}

\theoremstyle{remark}
\newtheorem{rem}[theo]{Remark}

\newcommand{\ab}{{\rm ab}}
\newcommand{\C}{{\rm C}}
\newcommand{\W}{{\rm W}}

\newcommand{\Spec}{\mathrm{Spec}}

\newcommand{\Hom}{\text{\rm Hom}}

\newcommand{\Q}{{\mathbb Q}}
\newcommand{\Z}{{\mathbb Z}}
\newcommand{\F}{{\mathbb F}}
\newcommand{\Frob}{{\mathit{Frob}}}

\newcommand{\Gal}{\mathrm{Gal}}

\newcommand{\et}{\text{\rm \'et}}

\def\rsw{{\mathrm{rar}}}      
\def\rswM{{\mathrm{cc}}}   
\def\rswMXD{{\mathrm{cc}_{X,D}}}      
\def\art{{\mathrm{ar}}}      

\def\ord{{\mathrm{ord}}}      
\def\Coker{{\mathrm{Coker}}}
\def\Ker{{\mathrm{Ker}}}    
\def\cont{{\mathrm{cont}}}  
\def\div{{\mathrm{div}}}    
\def\Div{{\mathrm{Div}}}    
\def\dlog{d{\mathrm{log}}}  
\def\et{{\text{\'{e}t}}}    
\def\Gal{{\mathrm{Gal}}}    
\def\Hom{{\mathrm{Hom}}}    
\def\Image{{\mathrm{Image}}}   
\def\mod{\;{\mathrm{mod}}\;}    
\def\ord{{\mathrm{ord}}}    
\def\Proj{{\mathrm{Proj}}}  

\def\ch{{\mathrm{ch}}}    

\def\cF{{\mathcal F}}

\def\cL{{\mathcal L}}

\def\cO{{\mathcal O}}

\def\cC{{\mathcal C}}

\def\cP{{\mathcal P}}
\def\cB{{\mathcal B}}
\def\cE{{\mathcal E}}
\def\cL{{\mathcal{L}}}
\def\cExt{{\cE}xt}

\def\Lam{\Lambda}
\def\lam{\lambda}

\def\k{\kappa}
\def\ep{\epsilon}

\def\bZ{{\mathbb Z}}
\def\bQ{{\mathbb Q}}

\def\bP{{\mathbb P}}
\def\bF{{\mathbb F}}

\def\bA{{\mathbb A}}

\def\fm{{\frak{m}}}          
\def\fp{{\frak{p}}}

\def\fq{{\frak{q}}}          
\def\fmK{\fm_K}

\def\tX{\tilde{X}}

\def\tX{\tilde{X}}

\def\tC{\tilde{C}}
\def\tD{\tilde{D}}

\def\tbeta{\tilde{\beta}}
\def\talpha{\tilde{\alpha}}

\def\tomega{\tilde{\omega}}

\def\qz{{\bQ}/{\bZ}}

\def\nz{\bZ/n\bZ}

\def\pz{\bZ/p\bZ}

\def\psz{\bZ/p^s\bZ}

\def\indlim#1{\underset{{\underset{#1}{\longrightarrow}}}{\mathrm{lim}}\; }
\def\projlim#1{\underset{{\underset{#1}{\longleftarrow}}}{\mathrm{lim}}\; }

\def\rmapo#1{\overset{#1}{\longrightarrow}}

\def\isom{\overset{\cong}{\longrightarrow}}

\def\KM#1#2{K^M_{#1}(#2)}

\def\mlam{m_\lam}

\def\Clam{C_\lam}

\def\Cmu{C_\mu}

\def\Klam{K_{\lam}}

\def\Blam{B_\lambda}
\def\fmlam{\fm_\lambda}
\def\Alam{A_\lam}
\def\klam{k(\lam)}

\def\fillog#1{\mathrm{fil}^{\log}_{#1}}
\def\fil#1{\mathrm{fil}_{#1}}

\def\gr#1{\mathrm{gr}_{#1}}

\def\FH#1#2{\fil{#2}H^1(#1)}

\def\FHlog#1#2{\fillog{#2}H^1(#1)}

\def\FXCC#1{F^{(#1)}W(X,C)}
\def\FXdCC#1{F^{(#1)}W(X',C')}
\def\FtXCC#1{F^{(#1)}W(\tX,\tC)}
\def\FXCCd#1{F^{(#1)}W(X',C')}

\def\FXCCap#1{F_{\Cap}^{(#1)}W(X,C)}
\def\FXtCCap#1{F_{\Cap}^{(#1)}W(\tX,\tC)}
\def\FXCCapV#1{F_{\Cap{V}}^{(#1)}W(X,C)}

\def\tFXCC#1{F_{\cB}^{(#1)}W(X,C)}

\def\tFtXCC#1{F_{\cB}^{(#1)}W(\tX,\tC)}
\def\tFXdCC#1{F_{\cB}^{(#1)}W(X',C')}

\def\tFB{F_{\cB}}

\def\hFXCC#1{\widehat{F}^{(#1)}W(X,C)}

\def\hFXtCC#1{\widehat{F}^{(#1)}W(\tilde{X},\tilde{C})}
\def\hFXCCd#1{\widehat{F}^{(#1)}W(X',C')}

\def\hFXCCkp#1{\widehat{F}^{(#1)}W(X_{k'},C_{k'})}

\def\WU{W(U)}
\def\WUL{W(U_L)}
\def\WUkp{W(U_{k'})}

\def\hWU{\widehat{F}^{(1)}\WU}

\def\CXD{{\rm C}(X,D)}

\def\CXC{{\rm C}(X,C)}
\def\CXDC{{\rm C}(X,D-C)}

\def\ZX{Z_1(X)^+}

\def\ZXCD{\mathrm{Div}(X)^+}
\def\ZXCC{\mathrm{Div}(X,C)^+}

\def\ZXC{Z_1(X,C)^+}

\def\ZXCCt{\mathrm{Div}(\tX,\tC)^+}

\def\ZXCCtd{\mathrm{Div}(\tX',\tC')^+}

\def\divXx#1{\div_{X,x}(#1)}

\def\Zinf{Z_{\infty}}

\def\Zinf{Z_{\infty}}

\def\piabXD{\pi_1^{\rm ab}(X,D)}

\def\piabXC{\pi_1^{\rm ab}(X,C)}
\def\piab#1{\pi^{\rm ab}_1(#1)}

\def\phiD{\phi_{X,D}}

\def\phiDC{\phi_{X,D-C}}

\def\piD{\pi_D}
\def\TL{T_L}

\def\Fth{F_{\theta}}

\def\qaq{\quad\text{and}\quad}
\def\qfor{\quad\text{for }}
\def\qwith{\quad\text{with }}
\def\dlog#1{\displaystyle{\frac{d{#1}}{#1}}}

\def\sym#1#2{\{#1\}_{#2}}

\def\chBX{\widehat{\cB}_X}

\def\WX{\Omega_X^1}
\def\wC{\omega_{C}}

\def\Lx{\Lambda_x}

\def\Ap{A_\fp}
\def\Kp{K_\fp}
\def\kp{k(\fp)}
\def\K2#1{K_2(#1)}

\def\Liotaqo{L(\bF_{q})^o}

\def\XSig{X_\Sigma}
\def\USig{U_\Sigma}

\def\CSig{C_\Sigma}
\def\thSig{\theta_\Sigma}
\def\tXSig{\widetilde{X}_\Sigma}
\def\tCSig{\widetilde{C}_\Sigma}
\def\tthSig{\widetilde{\theta}_\Sigma}

\begin{document}

\title[Wild class field theory for varieties over finite fields]
{Chow group of $0$-cycles with modulus and higher dimensional class field theory}

\author{Moritz Kerz}
\author{Shuji Saito}
\address{Moritz Kerz\\
NWF I-Mathematik\\
Universit\"at Regensburg\\
93040 Regensburg\\
Germany}
\email{moritz.kerz@mathematik.uni-regensburg.de}

\address{Shuji Saito\\
Interactive Research Center of Science, 
Graduate School of Science and Engineering,
Tokyo Institute of Technology\\
Ookayama, Meguro\\
Tokyo 152-8551\\
Japan
}
\email{sshuji@msb.biglobe.ne.jp}

\begin{abstract}
One of the main results of this paper is a proof of the rank one case of an existence conjecture on
lisse $\overline{\Q}_\ell$-sheaves on a smooth variety $U$ over a finite field due to Deligne and Drinfeld.
The problem is translated into the language of higher dimensional class field theory over finite fields,
which describes the abelian fundamental group of $U$ by Chow groups of zero cycles with moduli. 
A key ingredient is the construction of a cycle theoretic avatar of refined Artin conductor in ramification theory
originally studied by Kazuya Kato.
\end{abstract}

\maketitle

\tableofcontents

\section*{Introduction}\label{intro}

\bigskip

One of the main results of this paper is a proof of the rank one case of an existence conjecture on
lisse $\overline{\Q}_\ell$-sheaves on a smooth variety $U$ over a finite field suggested
by Deligne \cite{EK},  which was
motivated by work of Drinfeld \cite{Dr}, see also \cite{D2}.
The conjecture says that a compatible system of lisse $\overline{\Q}_\ell$-sheaves
on the integral closed curves on $U$, satisfying a certain boundedness condition for ramification at
infinity, should arise from a lisse sheaf on $U$. It thus reduces the understanding of lisse $\overline{\Q}_\ell$-sheaves on $U$ to that of lisse $\overline{\Q}_\ell$-sheaves on curves on $U$. 
A precise formulation is as follows.

\subsection*{Skeleton sheaves}

Let $U$ be a smooth variety over a finite field $k$.
Choose a compactification $U\subset X$ with $X$ normal and proper over $k$ such that $X\setminus U$ is the support of an effective Cartier divisor on $X$.
Consider a family $(V_Z)_Z$, where $Z$ runs through all closed integral curves on $U$ and where
$V_Z$ is a semi-simple lisse $\overline{\Q}_\ell$-sheaf on the normalization $\tilde Z$ of $Z$. We
say that that the family $(V_Z)_Z$ is compatible if for two different curves $Z_1$,
$Z_2$ the sheaves $V_{Z_1}$ and $V_{Z_2}$ become isomorphic up to semi-simplification after
pullback to the finite scheme $(\tilde Z_1 \times_X \tilde Z_2)_{\rm red}$. 
Such compatible families are also called {\em skeleton sheaves} and have been studied by
Deligne and Drinfeld, see \cite{EK}.

Let $Z^N$ be the normalization of the closure of $Z$ in $X$, which is by definition the smooth compactification of $\tilde Z$. Let $\psi_Z:Z^N \to X$ be the natural map and $Z_\infty$ be the set of points $y\in Z^N$ such that $\psi_Z(y)\not\in U$.
We say that a skeleton sheaf $(V_Z)_Z$ has bounded ramification if there exists an
effective Cartier divisor $D$ on $X$ with $|D| \subset X\setminus U$ and such that 
for all integral curves $Z$ on $U$, the following inequality of Cartier divisors on $Z^N$ holds:
\[
\underset{y\in \Zinf}{\sum}\; \art_y(V_Z) [y] \;\leq \;\psi_Z^* D.
\]
Here $\psi_Z^*D$ is the pullback of $D$ by $\psi_Z$ and $\art_y(V_Z)$ 
is the local Artin conductor of $V_Z$ at the point $y$, see \cite{Se}.

\begin{quest}[Deligne]
For any skeleton sheaf  $(V_Z)_Z$ with bounded ramification, does there exist a lisse
$\overline{\Q}_\ell$-sheaf $V$ on $U$ such that  for all curves $Z$ on $U$, the restrictions of $V$ to $\tilde{Z}$
become isomorphic to $V_Z$ after semi-simplification? 
\end{quest}

In this paper we prove the following:

\begin{theointro}
Question~I has a positive answer in rank one and for $\ch (k)\ne 2$.
\end{theointro}
\medbreak

\subsection*{Class field theory}

Using classical class field theory for curves over finite fields, 
Theorem~II is translated into the language of higher dimensional class field theory over finite fields,
and follows from Theorem~III explained below.
Instead of the class group involving higher $K$-theory which was used in earlier work, see
\cite{KS} for example, we use a relative Chow group of zero cycles with modulus. 

The principal idea is to describe the abelian fundamental group $\piab U$ of $U$ in terms
of the Chow groups $\CXD$  with modulus $D$, where $D$ is an effective Cartier divisor with support $|D|$ in $X\setminus U$. We define
\[
\CXD = \Coker\Big(\underset{Z\subset U}{\bigoplus} \; k(Z)_{D}^\times \rmapo{\div_Z} Z_0(U)\Big),
\]
where $Z_0(U)$ is the group of zero-cycles on $U$ and
$Z$ ranges over the integral closed curves on $U$. Here
$\div_Z:k(Z)^\times \to Z_0(U) $ is the divisor map on the function field $k(Z)$. The group 
$k(Z)_{D}^\times$ is the congruence subgroup of elements of $k(Z)^\times$ which are
congruent to $1$ modulo the ideal $I_D=\cO_X (-D)$ at all infinite places of $k(Z)$. 

We define
\[
k(Z)_{D}^\times =\underset{y\in \Zinf}{\bigcap}\;\Ker\big(\cO_{Z^N,y}^\times \to 
(\cO_{Z^N,y}/I_D \cO_{Z^N,y})^\times \big)\; \subset k(Z)^\times,
\]
where $\cO_{Z^N,y}$ is the local ring of $Z^N$ at $y$. 
Thus $\CXD$ is an extension of the Chow group of zero-cycles of $U$ 
which has been used repeatedly, see \cite{ESV} \cite{LevWei} \cite{Rus}.
It is also an extension of Suslin's singular homology $H_0^{sing}(U,\bZ)$, see  \cite{SV} and Remark \ref{rem.Suslinhom} below. In case $\dim(X)=1$ it is the class group with modulus $D$ used in classical class field theory.

We then introduce our class group of $U$ as 
\[
\C(U):= \varprojlim_D \CXD\;,
\]
where $D$ runs through all effective Cartier divisors on $X$ with $|D|\subset X\setminus U$ 
and endow it with the inverse limit topology where $\CXD$ is endowed with the discrete topology.
We show that the topological group $\C(U)$ is independent of the compactification $X$ of $U$, and 
construct a continuous map of topological groups, called the reciprocity map,
\begin{equation*}
\rho_U: \C(U) \to \piab U,
\end{equation*}
such that its composite with the natural map $Z_0(U) \to \C(U)$ is induced by the Frobenius maps
$Frob_x:\bZ \to \piab U$ for closed points $x$ of $U$.
The reciprocity map induces a continuous map
\begin{equation}\label{recmap.intro}
\rho_U^0: \C(U)^0 \to \piab U^0.
\end{equation}
Here $\piab U^0=\Ker\big(\piab U \to \piab {\Spec(k)}\big)$ and 
$\C(U)^0=\Ker(\C(U)\rmapo {\deg} \bZ)$, where $\deg$ is induced by the degree map $Z_0(U)\to \bZ$.
Now our main result, see also Theorem \ref{CFT.mainthm}, is the following.

\begin{theointro}[Existence Theorem]
Assuming $\ch(k)\ne 2$,
$\rho_U^0$ is an isomorphism of topological groups.
\end{theointro}

In case $\dim(X)=1$ the theorem is one of the main results in classical class field theory.
In higher dimension a special case of Theorem~III, describing only the tame quotient of $\piab
U$, is shown in \cite{SS} (see also \cite{Wi} and \cite{KeSc}).

In \cite{KS} an analog of  Theorem~III  is shown with
$\C(U)$ replaced by a  different class group $\C^{KS}(U)$ explained below, which can be described in terms
of higher local fields associated to chains of subvarieties on a compactification $X$ of $U$. Recall $\C(U)$ is
defined only in terms of points and curves on $U$.
There is a factorization of the reciprocity map
\begin{equation}\label{factorization}
\C(U) \to \C^{KS}(U) \to  \piab U
\end{equation}
and the main result of \cite{KS} over a finite field, see also \cite[Thm.\ 6.2]{Ras}, is a direct consequence of our Theorem~III if $\ch(k)\not=2$.

\medbreak

Using ramification theory in local class field theory,  Theorem~III implies.

\begin{corointro} Assume $\ch(k)\ne 2$.
For an effective divisor $D$ on $X$ with $|D|\subset X\setminus U$, $\rho_U$ induces an isomorphism of finite groups
\[
\CXD^0 \xrightarrow{\sim} \piabXD^0,
\]
where $\piabXD$ is the quotient of $\piab X$ which classifies abelian
\'etale coverings of $U$ with ramification over $X\setminus U$ bounded by the divisor $D$.
\end{corointro}

The finiteness of $\CXD^0$ is equivalent to the rank one case of Deligne's finiteness theorem 
(see \cite[Thm.\ 8.1]{EK}). Our arguments yield an alternative proof of this finiteness result. 
\def\art{{\mathrm{ar}}} 

\subsection*{Ramification theory}

The Pontryagin dual $\fil D H^1(U)$ of $\piabXD$ is the group of continuous characters 
$\chi:\piab U\to \qz$ such that for any integral curve $Z\subset U$, 
its restriction $\chi|_{Z}: \piab {\tilde{Z}} \to \qz$ to the normalization $\tilde{Z}$ of $Z$ satisfies the following inequality of Cartier divisors on $Z^N$, the smooth compactification of $\tilde{Z}$:
\[
\underset{y\in \Zinf}{\sum}\; \art_y(\chi|_{Z}) [y] \;\leq \;\psi_Z^* D,
\]
where $\art_y(\chi|_{Z})\in \bZ_{\ge 0}$ is the Artin conductor of $\chi|_{Z}$ at $y\in \Zinf$ and $\psi_Z^* D$
is the pullback of $D$ by the natural map $\psi_Z :Z^N \to X$ (see Definition~\ref{gloram.def}). 
\smallbreak

Our proof of Theorem~III depends in an essential way on ramification theory due to 
Kato \cite{Ka1} and its variant by Matsuda \cite{Ma}.
Let $\Klam$ be the henselization of $K=k(U)$ at the generic point $\lam$ of an irreducible
component $\Clam$ of $X\setminus U$ and let $H^1(\Klam)$ be the group of continuous characters $G_{\Klam} \to \qz$, where $G_{\Klam}$
is the absolute Galois group of $\Klam$. 
They introduced a ramification filtration $\fil m H^1(\Klam)$ ($m\in \bZ_{\geq 0}$) on $H^1(\Klam)$ 
which generalizes the ramification filtration for local fields with perfect residue fields (see \cite{Se}), 
and defined a natural injective map 
\begin{equation}\label{refinedArtin_intro}
\rsw_{\Klam}: \fil m H^1(\Klam)/\fil {m-1} H^1(\Klam)\; \hookrightarrow \Omega^1_X(m
\Clam) \otimes_{\cO_X} k(\Clam) \quad (m>1)
\end{equation}
which we call refined Artin conductor (indeed what Kato originally defined is refined Swan conductor and we use a variant for Artin conductor introduced by Matsuda), where $k(\Clam)$ is the function field of $\Clam$.
In case $C=X\setminus U$ is a simple normal crossing divisor on smooth $X$, results from the ramification theory imply 
\[
\fil D H^1(U) = \Ker\big(H^1(U) \to \underset{\lam\in I}{\bigoplus}\; H^1(\Klam)/\fil {\mlam} H^1(\Klam)\big).
\]
Here $H^1(U)$ denotes the group of continuous characters $\chi:\piab U\to \qz$, 
$I$ is the set of the generic points $\lam$ of $C$ and $\mlam$ is the multiplicity of $\Clam$ in $D$.
\medbreak

Now the basic strategy of the proof of Theorem~III is as follows (see \S \ref{CFT} for the details).
By an argument due to Wiesend we are allowed to replace $X$ by an alteration $f:X'\to X$
and $U$ by a smooth open $U'\subset f^{-1}(U)$. Then a Lefschetz theorem for $\piabXD$
(cf. \cite{KeSLef}) reduces the proof to the case where $X$ is a smooth projective
surface and $C=X\setminus U$ is a simple normal crossing divisor. 
The proof then proceeds by induction on the multiplicity of $D$ reducing to the tame case $D=C$. 
A key point is the construction of a natural map, which we call the {\em cycle conductor}, defined for Cartier divisors $D$ such that $D \geq 2C$:

\begin{equation}\label{eq.cycleconductor}
\rswM_{X,D}: \CXD^\vee:=\Hom(\CXD,\qz)\; \to\; H^0(C,\Omega^1_X(D+ \Xi)\otimes_{\cO_X}\cO_C),
\end{equation}
where $\Xi\subset X$ is an auxiliary effective Cartier divisor independent of $D$ such that it contains
none of the irreducible components of $C$. It satisfies  
\begin{equation}\label{Kercc}
\Ker(\rswM_{X,D}) = \C(X,D-C)^\vee\subset \CXD^\vee,
\end{equation}
and the following diagram
\begin{equation*}
\xymatrix{
\FH U D \ar[r]\ar[d]^{\Psi_{X,D}}& \FH {\Klam} {\mlam} \ar[r]^{\hskip -30pt \rsw_{\Klam}} 
&\Omega^1_X(D)\otimes_{\cO_X} k(\Clam) \\
\CXD^\vee \ar[rr]^{\rswM_{X,D}} && H^0(C,\Omega^1_X(D+\Xi)\otimes_{\cO_X}\cO_C)\ar[u]\\
}
\end{equation*}
commutes. Here $\Psi_{X,D}$ is the dual of the reciprocity map $\CXD \to \piabXD$ induced by $\rho_U$.
Therefore we consider the cycle conductor $\rswM_{X,D}$ as a cycle theoretic avatar of the refined Artin conductor of 
Kato and Matsuda.
\bigskip

By duality, the definition of cycle conductors is reduced to the construction of a natural map
\begin{equation}\label{cycleconductor_intro}
\phi_{X,D}: H^1(C,\Omega^1_X(-D+C-\Xi)\otimes_{\cO_X}\cO_C) \to \CXD
\end{equation}
such that the following sequence is exact
\[
H^1(C,\Omega^1_X(-D+C-\Xi)\otimes_{\cO_X}\cO_C) \rmapo{\phi_{X,D}} \CXD \to \C(X,D-C) \to 0.
\]

\subsection*{General base fields}

Let now $k$ be an arbitrary perfect field of characteristic $p>0$ and let  $U$ be a smooth variety of dimension $d$ over $k$ as above. We pose the following question. 

\begin{quest}
Is the natural map 
\[
\tau_U: \C(U) \to \C^{KS}(U)
\]
to Kato--Saito class group over a general perfect field $k$ an isomorphism?
\end{quest}

Recall that the Kato--Saito class group is defined in terms of Nisnevich cohomology groups
\begin{equation}\label{CKSU}
\C^{KS}(U) = \varprojlim_{D}  H^d(X_{\rm Nis} , \mathcal K^M_d(X,D) ).
\end{equation}
Here $ \mathcal K^M_d(X,D)$ is the relative Milnor $K$-sheaf of \cite{KS2} and $D$ runs
through all effective Cartier divisors on $X$ with $|D|\subset X\setminus U$.
\medbreak

In case $k$ is finite, the main result of \cite{KS2} with \eqref{factorization} implies that the reciprocity map $\rho_U^0$ 
(cf. \eqref{recmap.intro}) is an isomorphism if and only if $\tau_U$ is an isomorphism.
One should think of Question~V and our main theorem as a particular case of a
  Nisnevich descent and a motivic duality statement in a conjectural world of mixed motives with modulus which still has to be developed (see \cite{KSY}). In view of this lack of
a conceptional framework to approach the problem, our construction of the cycle conductor
\eqref{eq.cycleconductor} has to be technical and ad hoc. 
\medbreak

We hope that our technique may be used to approach Question~I in the higher rank case by constructing 
a non-abelian version of the cycle conductor \eqref{eq.cycleconductor}, where $\CXD^\vee$ is replaced by the set of skeleton sheaves on $X$ of a higher rank with ramification bounded by $D$ and refined Artin conductor is 
replaced by its non-abelian version constructed in \cite{TSa}.

\medskip

\begin{center}
\rule{3cm}{0.5mm}
\end{center}

\medskip

We give an overview of the content of the paper.

In \S1 we introduce a class group $\WU$ studied by Wiesend \cite{Wi}.
We define some filtrations on $\WU$ and explain their basic properties. 
Our class group $\CXD$ can be defined as a quotient of $\WU$ by a certain filtration.
We also introduce a tool to produce relations in $\WU$.

In \S2 we review some results on ramification theory.
The first subsection treats local ramification theory for henselian discrete valuation fields whose residue fields are not
necessarily perfect, originally due to Kato \cite{Ka1}, \cite{Ka} and \cite{Ka3}.
The refined Artin conductor, see \eqref{refinedArtin_intro}, is introduced, which plays a key role in this paper.
In the second subsection, some implications on global ramification theory are given.

In \S3 the reciprocity map $\rho_U$ is defined and the Existence Theorem is stated.
The basic strategy of the proof of the Existence Theorem is explained.
We explain an argument due to Wiesend, which allows us to replace $X$ by an alteration.
We reduce the proof to the case $\dim(X)=2$ by using the Lefschetz theorem for abelian fundamental groups allowing ramification along some divisor, which is proved in \cite{KeSLef}.

In \S4 we introduce the map $\phi_{X,D}$, see \eqref{cycleconductor_intro}.
It is the dual of the cycle conductor which is a key ingredient of the proof of the Existence Theorem.
Two key theorems are stated (Theorem \ref{keythm} and Theorem \ref{keythm2}). 
The first theorem gives a characterization of $\phi_{X,D}$ by its local components 
which are defined for pairs $(x,Z)$ where $x$ is a regular closed point of $C$ and $Z\subset X$ is 
an integral curve which intersects transversally with $C$ at $x$. 
We also state a lemma (see Lemma \ref{movinglemfinite}) which implies the key property \eqref{Kercc}.
The second theorem states a compatibility of $\phi_{X,D}$ with the refined Artin conductor.
The proof of the Existence Theorem is completed in \S5 using these key theorems.

In \S\ref{recclosedpoint} we introduce a local component of $\phi_{X,D}$ which only depends on a 
regular closed point of $C$, but not on a curve as above. 
It is used in the proof of the first key Theorem \ref{keythm} given in \S\ref{keythmproof}
as well as that of the second key Theorem \ref{keythm2} given in \S\ref{keythmproofII}.

The proof of Theorem \ref{keythm} depends on three technical key lemmas 
(Lemma~\ref{movinglemfinite}, Lemma~\ref{lem.3term} and Lemma~\ref{keylemIIfinite}),
which are restated in \S\ref{keylemmas} over a general perfect field.
The proof of these lemma occupies the later sections \S10 through \S14. 
The tool to produce relations in $\WU$ introduced in \S1 will play a basic role in the proof.

\medskip

There is work related to our main results by H.\ Russell involving a geometric method based on his joint work
with K.\ Kato on Albanese varieties with modulus. 

\bigskip  

{\em Acknowledgments.}
We would like to thank A.\ Abbes and T.\ Saito for much advice and for improving our understanding of
ramification theory.
We are very grateful to the referee for numerous constructive comments which brought about
substantial improvement of the paper. 
We profited from discussions with H.\ Russell on different versions of relative Chow groups with modulus.
The first author learned about the ideas of Deligne and Drinfeld in his joint work with
H.\ Esnault. He would like to thank her cordially for this prolific collaboration.
The proof of the main theorem of this paper hinges on seminal work of Kato on ramification theory and 
class field theory for higher local fields. We would like to express our admiration of the depth of ideas in his work.


\maketitle

\bigskip

\section{Wiesend class group and filtrations}\label{classgroup}
\bigskip

In the whole paper we fix a perfect field $k$ with $\ch(k)=p>0$. At many places we have to assume $p\not=2$.
Let $X$  be a proper normal scheme over $k$ and $C$ be the support of an effective Cartier
divisor on $X$ and put $U=X\setminus C$.
Note that $C$ is a reduced closed subscheme of pure codimension one.
\begin{defi}\label{curve.def}
Let $Z_1(X)^+$ be the monoid of effective $1$-cycles on $X$ and 
\[
\ZXC\subset \ZX
\]
be the submonoid of the cycles $Z$ such that none of the prime components of $Z$ is contained in $C$.
Take 
\[Z=\underset{1\leq i\leq r}{\sum}\; n_i Z_i \;\in \ZX,
\]
where $Z_1,\dots,Z_r$ are the prime components of $Z$ and $n_i\in \bZ_{\geq 0}$. 
We write
\[
k(Z)^\times =k(Z_1)^\times\oplus\cdots\oplus k(Z_r)^\times,\quad
\]
\[
|Z|= \underset{1\leq i\leq r}{\cup}\; Z_i\;\subset X,\quad
I_Z=\underset{1\leq i\leq r}{\prod}\; (I_{Z_i})^{n_i}\;\subset \cO_X,
\]
where $I_{Z_i}\subset \cO_X$ is the ideal of $Z_i$.
We say $Z$ is reduced if $n_i=1$ for all $1\leq i\leq r$ and integral if it is reduced and $r=1$.
We say $Z$ intersects $C$ transversally at $x\in X$ (denoted by $Z\Cap C$ at $x$) 
if $|Z|$ and $C$ are regular at $x$ and the intersection multiplicity 
\[
(Z,C)_x:={\rm length}_{\cO_{X,x}}(\cO_{X,x}/I_Z+I_C)=1.
\]
We say $Z$ intersects $C$ transversally (denoted by $Z\Cap C$) 
if $Z\Cap C$ at all $x\in Z\cap C$.
\end{defi}

\begin{defi}\label{classgroup.def}
For $Z\in \ZXC$, let $\psi_Z:Z^N\to |Z|$ be the normalization and put 
\[
\Zinf=\{y\in Z^N\;|\; \psi_Z(x)\in |Z|\cap C\}.
\]
For $y\in \Zinf$, let $k(Z)_y$ be the henselization of $k(Z)$ at $y$ and put
\[
k(Z)_\infty =\underset{y\in \Zinf}{\prod}\; k(Z)_y
\qaq
k(Z)_\infty^\times =\underset{y\in \Zinf}{\prod}\; k(Z)_y^\times.
\]
The Wiesend class group of $U$ is defined as
\begin{equation}\label{WCG}
W(U)=\Coker\Big(\underset{Z\subset X}{\bigoplus} \; k(Z)^\times \xrightarrow{\delta} 
\underset{Z\subset X}{\bigoplus}\; k(Z)_\infty^\times\;\oplus\; Z_0(U)\Big),
\end{equation}
where $Z$ ranges over the integral elements of $\ZXC$ \cite{Wi}, \cite{KeSc}. Here we map $k(Z)^\times$
diagonally in  $k(Z)_\infty^\times$ 
 and 
$k(Z)^\times \to Z_0(U)$ is the composite map
\[
\k(Z)^\times \rmapo{\div_{Z\cap U}} Z_0(Z\cap U)\hookrightarrow Z_0(U).
\]
Obviously $W(U)$ depends only on $U$, i.e.\ is independent of $(X,C)$ such that
$U=X\setminus C$.
\end{defi}

For a morphism $f:U' \to U$ of smooth varieties there is a canonical induced morphism
\begin{equation}\label{eq.Wnorm}
f_*: W(U') \to W(U),
\end{equation}
see \cite{Wi} and \cite[Sec.\ 7]{KeSc}. If $f$ is finite we also speak of the norm map and write $N_f $ for $f_*$.

\begin{defi}\label{symbol.def}
Let $Z\in \ZXC$.
\begin{itemize}
\item[(1)]
Assume $Z$ integral. For $x\in Z\cap C$, we have the natural map
\[
\sym {\hbox{ }} {Z,x}: 
\underset{y\in\psi_Z^{-1}(x)}{\bigoplus}\; k(Z)_y^\times \to \WU
\]
where $\psi_Z:Z^N\to Z$ is the normalization.
Taking the sum of these maps for $x\in Z\cap C$, we get
\[
\sym {\hbox{ }} {Z}: 
k(Z)_\infty^\times \to \WU.
\]
\item[(2)]
In general we write $Z=\underset{i\in I}{\sum}e_i Z_i$ where $\{Z_i\}_{i\in I}$ are the 
prime components of $Z$ and $e_i\in \bZ_{\geq 0}$, and define
\[
\sym {\hbox{ }} {Z,x}=\underset{i\in I}{\sum}\; e_i\sym {\hbox{ }} {Z_i,x}\;;\;
\underset{y\in\psi_Z^{-1}(x)}{\bigoplus}\; k(Z)_y^\times \to \WU,
\]
\[
\sym {\hbox{ }} {Z}=\underset{i\in I}{\sum} \; e_i\sym {\hbox{ }} {Z_i}\;;\;
k(Z)_\infty^\times \to \WU.
\]
where $\psi_Z:Z^N\to |Z|$ is the normalization.
\end{itemize}
\end{defi}

Let the notation be as in Definition \ref{symbol.def} and $Z\in \ZXC$.
Write $\cO_Z=\cO_X/I_Z$ and 
$\cO_{Z,x}^h$ for the henselization of $\cO_{Z,x}$ for $x\in |Z|$. We also write 
\[
\cO_{Z,C\cap Z}^h =\underset{x\in Z\cap C}{\prod}\; \cO^h_{Z,x},
\quad
\cO_{Z^N,C\cap Z}^h=\cO_{Z,C\cap Z}^h\otimes_{\cO_{Z}} {\cO_{Z^N}}=
\underset{y\in\psi_Z^{-1}(Z\cap C)}{\prod} \cO_{Z^N,y}^h .
\]
We have the natural maps
\begin{equation*}
\cO_{Z,C\cap Z}^h  \to  \cO_{Z^N,C\cap Z}^h  \hookrightarrow k(Z)_\infty\\
\end{equation*}

\bigskip

\begin{defi}\label{filtration.def0}
Let $D$ be an effective Cartier divisor such that $|D|= C$
and let $I_D=\cO_X(-D)$ be the ideal sheaf of $D$.
\begin{itemize}
\item[(1)]
We define $\FXCC D \subset \WU$ as the subgroup generated by 
\[
\sym {1+I_D\cO_{Z,C\cap Z}^h}{Z}
\]
for all $Z\in \ZXC$.
\item[(2)]
We define $\hFXCC D \subset \WU$ as the subgroup generated by 
\[
\sym {1+I_D \cO_{Z^N,C\cap Z}^h}{Z}
\]
for all $Z\in \ZXC$.
Note $\FXCC D \subset \hFXCC D$.
\item[(3)]
We define $\hWU\subset \WU$ as the subgroup generated by 
\[
\sym {1+\fm\cO_{Z^N,C\cap Z}^h}{Z}
\]
for all $Z\in \ZXC$, where $\fm$ is the Jacobson radical of $\cO_{Z^N,C\cap Z}^h$.
Note that $\hWU$ depends only on $U$ and $\hFXCC D\subset \hWU$ if $|D|=C$ and $\hWU/\FXCC D$ is $p$-primary torsion.
\item[(4)]
For a dense open subset $V\subset X$ containing the generic points of $C$, we define 
\[
\FXCCapV D= \underset{G}{\sum} \;\sym{1+\cO^h_{G,G\cap C}(-D)}{G}\subset \FXCC {D},
\]
where $G$ ranges over $\ZXC$ such that $G\Cap C$ and $G\cap C\subset V$.
In case $V=X$ we simply denote $\FXCCapV D=\FXCCap D$.
\end{itemize}
\end{defi}

\begin{rem}\label{rem.Suslinhom}
It is shown in \cite[Thm 3.1]{Sc} that there is a natural isomorphism
\[
W(U)/\hWU \simeq H_0^{sing}(U,\mathbb Z),
\]
where the right hand side is Suslin's singular homology.
\end{rem}

\begin{defi}\label{def.CXD}
Under the notation of Definition \ref{filtration.def0}, we put
\[
\CXD =\WU/\hFXCC {D}.
\]
By the weak approximation theorem, we have an isomorphism
\[
\CXD \simeq \Coker\Big(\underset{Z\subset X}{\bigoplus} \; k(Z)_{D}^\times \to Z_0(U)\Big),
\]
where $Z$ ranges over the integral elements of $\ZXC$, and 
\[
\begin{aligned}
k(Z)^\times \supset k(Z)_{D}^\times 
&=\Ker\big(k(Z)^\times \to \underset{y\in \Zinf}{\prod}\; k(Z)_y^\times/1+I_{D}\cO_{Z^N,y}\big)\\
&=\underset{y\in \Zinf}{\bigcap}\;\Ker\big(\cO_{Z^N,y}^\times \to 
(\cO_{Z^N,y}/I_D \cO_{Z^N,y})^\times \big)\;.
\end{aligned}
\]
Thus $\CXD$ is an extension of the Chow group of zero-cycles of $U$.
\end{defi}

\begin{lem}\label{hFWnorm}
Let $f:X'\to X$ be a morphism with $f(X') \cap U \ne \varnothing$ and let $D$ be an effective Cartier divisor on $X$ with
$|D|\subset C$. Set $U' = f^{-1}(U)$ and
 $D'=f^*D$. Then the pushforward \eqref{eq.Wnorm} satisfies
 $f_*(\hFXCCd {D'})\subset \hFXCC D$.
\end{lem}
\begin{proof}
This is a direct consequence of the definition and standard properties of the norm map for
local fields. 
\end{proof} 
\bigskip

In what follows we assume $\dim(X)=2$. 

\begin{defi}\label{Cartierdiv.def}
Let $X$ be a projective smooth surface over $k$.
Let $\ZXCD$ be the monoid of effective Cartier divisors on $X$ and $\ZXCC\subset \ZXCD$ be 
the submonoid of such Cartier divisors $D$ that none of the prime components of $D$ is 
contained in $C$.
Then $\ZX$ coincides with $\ZXCD$ and $\ZXC$ coincides with $\ZXCC$.
\end{defi}

\begin{defi}\label{def.cC}
Let $\cC$ be the category of triples $(X,C)$, where 
\begin{itemize}
\item 
$X$ is a projective smooth surface over $k$,
\item
$C$ is a reduced Cartier divisor on $X$.
\end{itemize} 
A morphism $f:(X',C') \to (X,C)$ in $\cC$ is a surjective map $f:X'\to X$ of schemes 
such that $C'=f^{-1}(C)_{\rm red}$.
For $f$ as above and for $D\in \ZXCD$, we let $f^*D\in \ZXCD$ be the pullback of $D$ as a Cartier divisor. 
Let $\cC_X\subset \cC$ be the category of the objects and the morphisms in $\cC$ over $X$.
\end{defi}

\begin{defi}\label{def.cB}
Let $X=(X,C)$ be in $\cC$.
\begin{itemize}
\item[(1)]
Let $\chBX\subset \cC_X$ be the subcategory of the object $(\tX,\tC)$, where
$g:\tX \to X$ is the composite of successive blowups at closed points in the preimages of $C$. 
\item[(2)]
Let $\cB_X\subset \chBX$ be the subcategory of the object $(\tX,\tC)$, where
$g:\tX \to X$ is the composite of successive blowups at 
closed points of regular loci of preimages of $C$.
\end{itemize}
\end{defi}

\begin{lem}\label{filtration.lem}
Let $X=(X,C)$ be in $\cC$ and $D\in \ZXCD$ such that $|D|=C$. 
For $g:(\tX,\tC)\to (X,C)$ in $\chBX$, we have 
\[
\FXCC D\subset \FtXCC {g^*D} \subset \hFXCC D.
\]
We have 
\[
\hFXCC D = \indlim {g:\tX\to X} \; \FtXCC {g^*D},
\]
where $g: (\tX,\tC)\to (X,C)$ ranges over $\chBX$. 
\end{lem}
\begin{proof}
Take integral $Z\in \ZXCC$ and let $Z'\in \ZXCCt$ be its proper transform.
Then $Z'$ is finite over $Z$ and we have
\[
\cO^h_{Z,C\cap Z}\subset \cO^h_{Z',\tC\cap Z'}\subset \cO^h_{Z^N,C\cap Z}.
\]
The first assertion follows from these facts. The second assertion follows from 
the fact that for any integral $Z\in \ZXCC$, there is $g:\tX\to X$ in $\chBX$ such that
the proper transform of $Z$ in $\tX$ is regular (see \cite[Appendix Th.A.1]{SaSa}).
\end{proof}

\begin{defi}\label{filtration.def}
For $X=(X,C)$ and $D$ as in Lemma \ref{filtration.lem}, we put
\[
\tFXCC D = \indlim {g:\tX\to X} \; \FtXCC {g^*D},
\]
where $g: (\tX,\tC)\to (X,C)$ ranges over $\cB_X$. Lemma \ref{filtration.lem} implies
\[
\FXCC D\subset\tFXCC D\subset  \hFXCC D.
\]
For $g:(\tX,\tC)\to (X,C)$ in $\cB_X$, we have 
\begin{equation}\label{eq.tFXCC}
\tFtXCC {g^*D}=\tFXCC D.
\end{equation}
\end{defi}

\begin{rem}\label{filtration.rem}
For $Z\in \ZXCC$, we have an isomorphism
\[
\cO_{Z,C\cap Z}\otimes_{\cO_X}\cO_C \simeq 
\underset{x\in Z\cap C}{\prod}\;\cO^h_{Z,x}\otimes_{\cO_X}\cO_C.
\]
For $D,D'\in \Div(X)^+$ with $D'\geq D$ and $|D|=|D'|=C$, we have isomorphisms
\[
\frac{1+I_D\cO_{Z,x}}{1+I_{D'}\cO_{Z,x}} \simeq \frac{1+I_D\cO_{Z,x}^h }{1+I_{D'}\cO_{Z,x}^h },\quad
\frac{1+I_D\cO_{Z,Z\cap C}}{1+I_{D'}\cO_{Z,Z\cap C}}
\simeq
\underset{x\in Z\cap C}{\bigoplus}\;
\frac{1+I_D\cO_{Z,x}^h }{1+I_{D'}\cO_{Z,x}^h },
\]
and we have
\[\sym {1+I_D\cO_{Z,x}^h}{Z,x}\; \subset \sym {1+I_D\cO_{Z,x}} {Z,x} + \FXCC {D'}, \]
\[\underset{x\in Z\cap C}{\sum}\;
\sym {1+I_D\cO_{Z,x}^h}{Z,x}\; \subset \sym {1+I_D\cO_{Z,C\cap Z}} {Z} + \FXCC {D'}. \]
\end{rem}
\medbreak

Let $X=(X,C)$ be in $\cC$.
We introduce a tool to produce relations in $\WU$ by using symbols in the Milnor $K$-group 
$\KM 2 {k(X)}$ of the function field $k(X)$ of $X$.

\begin{lem}\label{classgroup.lem1}
Let $a,b\in k(X)^\times$ and assume we can write as divisors
\begin{equation*}\label{classgroupcondition}
\div_X(a)=Z_a^+ - Z_a^-+W_a ,\quad \div_X(b)=Z_b^+ - Z_b^-+W_b,
\end{equation*}
where $Z_a^+, Z_a^-,Z_b^+, Z_b^-\;\in \Div(X,C)^+$ such that any common component of 
$Z_a^+\cup Z_a^-$ and $Z_b^+\cup Z_b^-$ does not intersect with $C$, and $W_a,W_b$ have support in $C$. Then 
\[
 \partial \{a,b\} := 
\sym{a}{Z_b^+}-\sym{a}{Z_b^-}-\sym{b}{Z_a^+}+\sym{b}{Z_a^-},
\]
vanishes in $\WU$.
Here the first term $\sym{a}{Z_b^+}$ denotes $\sym{a_{|Z_b^+}}{Z_b^+}$ where 
$a_{|Z_b^+}$ is the image of $a$ in $k(Z_b^+)_\infty^\times$, which is well-defined.
The other terms are defined similarly.
\end{lem}

\begin{proof}
We have a decomposition in $ \Div(X,C)^+$:
\[ Z_a^+ = \Phi_a^+ + E_a^+,\;\;  Z_a^- = \Phi_a^- + E_a^-,\;\;
Z_b^+ = \Phi_b^+ + E_b^+,\;\; Z_b^- = \Phi_b^- + E_b^-,\]
where $ E_a^{\pm}$ and $E_b^{\pm}$ do not intersect $C$, and every irreducible component of 
$\Phi_a^{\pm}$ and $\Phi_b^{\pm}$ intersects $C$. The assumption implies that $a$ (resp. $b$) is invertible 
at any generic point of $\Phi_b^{\pm}$ (resp. $\Phi_a^{\pm}$). 
Since $k(Z_a^{\pm})_\infty^\times= k(\Phi_a^{\pm})_\infty^\times$ and 
$k(Z_b^{\pm})_\infty^\times= k(\Phi_b^{\pm})_\infty^\times$, this implies that 
\[a_{|Z_b^{\pm}}\in k(Z_b^{\pm})_\infty^\times \qaq
b_{|Z_a^{\pm}}\in k(Z_a^{\pm})_\infty^\times\]
are well-defined. 

For $Z\in \Div(X,C)^+$ and $a\in k(X)^\times$ which is invertible at any generic point of $Z$, 
we write
\[a_{|Z}= \big((a_{|Z_i})^{e_i}\big)_{1\leq i\leq r}\;\in k(Z)^\times=\underset{1\leq i\leq r}{\prod} k(Z_i)^\times,\]
where $Z_1,\dots,Z_r$ are the irreducible components of $Z$ and $e_i$ is the multiplicity of $Z_i$ in $Z$
(see Definition \ref{curve.def}). Put 
\[
\begin{aligned}
& \alpha^+=a_{|\Phi_b^+} \in k(\Phi_b^+)^\times, \;\; \alpha^-=a_{|\Phi_b^-}  \in k(\Phi_b^-)^\times,\\
& \beta^+=b_{|\Phi_a^+}  \in k(\Phi_a^+)^\times,\;\; \beta^-=b_{|\Phi_a^-} \in k(\Phi_a^-)^\times.
\end{aligned}\]
For an integral curve $E\subset X$ such that $E\cap C=\varnothing$, let
$\gamma_E\in k(E)^\times$ be the image $\{a,b\}$ of the tame symbol
\[\partial_E: K_2(k(X)) \to k(E)^\times.\]
Obviously, the elements
\[
\delta( \alpha^+) ,  \delta( \alpha^-), \delta( \beta^+), \delta( \beta^-), \delta(\gamma_E) 
\]
map to zero in $ \WU$, with $\delta$ as in \eqref{WCG}. In order to finish the proof of the lemma it suffices to
show the equality:
\[
 \partial\{a,b\} = \delta( \alpha^+) - \delta( \alpha^-) -  \delta( \beta^+) + \delta(
 \beta^-) + \sum_{E} \delta(\gamma_E)  \in  \underset{Z\subset X}{\bigoplus}\; k(Z)_\infty^\times\;\oplus\; Z_0(U),
\]
where the last sum ranges over the irreducible components $E$ of $E_a^{\pm}\cup E_b^{\pm}$. 
This follows from the fact that the contributions of the right hand side at any closed point $x\in U$ cancel out
as a consequence of the Gersten complex for $K$-theory
\[
K_2(k(X))  \rmapo{\partial_y} \bigoplus_{y \in \Spec (\cO_{U,x})^{(1)}} K_1(y)  \to K_0(x)=\Z.
\]
\end{proof}

\begin{lem}\label{classgroup.lem2}
Fix $D\in \Div(X)^+$ with $|D|=C$, and $F\in \Div(X,C)^+$ and $a\in H^0(X,\cO_X(-D+F))$ with $a\not=0$.
Define $Z\in \Div(X)^+$ by
\[ Z= \div_X(1+a) + F.\]
\begin{itemize}
\item[(1)]
We have $Z\cap C=F\cap C$.
\item[(2)]
Let $b\in k(X)^\times$ and assume we can write as divisors
\[\div_X(b) = F_1- F_2+W,\]
where $W$ has support in $C$, and $F_1, F_2\in \Div(X,C)^+$ such that  
$Z\cup F$ and $F_1\cup F_2$ have no common irreducible component which passes through $F\cap C$. 
Then we have 
\[\sym{1+a}{F_1}-\sym{1+a}{F_2}-\sym{b}{Z}+\sym{b}{F}=0 \;\in \WU.\]
\item[(3)]
In (2) assume further $F_i=Z_i+ G_i$ for $i=1,2$, where $Z_i,G_i\in \Div(X,C)^+$ such that 
$G_i\cap F\cap C=\varnothing$. Then we have 
\[\sym{1+a}{Z_1}-\sym{1+a}{Z_2}-\sym{b}{Z}+\sym{b}{F}\;\in \FXCC {D}.\]
\item[(4)]
Let $b\in k(X)^\times$ be such that $b= u \pi^n$, where $u\in \cO_{X,F\cap C}^\times$ and 
$\pi\in \cO_{X,F\cap C}$ is a local equation of $C$ around $F\cap C$ and $n\in \bZ$. Then we have
\[\sym{b}{Z}-\sym{b}{F}\;\in \FXCC {D}.\]
\end{itemize}
\end{lem}
\begin{proof}
For $x\in C \backslash F$, $1+a$ is regular in a neighborhood of $x$ and its restriction to $C$ is $1$.
Hence $x\not\in Z$.
For $x\in F\cap C$, we can write
$\displaystyle{ a =u \pi/f}$, where $u\in \cO_{X,x}$ and $f$ (resp. $\pi$) is a local equation of $F$ (resp. $D$)
at $x$. Then $f(1+a)= f+u\pi$ is a local equation of $Z$ which vanishes at $x$ since $f$ and $\pi$ do. 
So $x\in Z$, which proves (1). 

By (1), the assumption of (2) implies that any common component of $Z\cup F$ and $F_1\cup F_2$
does not intersect $C$. Hence (2) follows from Lemma \ref{classgroup.lem1}.

As for (3) note $a|_{G_i}\in \cO_{G_i,G_i\cap C}(-D)$ since $a\in H^0(X,\cO_X(-D+F))$ and
$G_i\cap F\cap C=\varnothing$. This implies $\sym{1+a}{G_i}\in \FXCC {D}$ and (3) follows from (2).

The assumption of (4) implies $\div_X(b) = G_1- G_2+W$,
where $W$ has support in $C$, and $G_1,G_2\in \Div(X,C)^+$ such that $(G_1\cup G_2) \cap F\cap C=\varnothing$ .
Thus (4) follows from (3).
\end{proof}


\section{Review of ramification theory}\label{ReviewRT}

\subsection{Local ramification theory}\label{RTlocal}
\bigskip

In this subsection $K$ denotes a henselian discrete valuation field of $\ch(K)=p>0$ with 
 ring  of integers $\cO_K$ and residue field $E$.
Let $\pi$ be a prime element of $\cO_K$ and $\fmK=(\pi)\subset \cO_K$ be the maximal ideal.
By the Artin--Schreier--Witt theory, we have a natural isomorphism for $s\in \bZ_{\geq 1}$,
\begin{equation}\label{ASW.eq}
\delta_s: W_s(K)/(1-F)W_s(K) \isom H^1(K,\psz),
\end{equation}
where $W_s(K)$ is the ring of Witt vectors of length $s$ and $F$ is the Frobenius.
We have the Brylinski--Kato filtration
\[
\fillog m W_s(K) = \{(a_{s-1},\dots,a_1,a_0)\in W_s(K)\;|\; p^i v_K(a_i)\geq -m\},
\]
where $v_K$ is the normalized valuation of $K$. In this paper we use its non-log version introduced by 
Matsuda \cite{Ma}:
\[
\fil m W_s(K) = \fillog {m-1} W_s(K) + V^{s-s'} \fillog {m} W_{s'}(K),
\]
where $s'=\min\{s,\ord_p(m)\}$ and $V: W_{s-1}(K) \to W_s(K)$ is the Verschiebung. 
We define ramification filtrations on $H^1(K):=H^1(K,\qz)$ as
\[ 
\FHlog K m = H^1(K)\{p'\} \oplus \underset{s\geq1}{\cup} \delta_s(\fillog m W_s(K)) \quad
\quad (m\ge 0 )  ,
\]
\[ 
\FH K m = H^1(K)\{p'\} \oplus \underset{s\geq1}{\cup} \delta_s(\fil m W_s(K))\quad
\quad (m\ge 1 )    ,
\]
where $H^1(K)\{p'\}$ is the prime-to-$p$ part of $H^1(K)$.
We note that $\FH K m$ is shifted by one from Matsuda's filtration
\cite[Def.3.1.1]{Ma}.
We also let $\FH {K} 0$ be the subgroup of all unramified Galois characters.

\begin{defi}
For $\chi \in H^1(K)$ we denote the minimal $m$ with $\chi \in \FH{K} m$ by $\art_K(\chi)$
and call it the Artin conductor of $\chi$.
\end{defi}

In case the field $E$ is perfect this definition coincides with the classical definition, see
\cite[Prop.\ 6.8]{Ka}.

We have the following fact (cf.\ \cite{Ka} and \cite{Ma}).

\begin{lem}\label{CFT.lem0}
\begin{itemize}
\item[(1)]
$\FH {K} 1$ is the subgroup of tamely ramified characters.
\item[(2)]
$\FH {K} {m}\subset \FHlog {K} m \subset \FH {K} {m+1}$.
\item[(3)]
$\FH {K} m=\FHlog {K} {m-1}$ if $(m,p)=1$.
\end{itemize}
\end{lem}

The structure of graded quotients: 
\[
\gr m H^1(K)= \fil m H^1(K)/\fil {m-1} H^1(K) \quad(m>1)
\]
are described as follows. 
Let $\Omega_K^1$ be the absolute K\"ahler differential module and put
\[
\fil m \Omega_K^1 =\fm_K^{-m}\otimes_{\cO_K} \Omega_{\cO_K}^1.
\]
We have an isomorphism
\begin{equation}\label{grOmega}
\gr m \Omega_K^1 = \fil m \Omega_K^1/\fil {m-1} \Omega_K^1 \simeq 
\fmK^{-m} \Omega_{\cO_K}^1\otimes_{\cO_K} E.
\end{equation}
We have the maps
\begin{equation}\label{Fsd}
F^sd: W_s(K) \to \Omega_K^1\;;\; (a_{s-1},\dots,a_1,a_0) \to 
\underset{i=0}{\overset{s-1}{\sum}} \; a_i^{p^i-1} da_i.
\end{equation}
and one can check
$F^sd(\fil m W_s(K))\subset \fil m\Omega_K^1$. 

\begin{theo}\label{thm.Masuda}(\cite[3.2.3]{Ma})
Assume $p\not=2$ and $m>1$.
\begin{itemize}
\item[(1)]
The maps $F^sd$ induces an injective map
\begin{equation}\label{rswK}
\rsw_K: \gr m  H^1(K) \hookrightarrow \gr m \Omega_K^1.
\end{equation}
\item[(2)]
If the residue field of $K$ is perfect the map \eqref{rswK} is surjective.
\end{itemize}
\end{theo}

The map $\rsw_K$ is called the refined Artin conductor for $K$.


\begin{defi}\label{filMilnorK}
Let $K$ be as before and $K_N^M(K)$ be the $N$-th Milnor $K$-group of $K$. 
For an integer $m\geq 1$, we define $V^m\KM N K\subset \KM N K$ as a subgroup generated by
the elements of the form
\[
\{1+ a,b_1,\dots,b_{N-1}\}\qaq \{1+ a\pi,b_1,\dots,b_{N-2},\pi\},
\]
where $a\in \fmK^m$ and $b_1,\dots,b_N\in \cO_K^\times$.
\end{defi}

The following lemma is proved by a similar argument as the proof of \cite[Lem.(4.2)]{BK}.

\begin{lem}\label{lem.grKM}
Assume $\ch(E)\not=2$.
There is a canonical surjective map
\[
\rho^m_K: \fmK^{m-1} \Omega^{N-1}_{\cO_K}\otimes_{\cO_K} E \to V^{m-1}\KM N K/V^{m}\KM N K,
\]
such that
\[
\rho^m_K(adb_1\wedge \cdots \wedge db_{N-1})=  \{1+ ab_1\cdots b_{N-1},b_1,\dots,b_{N-1}\}.
 \]
where $a\in \fmK^{m-1}$ and $b_1,\dots,b_N\in \cO_K$.
\end{lem}

\def\pairI{\langle\;,\, \rangle_{\Omega}}
\def\pairII{\langle\;,\, \rangle_K}
\def\vGam{\varGamma}
\medbreak

Let $K$ be an $N$-dimensional local field, namely there is a sequence of fields
$k_0,\dots,k_N$ such that $k_0$ is finite, $k_N=K$, and for $1\leq i\leq N$, $k_i$ is a henselian discrete valuation field with residue field $k_{i-1}$. 
In \cite{Ka1} Kato defined the so-called reciprocity map for $K$:
\begin{equation}\label{recK}
\Psi_K: H^1(K) \to \Hom(\KM N K,\qz).
\end{equation}

\begin{lem}\label{lemRes}
Assume $\ch(K)=p>0$ with $p\not=2$.
\begin{itemize}
\item[$(i)$]
For $m\in \bZ_{\geq 1}$ and $\chi\in H^1(K)$, we have an equivalence of conditions:
\[
 \chi\in \FH K m \; \Longleftrightarrow\; \Psi_K(\chi)(V^{m}\KM N K)=0.
\] 
\item[$(ii)$]
The following diagram is commutative
\begin{equation}\label{commutativitypairing}
\xymatrix{ 
\FH K m \; \ar[r]^{\hskip -60pt\Psi_K } \ar[d]_{-\rsw_K}  &\Hom(\KM N K/V^{m}\KM N K,\qz) 
\ar[d]^{(\rho^m_K)^\vee} \\
\fmK^{-m} \Omega_{\cO_K}^1\otimes_{\cO_K} E \ar[r]^{\hskip -30pt\sigma}
 &\Hom(\fmK^{m-1} \Omega_{\cO_K}^{N-1}\otimes_{\cO_K} E,\qz) \\
}\end{equation}
where the right vertical map is induced by $\rho^m_K$ and $\sigma$ is induced by the pairing
\begin{multline*}
\pairI : 
\fmK^{-m} \Omega_{\cO_K}^1\otimes_{\cO_K} E \times \fmK^{m-1}\Omega_{\cO_K}^{N-1}\otimes_{\cO_K} E \to
\fmK^{-1}\Omega_{\cO_K}^{N}\otimes_{\cO_K} E \to \bF_p \simeq\pz, 
\end{multline*}
where the last map is induced by $Res^{\Omega}_{K/\bF_p}: \Omega_K^{N} \to \bF_p$,
which is the composite of the residue map $Res^{\Omega}_{K/k_0}: \Omega_K^{N} \to k_0$ from \cite[\S2 Prop.3]{Ka1}
and the trace map $k_0\to \bF_p$.
\end{itemize}
\end{lem}
\medbreak

A variant of $(i)$ and $(ii)$ for $\FHlog K m$ is stated in \cite[\S 3.5]{Ka3}. 
We will sketch a proof of the lemma in the appendix \S\ref{appendix}.

\subsection{Global ramification theory}\label{RTglobal}

Let $X$ be a normal variety over a perfect field $k$. Let $U\subset X$ be an open
subscheme which is smooth over $k$ and whose reduced complement $C\subset X$ is
the support of an effective Cartier divisor. Our
aim in this section is to introduce the abelian fundamental group $\pi_1^\ab(X,D)$
classifying abelian
\'etale coverings of $U$ with ramification bounded by $D$. Here $D\in \Div(X)^+$ is an effective
divisor with support in $C$.

Let $I$ be the set of generic points of $C$ and $\Clam=\overline{\{\lam\}}$ for $\lam\in I$.
For $\lam\in I$ let $\Klam$ be the henselization of $K=k(X)$ at $\lam$.
Note that $\Klam$ is a henselian discrete valuation field with residue field $k(\Clam)$. 
We write $H^1(U)$ for the \'etale cohomology group $H^1(U,\Q/\Z)$.

\begin{prop}\label{filU.lem}
\mbox{}
\begin{itemize}
\item[(1)]
Assume $C$ is regular at a closed point $x$ and $x\in \Clam$ for $\lam\in I$. 
Let $F\in \ZXC$ be such that $F\Cap C$ at $x$ and let $k(F)_x$ be the henselization of $k(F)$ at $x$.
Take $\chi\in H^1(U)$ and let $\chi|_{\Klam}\in H^1(\Klam)$ and $\chi|_{F,x} \in H^1(k(F)_x)$ be its restrictions.
For an integer $m\ge 0$, we have an implication:
\begin{equation*}\label{filU.lem.eq}
\chi|_{\Klam}\in \FH {\Klam} m \;\Longrightarrow\; \chi|_{F,x}\in \FH {k(F)_x} m.
\end{equation*}
\item[(2)]
Assume $C=C_\lambda$ is regular and irreducible. Let $T_X$ be the tangent sheaf of $X$.
There is  a dense open subset $V_\chi\subset \mathbb P(T_X |_{C_\lambda})$ (depending on $\chi$)
such that for any integral  $F\in \ZXC$ and for any $x$ with $F\Cap C$ at $x$ the implication
$$
T_F(x) \in V_\chi  \;\Longrightarrow\;    \art_{\Klam}(\chi|_{\Klam})  =   \art_{k(F)_x}(\chi|_{F,x}) 
$$
holds.
\item[(3)]
Assume $C$ is a simple normal crossing divisor in a neighborhood of a closed point $x\in C$.
Let $g:X'=Bl_x(X) \to X$ be the blowup at $x$ and $E\subset X'$ be the exceptional divisor and $K_E$ be the henselization of $K$ at its generic point. 
For a Cartier divisor $D$ supported on $C$ we put
\[
m_E=\underset{\lam\in I_x}{\sum}\;\mlam (D),
\]
where $I_x$ be the set of irreducible components of $C$ containing $x$ and $m_\lambda(D)$ is the multiplicity of $D$ at $\lambda$. 
Then, for 
\[
\chi\in \Ker \big(H^1(U) \to \underset{\lam\in I_x}{\bigoplus}\; H^1(\Klam)/\FH \Klam {m_\lambda(D)}\big),
\] 
we have $\chi|_{K_E}\in \FH {K_E}{m_E}$.
\end{itemize}
\end{prop}
\begin{proof}
(1) and (2) follow from \cite[(7.2.1)]{Ma}. (3) is proved by the same argument as \cite[Th.(8.1)]{Ka} using \cite[Cor.4.2.2]{Ma} instead of 
\cite[Th.(7.1)]{Ka}.
\end{proof}

\medskip

\begin{coro} \label{globalrt.defram}
Assume $C$ is a simple normal crossing divisor.
For $\chi\in H^1(U)$ and a Cartier divisor $D$ supported on $C$, the following are equivalent
\begin{itemize}
\item[(1)] for all generic points $\lambda$ of  $C$ we have $\chi|_{K_\lambda} \in \fil{m_\lambda(D)} H^1(K_\lambda) $,
\item[(2)] for all integral  $Z \in  \ZXC$ 
and $x\in \Zinf$, we have (see Definition \ref{classgroup.def})
\[\chi|_{Z,x} \in \fil{m_x(\psi_Z^* D)} H^1( k(Z)_x )\;.\]
Here $\chi|_{Z,x} \in H^1( k(Z)_x )$ is the restriction of $\chi$ and $m_x$ is the multiplicity at $x$.  
\end{itemize}
\end{coro}

\begin{proof}
The implication (1)$\Rightarrow$(2) follows from Proposition~\ref{filU.lem}(1) and (3) by observing that 
for integral  $Z \in  \ZXC$ there is a chain of blowups in closed points such that the strict transform of $Z$ becomes smooth and such that its intersection with the total transform of $C$ is transversal.
The implication (2)$\Rightarrow$(1) follows from Proposition~\ref{filU.lem}(2).
\end{proof}

For general $X$ and $C$, not necessarily of normal crossing, we make the following definition.

\begin{defi} \label{gloram.def}
 For $D\in \Div(X)^+$ with support in $C$ 
we define $\fil D H^1(U)$ to be the subgroup of $\chi\in H^1(U)$ satisfying property (2)
in Corollary~\ref{globalrt.defram}.
Define
\begin{equation}\label{piabXD}
\piabXD=\Hom(\FH U D,\qz),
\end{equation}
endowed with the usual pro-finite topology of the dual.
\end{defi}

One should think of $\piabXD$ as the quotient of $\pi_1^{ab}(U)$ classifying abelian \'etale coverings of $U$
with ramification bounded by $D$.

\begin{prop}\label{grt.exhaust}
The filtration $\fil D H^1(U)$ is exhaustive, i.e.\ 
\[
\bigcup_D \fil D H^1(U) = H^1(U),
\]
where $D \in \Div(X)^+$ runs through all divisors with support in $C$. 
\end{prop}

A proof can be found in \cite[Sec.\ 3.3]{EK}.


\section{Existence Theorem}\label{CFT}

\bigskip

In this section $k$ is assume to be finite.
Let $U$ be a smooth variety over $k$.
Choose a compactification $U\subset X$ with $X$ normal and proper over $k$ such that the
reduced subscheme $C=X\setminus U$ of $X$ is the support of an effective Cartier divisor on $X$. Put $K=k(X)$. 
In \S \ref{classgroup} we defined the relative Chow group of zero cycles $\CXD$, where $D\in
\Div(X)^+$ is a Cartier divisor with support in $C$. We endow this relative Chow group with the
discrete topology. We endow the group
\[
\C(U)=\varprojlim_D \CXD
\]
with the inverse limit topology. Here $D$ runs through all effective Cartier divisors on
$X$ with support in $C$.

\begin{lem}
The topological group $\C(U)$ does not depend on the choice of the compactification
$X$ of $U$.
\end{lem}
\begin{proof}
Let us write $\C(U\subset X)$ for the class group relative to the compactification $X$ in
the following.
Assume $U\subset X_1$ and $U\subset X_2$ are two compactifications. 
Considering the normalization of the Zariski closure of the diagonal $U\to X_1 \times_k X_2 $, we may assume that
there is a morphism $f:X_2\to X_1$ which is the identity on $U$. It is then sufficient to show that 
the pushforward map \eqref{eq.Wnorm}
\begin{equation}\label{cft.compeq} 
 f_* \;:\; \C(U\subset X_2)\to \C(U\subset X_1)
\end{equation}
is an isomorphism. For an effective Cartier divisor $D$ on $X_1$ with support in $X_1 \setminus U$, one easily see that
$f_* \;:\;  \C(X_2,f^* D ) \rightarrow  \C (X_1 , D) $ is an isomorphism (see Definition \ref{filtration.def0}(2)).
As the divisors $f^* D$ are cofinal in the system of all divisors on $X_2$ with support
in $X_2\setminus U$, the isomorphy of \eqref{cft.compeq} follows. 
\end{proof}
\bigskip

In fact it is also clear from the proof that $U \mapsto \C(U)$ is a covariant functor from the category
of smooth
varieties over $k$ to the category of topological abelian groups.

\begin{prop}\label{cft.recmap}
There is a unique continuous reciprocity homomorphism $\rho_U$ making the diagram 
\[
\xymatrix{
Z_0(U)  \ar[r]  \ar[rd] & \C(U)  \ar[d]^{\rho_U}  \\
 &  \pi_1^\ab(U)
}
\]
commutative. Here the diagonal arrow is induced by the Frobenius homomorphisms $\Frob_x :
\Z \to \pi_1^\ab(U)$ for closed points $x\in U$. Moreover, $\rho_U$ induces a
homomorphism
\[
\rho_{X,D} :\C(X,D) \to \pi_1^\ab(X,D).
\]
\end{prop}

Recall that the pro-finite fundamental group  $\pi_1^\ab(X,D)$ classifies abelian
\'etale coverings of $U$ with ramification over $C$ bounded by the divisor $D$, see Definition~\ref{gloram.def}. 
In what follows, for a topological abelian group $M$, we write 
\[M^\vee=\Hom_\cont(M,\qz),\] 
where we endow $\qz$ with the discrete topology. 
\medbreak

\begin{proof}[Proof of Proposition \ref{cft.recmap}] 
In \cite{Wi}, \cite{KeSc} a continuous reciprocity homomorphism $r_U:\W(U)\to \pi_1^\ab(U)$ is constructed.
In order to accomplish the proof of the proposition we need some ramification theory. It is
sufficient to show that for any character 
\[
\chi \in (\pi_1^\ab(U))^\vee \cong H^1(U)
\]
there is a divisor $D\in \Div(X)^+$ with support in $C$ such that $r_U^*\chi \in
\Hom(W(U)\to \Q/\Z) $ factors through $\C(X,D)$. In view of Definition \ref{gloram.def},
ramification properties of classical local class field theory (see \cite[Sec.\ XV.2]{Se}) imply
that the map $r_U$ induces a map
\[
\Psi_{X,D} : \fil D H^1(U) \longrightarrow \C(X,D)^\vee .
\]
Finally, the proposition follows from Proposition \ref{grt.exhaust}. 
\end{proof}

Define topological groups $\C(U)^0$ and $\pi_1^\ab(U)^0$ as kernels in the commutative diagram
\begin{equation}
\xymatrix{
 0\ar[r] & \C(U)^0  \ar[r] \ar[d] & \C(U)    \ar[d]^{\rho_U}   \ar[r]^{f_*} &  \C(\Spec k)\ar[d]^{\rho_k} \\  
0\ar[r] & \pi_1^\ab(U)^0 \ar[r]  &  \pi_1^\ab(U) \ar[r]^{f_*} &  \piab {\Spec k}
}
\end{equation}
where $f:U \to \Spec k$ is the natural morphism.
Note that $ \C(\Spec k)=\bZ$ and $\rho_k$ maps $1\in\bZ$ to the Frobenius over $k$. Let
\[
 \rho^0_U\;:\; \C(U)^0  \to \pi_1^\ab(U)^0
\]
be the induced map.
\medbreak

Our main theorem says:

\begin{theo}[Existence Theorem] \label{CFT.mainthm}
Over a finite field $k$ with $\ch( k)\ne 2$, $\rho^0_U$ is an isomorphism of topological groups.
\end{theo}

\begin{coro}\label{CFT.coro} Assume $\ch(k)\ne 2$.
For an effective divisor $D\in \Div(X)^+$ with support in $C$, 
$\rho^0_U$ induces an isomorphism of finite groups
\[
\rho_{X,D}\;:\; \CXD^0 \xrightarrow{\sim} \piabXD^0.
\]
\end{coro}
\begin{proof}
In view of Definition \ref{gloram.def}, the corollary follows from the Theorem~\ref{CFT.mainthm} by using standard 
ramification properties in local class field theory as is explained in \cite[Sec.\ XV.2]{Se}.
\end{proof}

The proof of the Existence Theorem is begun in this section and completed in \S\ref{CFT2} assuming some technical lemmas that will be shown in later
sections.

\bigskip

Now we start the proof of the Existence Theorem \ref{CFT.mainthm}. Consider the
following property
\begin{itemize}
\item[] ${\bf I}_{U}$: $\rho_U$ induces a surjection
\begin{equation}\label{tameCFT.eq1}
\Psi_{U} : H^1(U,\Q / \Z ) \to \C(U)^\vee = \Hom_\cont ( \C(U) , \Q / \Z ).
\end{equation}
\end{itemize}

Note that we already know that the map \eqref{tameCFT.eq1} is injective by Chebotarev
density theorem \cite{Se2}.

\medskip

We now give an overview of the steps in the proof of the Existence
Theorem~\ref{CFT.mainthm}.

\begin{itemize}
\item In Lemma \ref{cft.proper} we show that property  ${\bf I}_{U}$
  implies the Existence Theorem for the triple $(X,C,U)$.
\item In Lemma \ref{cft.trick} combined with de Jong's alteration theorem we show how to
  reduce the proof of  ${\bf I}_{U}$ to the situation where $C$ is a simple normal
  crossing divisor.
\item We use a Lefschetz hyperplane theorem \cite{KeSLef} ($C$ simple normal crossing) which allows us to reduce the
  proof of ${\bf I}_{U}$ to the case $\dim(X)=2$.
\item In  \S \ref{keytheorem}  we study  (for $\dim(X)=2$) ramification filtrations on the Galois side and the
  class group side and compare graded pieces to complete the proof of  ${\bf I}_{U}$ in \S
  \ref{CFT2}. The
  understanding of the filtration on the class group side is our key new ingredient. 
\end{itemize}

\begin{lem}\label{cft.proper}
Property ${\bf I}_U$ implies that the map $\rho_U$ in
Theorem~\ref{CFT.mainthm} is an isomorphism of topological groups.
\end{lem}

\begin{proof}
As $\pi_1^\ab(X,D)^0$ is finite by \cite{KeSLef} for any effective divisor $D$ with
$|D|\subset C$, it is enough to show that $\rho_U$
induces an isomorphism
$\C(X,D)^0 \to \pi_1^\ab(X,D)^0 $ of abstract groups. 
It is sufficient to show that dually
 $$\Psi_{X,D}: \fil D H^1(U)\to  \C(X,D)^\vee $$
is an isomorphism. The latter is a direct consequence of ${\bf I}_{U}$ and classical
ramification theory for local fields.
\end{proof}

\bigskip

We next introduce certain reduction techniques for property ${\bf I}_U$, based on methods of Wiesend. 
In the following lemma we denote by $f:X'\to X$ an alteration with $X'$ normal. By $U' \subset
f^{-1}(U)$ we denote an open smooth subscheme of $X'$, which is the complement of the
support of an effective
Cartier divisor.
We use the notation $$f:U'\to U, \quad C'=X'\setminus U'.$$

\begin{lem}[Wiesend trick]\label{cft.trick} \mbox{}
\begin{itemize}
\item[(i)] For $f:X'\to X$ and $U'\subset f^{-1}(U)$ as above, the implication $  {\bf I}_{U'}
  \Rightarrow  {\bf I}_U $ holds.
\item[(ii)] Assume that for any character $\chi \in \C(U)^\vee$  we can find $f:X'\to X$ and $U' \subset X'$ as above such that $f^* (\chi) =0$. Then property
  ${\bf I}_U$  holds.
\end{itemize}
\end{lem}

\begin{proof}
We first explain the proof of (i). 
Consider the commutative diagram of abstract groups
\begin{equation}\label{cft.cartsq}
\xymatrix{
H^1(U)  \ar[r]^{\Psi_U} \ar[d]_{f^*} & \C(U)^\vee  \ar[d]^{f^*} \\
H^1(U')  \ar[r]_{\Psi_{U'}}  & \C(U')^\vee 
}
\end{equation}
It is sufficient to see that a character $\chi \in
\C(U)^\vee$ such that $f^*(\chi) $ is  of the form $\Psi_{U'}(\sigma)$ with $\sigma \in H^1(U')$ is in  the image of
$\Psi_U$. We can choose another alteration $f':X'' \to X'$ with the property that
$f'^{-1}(\sigma)=0$. This means that without loss of generality we can assume that
$f^*(\chi) = 0 \in \C(U')^\vee$.

Shrinking $U'$ we can also assume that $U'\to f(U') \subset X$ is the composition of a
finite surjective radicial map $U' \to U_\et$ and a finite \'etale map $U_\et \to f(U')$.
Then the maps 
\[
H^1(U_\et) \to H^1(U') \quad \text{ and } \quad \C(U_\et)^\vee \to \C(U')^\vee
\]
are isomorphisms: For the first map this is clear. As for the second, in view of the definition of
the norm map for the Wiesend class groups (cf. \cite[Lem.7.3]{KeSc} and its proof), 
it follows from the facts that the pushforward map $Z_0(U') \to Z_0(U_\et)$ is an isomorphism,  and that 
for a finite surjective radicial covering $Z'\to Z$ of integral normal curves, 
the norm map $k(Z')^\times \to k(Z)^\times$ is an isomorphism as well as 
the norm map $k(Z')_y^\times \to k(Z)_x^\times$ for the henselizations at closed points $x\in Z$
and $y\in Z'$ lying over $x$.
Therefore we can without loss of generality assume that $X'\to X$ is generically \'etale.
In this situation we finally conclude that $\chi$ is in the image of $\Psi_U$ by using
Wiesend's method, see
\cite[Prop.\ 3.7]{KeSc}.

The proof of (ii) is a variant of the proof of (i).
\end{proof}

\begin{lem}\label{cft.lefschetz}
Assume that property ${\bf I}_U$ holds for all smooth varieties $U$ with $\dim(U)=2$. Then
it holds for arbitrary smooth $U$.
\end{lem}

\begin{proof}
By Lemma~\ref{cft.proper} we obtain Corollary~\ref{CFT.coro} for two-dimensional $X$. In
the general case we reduce the proof of property ${\bf I}_U$ to the case $C$ is simple
normal crossing and $X$ is projective by Lemma~\ref{cft.trick} and de Jong's alteration
theorem~\cite{dJ}. This means that for such $(X,C,U)$ we have to show that the map
\[
\C(X,D)^0 \to \pi_1^\ab(X,D)^0
\]
is an isomorphism for all $D$.

Let $\cL$ be an ample line bundle on $X$.
Let $i:Y \hookrightarrow X$ be a smooth hypersurface section, which is the zero locus of some section of
$\cL^{\otimes n}$ ($n\gg 0$), such that $Y\times_X C$ is a reduced simple normal crossing divisor on $Y$
and let $E=Y\times_X D$. Consider the commutative diagram
\[
\xymatrix{
 & \C(Y,E)^0  \ar[r]^{\sim}_{\rho_{Y,E}}   \ar[d]_{i_*} &   \pi_1^\ab(Y,E)^0  \ar[d]^{\wr}\\
 Z_0(U)^0   \ar@{->>}[r] & \C(X,D)^0 \ar@{->>}[r]_{\rho_{X,D}}  & \pi_1^\ab (X,D)^0 
}
\]
The map $\rho_{Y,E}$ is an isomorphism by induction on dimension. The right vertical map
is an isomorphism for $n$ sufficiently large \cite{KeSLef}. The map $\rho_{X,D}$ is
surjective because of Chebotarev density \cite{Se2} and the finiteness of $ \pi_1^\ab
(X,D)^0$, see  \cite{KeSLef}. So we have to show injectivity of  $\rho_{X,D}$.

 For an $\alpha\in \C(X,D)^0$ with $\rho_{X,D} ( \alpha) =0$ use a Bertini argument to
 choose  $Y$ as above 
 which contains the support of a lift of $\alpha$ to $Z_0(U)$. Then $\alpha$ is in the
 image of $i_*$. A diagram chase shows that $\alpha = 0$.

\end{proof}

\section{Cycle conductor}\label{keytheorem}
\bigskip

Let the notation be as in \S\ref{classgroup}. 
Let $X=(X,C)$ be in $\cC$ and recall $\dim(X)=2$ (cf. Definition \ref{def.cC}).
Let $\{\Clam\}_{\lam\in I}$ be the set of prime components of $C$.
Fix a Cartier divisor 
\begin{equation}\label{eqD1}
\displaystyle{D=\underset{\lambda\in I}{\sum} \mlam \Clam\qwith \mlam\geq 2}.
\end{equation}
For a Cartier divisor $F$ on $X$ and $Z\in \ZXCC$, we write 
\[
\cO_Z(F) = \cO_{Z}\otimes_{\cO_X}\cO_X(F).
\]
\medbreak
In this section we assume $k$ is finite and introduce the key homomorphism
called the {\em cycle conductor} for $(X,D)$:
\begin{equation}\label{Rsw2}
\rswMXD: \CXD^\vee \to H^0(C,\Omega^1_X(D+\Xi)\otimes_{\cO_X}\cO_C),
\end{equation}
and state its basic properties. Here $\Xi\in \Div(X,C)^+$ is some sufficiently big Cartier
divisor introduced below. First we note the canonical duality isomorphism
\begin{equation}\label{duality1}
H^0(C,\Omega^1_X(D+\Xi)\otimes_{\cO_X}\cO_C)\simeq H^1(C,\Omega^1_X(-D+C-\Xi)\otimes_{\cO_X}\cO_C)^\vee.
\end{equation}
Indeed, letting $\wC$ be the dualizing sheaf of $C$, we have
\[
\wC \simeq \cExt^1_{\cO_X}(\cO_C,\Omega^2_X) \simeq \Omega^2_X(C)\otimes_{\cO_X}\cO_C,
\]
where the second isomorphism follows from the long exact sequence for $\cExt$ induced by the exact sequence
$ 0\to \cO_X(-C) \to \cO_X \to \cO_C \to 0$.
Thus the Serre duality implies that the pairing
\[
\Omega^1_X(D+\Xi)\otimes_{\cO_X} \Omega^1_X(-D+C-\Xi)\otimes_{\cO_X}\cO_C
\to \Omega^2_X(C)\otimes_{\cO_X}\cO_C\simeq \wC
\]
induces a perfect pairing of abelian groups
\begin{equation}\label{duality0}
H^0(C,\Omega^1_X(D+\Xi)\otimes_{\cO_X}\cO_C) \times H^1(C,\Omega^1_X(-D+C-\Xi)\otimes_{\cO_X}\cO_C)
\to H^1(C,\wC) \rmapo {Tr_{C/\bF_p}} \pz,
\end{equation}
where $Tr_{C/\bF_p}$ is the composite
$H^1(C,\wC)\rmapo {Tr_{C/k}} k \rmapo {Tr_{k/\bF_p}} \pz$. This induces \eqref{duality1}.
Hence, by rewriting $D$ by $D+C$, the construction of \eqref{Rsw2} is reduced to that of its dual map: 
\begin{equation*}\label{rswMdual0}
H^1(C,\Omega^1_X(-D+C-\Xi)\otimes_{\cO_X}\cO_C) \to \CXD.
\end{equation*}
or equivalently that of a map (see Theorem \ref{keythm} below):
\begin{equation}\label{rswMdual}
\phiD: H^1(C,\Omega^1_X(-D-\Xi)\otimes_{\cO_X}\cO_C) \to \WU/\hFXCC {D+C}
\end{equation}
for a Cartier divisor 
\begin{equation}\label{eqD0}
\displaystyle{D=\underset{\lambda\in I}{\sum} \mlam \Clam\qwith \mlam\geq 1}.
\end{equation}
\bigskip


Let $x$ be a regular closed point of $C$.
Let $Z\in \ZXCC$ be such that $Z\Cap C$ at $x$.
Locally on a neighborhood of $x$, we have an exact sequence
\begin{equation}\label{OmegaXZes}
0\to \Omega^1_X(-D) \to \Omega^1_X(\log Z)(-D) \to \cO_Z(-D) \to 0.
\end{equation}
Tensored with $\cO_C$, it induces a boundary map
\begin{equation}\label{pZx}
\partial_{Z,x} : \cO_{Z,x}(-D)\otimes \k(x) \to H^1_x(C,\Omega^1_X(-D)\otimes_{\cO_X}\cO_C). 
\end{equation}
We will let $\partial_{Z,x}$ denote also the composite of $\partial_{Z,x}$ and the natural map
\[ H^1_x(C,\Omega^1_X(-D)\otimes_{\cO_X}\cO_C) \to  H^1(C,\Omega^1_X(-D)\otimes_{\cO_X}\cO_C).\]
On the other hand we define a map
\begin{equation}\label{muZx}
\mu_{Z,x} : \cO_{Z,x}(-D)\otimes \k(x) \to \WU/\FXCC {D+C}
\end{equation}
as the composite of the natural injection 
\[\cO_{Z,x}(-D)\otimes \k(x) \to k(Z)_x^\times/(1+I_{D+C}\cO^h_{Z,x}) \;;\; a \to 1+a\]
and the map 
\[k(Z)_x^\times/(1+I_{D+C}\cO^h_{Z,x}) \to \WU/\FXCC {D+C}\]
induced by $\sym {\hbox{ }} {Z,x}$ (cf. Definition~\ref{symbol.def} and the notation below it).
\medbreak

We now state the first key theorem for the proof of Theorem \ref{CFT.mainthm}.
Its proof will be given in \S\ref{keythmproof}. 
Recall that $k$ is assumed to be finite in this section.

\begin{theo}\label{keythm}
Assume $p=\ch(k)\not=2$. 
There exists $\Xi\in \ZXCC$ (cf. Definition~\ref{Cartierdiv.def}) and a natural map
\[
\phiD: H^1(C,\Omega^1_X(-D-\Xi)\otimes_{\cO_X}\cO_C)\to W(U)/\hFXCC {D+C}
\]
such that $C_{sing}\subset\Xi$ and $\Xi$ is independent of $D$ as in \eqref{eqD0}, and that the following conditions hold:
\begin{itemize}
\item[$(i)$]
For any closed point $x$ of $C\backslash \Xi$ and $Z\in \ZXCC$ such that $Z\Cap C$ at $x$, 
the following diagram is commutative:
\begin{equation}\label{keythm.CDZ}
\xymatrix{
\cO_{Z,x}(-D)\otimes \k(x) \ar[r]^{\hskip -30pt\partial_{Z,x}} \ar[rd]_{\mu_{Z,x}} &  
H^1(C,\Omega^1_X(-D-\Xi)\otimes_{\cO_X}\cO_C) \ar[d]^{\phiD} \\
& \hskip 30pt W(U)/\hFXCC {D+C}.\\
}
\end{equation}
\item[$(ii)$]\;\;
$\Image(\phiD) = \Image(\hFXCC {D})$.
\end{itemize}
\end{theo}

\begin{rem}\label{keythm.rem0}
The images of $\partial_{Z,x}$ for closed points $x$ of $C\backslash \Xi$ and $Z\in \ZXCC$ with 
$Z\Cap C$ at $x$ generate $H^1(C,\Omega^1_X(-D-\Xi)\otimes_{\cO_X}\cO_C)$.
Thus the condition $(i)$ uniquely characterizes $\phiD$.
\end{rem}

Theorem~\ref{keythm}$(ii)$ follows from $(i)$ and Lemma~\ref{movinglemfinite} below,
whose proof will be given in \S\ref{keylemproofV} 
(see Lemma~\ref{movinglem} and the leitfaden in \S\ref{keylemmas}).
It concerns moving elements of $\WU$ to symbols on curves transversal to $C$.
Take any dense open subset $V\subset X$ containing the generic points of $C$ and recall Definition \ref{filtration.def0}. 

\begin{lem}[moving]\label{movinglemfinite}
Assume $p\not=2$. 
For any integer $N>0$, we have 
\[
\hFXCC D  \subset \FXCCapV D \;+\;  \hFXCC {D+N\cdot C}.
\]
\end{lem}

\bigskip

For $D$ as \eqref{eqD1} let
\begin{equation*}
\rswMXD: \CXD^\vee \to H^0(C,\Omega^1_X(D+\Xi)\otimes_{\cO_X}\cO_C),
\end{equation*}
be the map induced by 
\[\phiDC : H^1(C,\Omega^1_X(-D+C-\Xi)\otimes_{\cO_X}\cO_C) \to \CXD=W(U)/\hFXCC {D} \]
using the duality \eqref{duality1}. By Theorem \ref{keythm}$(ii)$ we have an exact sequence
\begin{equation}\label{rswMXDexactseq}
0 \to \CXDC^\vee \to \CXD^\vee \rmapo{\rswMXD } H^0(C,\Omega^1_X(D+\Xi)\otimes_{\cO_X}\cO_C).
\end{equation}
The second key theorem concerns compatibility of the cycle conductor with ramification theory reviewed in \S\ref{ReviewRT}.
Its proof will be given in \S\ref{keythmproofII}.

\begin{theo}\label{keythm2}
For $\lam\in I$, let $\mlam$ be as in \eqref{eqD1} and 
\[
 \rsw_{\Klam} :\gr {\mlam} H^1(\Klam) \to  \fil {\mlam} \Omega^1_{\Klam}
\] 
be the refined Artin conductor for $\Klam$ in Theorem~\ref{thm.Masuda}. We note
\[
\Omega^1_X(D)\otimes_{\cO_X} k(\Clam) \simeq \gr {\mlam} \Omega^1_{\Klam}.
\]
Then the following diagram commutes
\begin{equation*}\label{Rsw.commute2}
\xymatrix{
\FH U D \ar[r]\ar[d]^{\Psi_{X,D}}& \FH {\Klam} {\mlam} \ar[r]^{\hskip -30pt -\rsw_{\Klam}} 
&\Omega^1_X(D)\otimes_{\cO_X} k(\Clam) \\
\CXD^\vee \ar[rr]^{\rswMXD} && H^0(C,\Omega^1_X(D+\Xi)\otimes_{\cO_X}\cO_C)\ar[u]^{\iota_\lam}\;. \\
}
\end{equation*}
\end{theo}


\section{Proof of Existence theorem}\label{CFT2}

\bigskip

Let the notation be as in \S\ref{CFT}. In this section we always assume $\ch (k)\ne 2$ and $\dim(X)=2$, 
but we do not generally assume that $C$ is simple normal crossing or that $X$ is smooth. 
 We prove property ${\bf I}_{U}$ in this case. Using Lemma~\ref{cft.lefschetz} we are
 going to reduce 
the proof of the Existence  Theorem~\ref{CFT.mainthm} to Theorems \ref{keythm} and
\ref{keythm2}. 

\medbreak
 
We start with the tame case, which is essentially due to Wiesend. 
Note that the second statement of the following proposition is motivated by the fact that an abelian \'etale covering $U'\to U$ is tame along $C$ if and only if its pullback to integral $F\in \ZXC$ such that $F\Cap C$ is 
tame along $C\cap F$, see Proposition~\ref{filU.lem}.

\begin{prop}\label{tameCFT} Assume $\dim(X)=2$ and that $X/k$ is smooth. 
The reciprocity map $\rho_U$ induces an isomorphism of finite groups
\[
\CXC^0 \xrightarrow{\sim} \piabXC^0.
\]
 Moreover the closure of the image of
$\FXCCap C$ in $\C(U)$
is equal to $\hFXCC C$.
\end{prop}

Here $C$ is considered as a reduced effective Cartier divisor.

\begin{proof}
It is shown in \cite[Thm.\ 8.3]{KeSc} (see also Remark \ref{rem.Suslinhom}) that $\rho_U$ induces an isomorphism
\[
\WU^0/\hWU \xrightarrow{\sim} \piabXC^0,
\]
where $\WU^0=\Ker\big(\WU \rmapo {f_*} W(\Spec k)\big)$ for the natural map $f:U\to \Spec k$ (note $W(\Spec k)=\bZ$).
The verbatim same argument shows the proposition.
\end{proof}

\begin{coro}\label{tameCFT.cor}
 Assume $\dim(X)=2$.
For any effective Cartier divisor $D$ on $X$ with $|D|= C$, $\C(X,D)^0$ is torsion of finite exponent and 
$\hWU/ \hFXCC {D}$ is of finite exponent of $p$-power.
\end{coro}
\begin{proof}
Using de Jong's alteration result \cite{dJ} and a norm trick we can assume without loss of
generality that $X/k$ is smooth.
Then by Proposition~\ref{tameCFT}, we have
\[
\WU^0/(\FXCCap {C} +   \hFXCC {D}) \cong \WU^0/  \hFXCC {C},
\]
and it is isomorphic to $\pi_1^\ab(X,C)^0$, which is finite by \cite[Th.2.7]{KeSc}.
On the other hand $p^m \FXCCap {C}\subset \FXCC {D}$ if $p^m C\geq D$ since
for $F\in \ZXCC$ such that $F\Cap C$ at $x\in F\cap C$, 
$(1+\cO_{F,x}(-C))^{p^m}\subset 1+\cO_{F,x}(-p^m C)$.
\end{proof}

\medbreak

Now we turn to the proof of property ${\bf I}_{U}$ in the wild case.                              
By Wiesend's trick (Lemma~\ref{cft.trick}) and a standard fibration technique \cite[XI, Prop. 3.3]{SGA4}, 
we can assume that there is a proper smooth curve $S$ over $k$ and morphisms
\[
f:X\to S\qaq \sigma: S \to X
\]
where $f$ is a proper surjective morphism with smooth generic fiber and $\sigma$ is a
section of $f$. Let $J$ be the set of generic points $\lam$ of $C$ which lie over the generic
point $\eta$ of $S$ and let $C_\lam = \overline{\{\lam\}}$ be the closure of $\lambda\in J$ in $X$. 
We can assume:
\begin{itemize}
\item
$f(\sigma(S)\cap C)$ does not contain $\eta$.
\item $C$ is a Cartier divisor on $X$. 
\item The induced morphism $\Clam\to S$ is an isomorphism for each $\lam \in J$.
\item $f|_U: U \to S$ is smooth.
\end{itemize}

Let us fix an algebraic closure $\overline{k(S)}$ of $k(S)$. Write $\bar \eta = \Spec\,
\overline{k(S)}$. Let us consider pairs $\Sigma = (T,\theta)$ where
\begin{itemize}
\item
 $T$ is the normalization of $S$ in a finite subextension of $k(S)$  in the field extension $k(S)
\subset \overline{k(S)}$ 
\item $\theta$ is an effective divisor on $T$.
\end{itemize}
Clearly for such $\Sigma=(T,\theta)$ there is a canonical map $T \to S$. We define a directed 
 partial ordering on the set of all $\Sigma$ by setting $$\Sigma_1=(T_1,\theta_1)\le
 \Sigma_2 = (T_2, \theta_2) , $$ 
if $k(T_1) \subset k(T_2)$, which means that  the map $T_2 \to S$ factors canonically through 
\[T_2 \xrightarrow{g_{\Sigma_2,\Sigma_1}} T_1  \to S\;\;\text{ and if }\;
g_{\Sigma_2,\Sigma_1}^* (\theta_1) \le \theta_2 .\]
Let $\XSig$ be the normalization of $X\times_S T$ and 
write $\CSig\in \Div(X_\Sigma)^+ $ for the pullback of $\cup_{\lambda\in J} C_\lambda $ to $\XSig$.
We also write $\thSig\in \Div(X_\Sigma)^+$ for the pullback of $\theta$ to $\XSig$.
By $U_\Sigma$ we denote the preimage of $U$ in $X_\Sigma \setminus {\rm supp}(\thSig)$.
Using the compatibility of \'etale cohomology with directed inverse limits of schemes we
get an isomorphism
\begin{equation}\label{pf.eqlim1}
\varinjlim_\Sigma  H^1( U_\Sigma )  \xrightarrow{\sim} H^1(U_{\bar \eta} ).
\end{equation} 

Thinking of $U_{\bar \eta} $ as a smooth curve over $\bar \eta$ with compactification
$X_{\bar \eta} $ we endow  the cohomology group $H^1(U_{\bar \eta} )$ with the
ramification filtration 
$$\fil m H^1(U_{\bar\eta}) =     
\Ker \big(H^1(U_{\bar\eta}) \to \underset{\lam\in J}{\bigoplus}\; H^1(K_{\lam}^- )/ \fil m H^1( K_\lam^-) \big)
$$ 
where $K_{\lam}^-$ is the quotient field of the henselization of $X_{\bar\eta}$ at the preimage of $\lam$.
Note that
\begin{equation}\label{pf.eqlim1.1}
\fil m H^1(U_{\bar\eta}) =  \varinjlim_\Sigma \fil m H^1(U_{\Sigma,\eta}),
\end{equation} 
where $U_{\Sigma,\eta}=U_{\Sigma}\times_S \eta$ and 
$$ \fil m H^1(U_{\Sigma,\eta}) =     
\Ker \big(H^1(U_{\Sigma,\eta}) \to \underset{\lam\in J}{\bigoplus}\; H^1(K_{\Sigma,\lam})/ 
\fil m H^1( K_{\Sigma,\lam}) \big),
$$ 
and $K_{\Sigma,\lam}$ is the quotient field of the henselization of $\XSig$ at the preimage of $\lam$.
If we fix $T$ and take ${\rm supp}(\theta)$ large enough, then $X_\Sigma \setminus{\rm supp}(\thSig)$ is smooth and 
$\CSig\setminus {\rm supp}(\thSig)$ is a regular divisor on it.
Hence Corollary~\ref{globalrt.defram} implies 
\begin{equation}\label{pf.eqlim1.2} 
\fil m H^1(U_{\Sigma,\eta}) = \varinjlim_\theta \fil{m\CSig+\thSig}  H^1( \USig ) ,
\end{equation} 
where $\theta$ ranges over the effective divisors on $T$ and $\fil{m\CSig+\thSig}  H^1( \USig)$ is defined as
Definition \ref{gloram.def} for the effective Cartier divisor $m\CSig+\thSig$ on $X_\Sigma$.
Combining \eqref{pf.eqlim1}, \eqref{pf.eqlim1.1} and \eqref{pf.eqlim1.2}, we get an isomorphism
\begin{equation}\label{pf.eqlim2}
\fil m H^1(U_{\bar \eta} ) \xrightarrow{\sim} \varinjlim_\Sigma \fil{m\CSig+\thSig}  H^1( \USig ) .
\end{equation}
Composing \eqref{pf.eqlim2} with the dual reciprocity map (see Proposition~\ref{cft.recmap}),
we get a homomorphism
\begin{equation}
\Psi_{\bar\eta}^{(m)} : \fil m H^1(U_{\bar\eta} )  \to \varinjlim_\Sigma \C(X_\Sigma, m\CSig + \thSig)^\vee.
\end{equation}
By Wiesend's trick (Lemma~\ref{cft.trick}), the surjectivity of $\Psi_{\bar\eta}^{(m)}$ implies ${\bf I}_{U}$. 
\medbreak

The following result is essentially due to Wiesend.

\begin{lem}[Wiesend]
The map  $ \Psi_{\bar\eta}^{(1)} $ is surjective.
\end{lem}
\begin{proof}
Consider $\chi \in \C(X_\Sigma, \CSig + \theta)^\vee $ for some $\Sigma=(T,\theta)$. By
Wiesend's trick Lemma~\ref{cft.trick} it is enough to construct a quasi-finite map $U' \to U_\Sigma$ with dense image such that the pullback of $\chi$ to $U'$ vanishes.  
This is proved by the same argument as the proof of \cite[Prop.3.6]{KeSc}.
For convenience of the readers we recall it (see the first part of \cite[\S4]{KeSc}).

By Corollary~\ref{tameCFT.cor}, $\C(\XSig, \CSig + \thSig)^0$ is of finite exponent so that the map
\[
\varinjlim_n \Hom_\cont (\C(\XSig, \CSig + \thSig), \Z / n  )  \xrightarrow{\sim} \C(\XSig, \CSig + \thSig)^\vee
\]
is an isomorphism, where $n$ ranges over all positive integers. Thus we can find $n>0$ such that 
$$\chi \in  \Hom_\cont (\C(\XSig, \CSig + \thSig), \nz) \subset \C(\XSig, \CSig + \thSig)^\vee.$$
If we pull back $\chi$ along the section $\sigma: T \to X_\Sigma$ we get a character in
$\C(T, \theta'  )^\vee$ for some effective divisor $\theta'$ on $T$. By one-dimensional global class field
theory it comes from a cohomology element of $H^1(T \setminus |\theta'| )$ via the dual reciprocity map. 
By making a base change in the base $T$, we can assume that this cohomology element vanishes.   

The maximal abelian pro-finite \'etale covering $U'\to (U_\Sigma)_\eta$ which splits over the image of
$\sigma$, and whose Galois group is $n$-torsion and which is tame along $(\CSig)_\eta$ is finite.
Let $\chi':W(U') \to \nz$ be the pullback of $\chi$ via the composite map 
\[W(U') \to W(\USig) \twoheadrightarrow \C(\XSig, \CSig + \thSig),\]
where the first map is the map \eqref{eq.Wnorm} induced by $U'\to \USig$.
By a specialization argument for tame fundamental groups and one-dimensional class field theory for the fibers of
$U_\Sigma \to T$, we can deduce that the pullback of $\chi'$ to the
class group of the fibers of $U'\to T$ are trivial (see the last paragraph of page 2579 of \cite{KeSc}).
By the Chebotarev density theorem (see \cite[Th.7]{Se2}) this implies $\chi'$ is trivial as desired.
\end{proof}

\medskip

We now prove $\Psi_{\bar\eta}^{(m)} $ is surjective for  $m\ge 2$ by induction on $m$.
Consider the exact localization sequence
$$
0\to H^1(X_{\bar\eta}) \to H^1(U_{\bar\eta}) \to \underset{\lam\in J}{\bigoplus}\;
H^1(K_\lam^-)
 \rmapo{\iota} H^2(X_{\bar\eta},\qz).
$$
It induces the exact sequences in the commutative diagram
\[
\xymatrix{
0 \ar[r] &  H^1(X_{\bar\eta}) \ar[r] \ar@{=}[d] & \fil {m-1}  H^1(U_{\bar\eta}) \ar[r] \ar[d] & \bigoplus_{\lam\in J}\;
\fil {m-1} H^1(K_\lam^-) \ar[r]^(0.65){\iota}  \ar[d] & \iota( \fil {m-1} ) \ar[d] \ar[r]
& 0\\
0 \ar[r] &  H^1(X_{\bar\eta}) \ar[r] & \fil {m}  H^1(U_{\bar\eta}) \ar[r] & \bigoplus_{\lam\in J}\;
\fil {m} H^1(K_\lam^-) \ar[r]^(0.65){\iota} & \iota( \fil {m} )  \ar[r]
& 0
}
\]
Here the vertical maps are the canonical inclusions. Taking cokernels of the vertical maps we
get an exact sequence
\begin{equation}\label{pf.lex1}
0 \to \gr m  H^1(U_{\bar\eta}) \to  \bigoplus_{\lam\in J}  \gr m H^1(K_\lam^-)  \xrightarrow{\iota} \iota( \fil {m} ) / \iota( \fil {m-1} ) \to 0.
\end{equation}
The map $\iota$ in  \eqref{pf.lex1} vanishes, because $ \iota( \fil {m} ) / \iota( \fil
{m-1} )$ is a subquotient of the cohomology group
$H^2(X_{\bar\eta},\qz)$, which has no $p$-torsion by \cite[X, Thm.\ 5.1]{SGA4}, and $ \gr m H^1(K_\lam^-)$ is 
a $p$-primary torsion group. This implies the exactness of the left column of the following diagram:
\begin{equation}\label{pf.diagram2}
\xymatrix{
\fil{m-1} H^1(U_{\bar\eta}) \ar[r]^(0.4){\Psi^{(m-1)}_{\bar\eta}} \ar[d] & 
\varinjlim_\Sigma \C(\XSig,(m-1)\CSig+ \thSig)^\vee
\ar[d] \\
\fil{m} H^1(U_{\bar\eta}) \ar[r]^(0.4){\Psi^{(m)}_{\bar\eta}} \ar[d] & 
\varinjlim_\Sigma \C(\XSig,m\CSig+ \thSig)^\vee
\ar[d]^{\rm cc}\\
 \bigoplus_{\lam\in  J}  \gr m H^1(K_\lam^-)  \ar@{->>}[r]^(0.45){-\oplus_\lam {\rsw}_\lam}  \ar[d] &  
\bigoplus_{\lam\in J}  \Omega^1_{X_{\bar\eta}} (mC)\otimes  \overline{k(C_\lam)} \\
0 &  
}
\end{equation}
The map $\rm cc$ is induced by the cycle conductor defined in \S \ref{keytheorem} as follows.
For $\Sigma=(T,\theta)$ take the minimal desingularization $\tXSig\ \to \XSig$ and 
let $\tCSig$ be the proper transform of $\CSig$ in $\tXSig$. Note that the projection $\tCSig\to \CSig$ is an 
isomorphism. Let $\tthSig\in \Div(\tXSig)^+$ be the pullback of $\theta$ to $\tXSig$.
If ${\rm supp}(\theta)\subset T$ is sufficiently large, we have 
$\tXSig \setminus {\rm supp}(\tthSig) \simeq \XSig \setminus {\rm supp}(\thSig)$ so that the natural map
\[\varprojlim_\theta \C(\tXSig,m\tCSig+ \tthSig) \to \varprojlim_\theta  \C(\XSig,m\CSig+ \thSig)\]
is an isomorphism, where the limit is taken over the effective divisors on a fixed $T$.
Hence Theorem~\ref{keythm} and \eqref{rswMXDexactseq} applied to $(\tXSig,m\tCSig+ \tthSig)$ implies an exact sequence
\[
\varinjlim_\theta \C(\XSig,(m-1)\CSig+ \thSig)^\vee \to \varinjlim_\theta \C(\XSig,m\CSig+ \thSig)^\vee
\rmapo{cc} \bigoplus_{\lam\in J}  \Omega^1_{\XSig} (m\CSig)\otimes  k(C_\lam). \]
The right vertical sequence in \eqref{pf.diagram2} is deduced from this so it is exact.
The map ${\rsw}_\lam$ in \eqref{pf.diagram2} is the refined Artin conductor for $K_\lam^-$ recalled in 
Theorem~\ref{thm.Masuda} and it is surjective. The lower square of \eqref{pf.diagram2} commutes by Theorem~\ref{keythm2}. The commutativity of the upper square is obvious.
A diagram chase shows that the surjectivity of $\Psi^{(m-1)}_{\bar\eta}$ implies the
surjectivity of $\Psi^{(m)}_{\bar\eta}$. This finishes the induction and therefore the
proof of Theorem~\ref{CFT.mainthm}. 


\section{Reciprocity at a closed point}\label{recclosedpoint}

Let the notation be as in \S\ref{classgroup}. 
Let $X=(X,C)$ be in $\cC$ (cf. Definition \ref{def.cC}).
The purpose of this section is  to reduce Proposition~\ref{prop.Recpoint} and Corollary~\ref{coro.Recpoint} to 
Lemma \ref{lem.3term} below.
It will play a crucial role in the proof of Theorem~\ref{keythm} in \S\ref{keythmproof}.
We note that the arguments in this section work over any perfect field $k$ (not necessarily finite).
\medbreak

For a regular closed point $x$ of $C$, there is a natural isomorphism
\[
H^1_x(C,\Omega^1_X(-D)\otimes_{\cO_X}\cO_C)\simeq
\frac{\Omega^1_X(-D)\otimes_{\cO_X} k(\Clam)}{\Omega^1_X(-D)\otimes_{\cO_X} \cO_{C,x}}
\]
where $\Clam$ is the irreducible component of $C$ containing $x$. 
Thus we have a natural map
\begin{equation}\label{keythmproof.eq0}
\iota_x: \frac{\Omega^1_X(-D)\otimes_{\cO_X} \cO_{C,x}(x)}
{\Omega^1_X(-D)\otimes_{\cO_X} \cO_{C,x}}  \to H^1_x(C,\Omega^1_X(-D)\otimes_{\cO_X}\cO_C)
\end{equation}
We will let $\iota_x$ denote also the composite of $\iota_x$ and the natural map
\[ H^1_x(C,\Omega^1_X(-D)\otimes_{\cO_X}\cO_C) \to  H^1(C,\Omega^1_X(-D)\otimes_{\cO_X}\cO_C).\]
Let $Z\in \ZXCC$ be such that $Z\Cap C$ at $x$.
Locally on a neighborhood of $x$, we have a commutative diagram of exact sequences (cf. \eqref{OmegaXZes})
\begin{equation*}\label{OmegaXZes2}
\xymatrix{
0\ar[r] & \Omega^1_X(-D) \ar[r]\ar[d]^{=} & \Omega^1_X(\log Z)(-D) \ar[r]\ar[d] & \cO_Z(-D) \ar[r]\ar[d] & 0 \\ 
0\ar[r] & \Omega^1_X(-D) \ar[r] & \Omega^1_X(Z-D) \ar[r] & \Omega^1_X(Z-D)\otimes_{\cO_X} \cO_Z \ar[r] & 0\;.\\
}
\end{equation*}
Tensored with $\cO_C$, it induces a commutative diagram of boundary maps
\begin{equation*}\label{pdZx}
\xymatrix{
\cO_{Z,x}(-D)\otimes \k(x) \ar[r]^{\hskip -20pt\partial_{Z,x}} \ar[d]^{d_{Z,x}} & H^1_x(C,\Omega^1_X(-D)\otimes_{\cO_X}\cO_C)\ar[d]^{=} \\ 
\frac{\Omega^1_X(-D)\otimes_{\cO_X} \cO_{C,x}(x)}
{\Omega^1_X(-D)\otimes_{\cO_X} \cO_{C,x}}   \ar[r]^{\hskip -20pt\iota_x} & H^1_x(C,\Omega^1_X(-D)\otimes_{\cO_X}\cO_C) \\ 
}
\end{equation*}
where $d_{Z,x}$ has the following explicit description:
Let $f\in \cO_{X,x}$ be a local equation of $Z$ at $x$. Then, for $a\in \cO_{X,x}(-D)$, we have
\begin{equation}\label{dZxexplicit}
d_{Z,x}(a_{|Z}) = a_{|C}\cdot \frac{df}{f} \;\in \frac{\Omega^1_X(-D)\otimes_{\cO_X} \cO_{C,x}(x)}
{\Omega^1_X(-D)\otimes_{\cO_X} \cO_{C,x}} .
\end{equation}

\begin{prop}\label{prop.Recpoint}
Let $x$ be a regular closed point of $C$.
Let $F,Z_1,Z_2\in \ZXCC$ be such that $F\Cap C$ at $x$, and $F\Cap Z_i$ and $Z_i\Cap C$ at $x$ for $i=1,2$.
For $a\in \cO_{F,x}(-D)\otimes \k(x)$ and $b_i\in \cO_{Z_i,x}(-D)\otimes \k(x)$ for $i=1,2$ satisfying
\[
d_{F,x}(a)= d_{Z_1,x}(b_1) +  d_{Z_2,x}(b_2)\;\in \;  \frac{\Omega^1_X(-D)\otimes_{\cO_X} \cO_{C,x}(x)}
{\Omega^1_X(-D)\otimes_{\cO_X} \cO_{C,x}},
\]
we have
\[\mu_{F,x}(a)= \mu_{Z_1,x}(b_1) +  \mu_{Z_2,x}(b_2)\;\in \; \WU/\FXCC {D+C}.\]
\end{prop}

Before going to the proof, we give its corollary. 

\begin{coro}\label{coro.Recpoint}
For a regular closed point $x\in C$, there exists a unique homomorphism
\[
\mu_x: \frac{\Omega^1_X(-D)\otimes_{\cO_X} \cO_{C,x}(x)}{\Omega^1_X(-D)\otimes_{\cO_X} \cO_{C,x}} \to\WU/\FXCC {D+C}
\]
such that for every $Z\in \ZXCC$ with $Z\Cap C$ at $x$, the following diagram is commutative
\begin{equation}\label{recpoint.CD}
\xymatrix{
\cO_{Z,x}(-D)\otimes \k(x) \ar[r]^{\hskip -10pt d_{Z,x}} \ar[rd]_{\mu_{Z,x}} &  
\frac{\Omega^1_X(-D)\otimes_{\cO_X} \cO_{C,x}(x)}{\Omega^1_X(-D)\otimes_{\cO_X} \cO_{C,x}}  \ar[d]^{\mu_x} \\
& \hskip 30pt W(U)/\FXCC {D+C}.\\
}\end{equation}
\end{coro}

\begin{rem}\label{prop.Recpoint.rem}
By Proposition \ref{prop.Recpoint}, Theorem~\ref{keythm}$(i)$
is equivalent to the commutativity of the following diagram for any closed point $x$ of $C\backslash \Xi$:
\begin{equation}\label{keythm.CD}
\xymatrix{
&  H^1(C,\Omega^1_X(-D-\Xi)\otimes_{\cO_X}\cO_C) \\
\frac{\Omega^1_X(-D)\otimes_{\cO_X} {\cO_{C,x}(x)}}{\Omega^1_X(-D)\otimes_{\cO_X} {\cO_{C,x}}}
 \ar[ru]^{\iota_x} \ar[rd]^{\mu_x} & \Biggl\downarrow\;\phiD\\
& \hskip 30pt W(U)/\hFXCC {D+C}.\\
}
\end{equation}
\end{rem}
\bigskip

Now we deduce Corollary~\ref{coro.Recpoint} from Proposition~\ref{prop.Recpoint}.
Put
\begin{equation}\label{Lx}
\Lx=\frac{\Omega^1_X(-D)\otimes_{\cO_X} {\cO_{C,x}(x)}}{\Omega^1_X(-D)\otimes_{\cO_X} {\cO_{C,x}}}.
\end{equation}
We have
\[\Omega^1_{X,x}=\cO_{X,x}\cdot d\pi\oplus \cO_{X,x}\cdot df.\]
Hence any element $\xi\in \Lx$ is written in a unique way as
\begin{equation}\label{Lxbasis}
\xi=\frac{1}{f}(\alpha d\pi + \beta df)\qwith \alpha,\beta\in \cO_{X,x}(-D)\mod \cO_{X,x}(-D-C).
\end{equation}
For $Z\in \ZXCC$ such that $Z\Cap C$ at $x$, put
\[A(Z)=\cO_{Z,x}\otimes\k(x).\]
Let $F,Z\in \ZXCC$ be such that $F\Cap C$, $F\Cap Z$, $Z\Cap C$ at $x$.
From \eqref{dZxexplicit} one easily see that the map
\[d_{F,Z} : A(F)\oplus A(Z) \to \Lx\;;\; (a,b) \to d_{F,x}(a) + d_{Z,x}(b)\]
is an isomorphism. Define the composite map
\[\mu_{F,Z} : \Lx \rmapo{(d_{F,Z})^{-1}} A(F)\oplus A(Z) \rmapo{\mu_{F,x}\oplus \mu_{Z,x}} \WU/\FXCC{D+C}.\]
We claim $\mu_{F,Z}$ independent of the choice of $F,Z$ as above.
Indeed, assume given two choices $F,Z_1$ and $F,Z_2$ as above and $a,a'\in A(F)$ and $b_i\in A(Z_i)$ with $i=1,2$.
Then, by Proposition~\ref{prop.Recpoint} an equality
\[d_{F,x}(a) + d_{Z_1,x}(b_1) = d_{F,x}(a') + d_{Z_2,x}(b_2)\]
implies an equality
\[\mu_{F,x}(a) + \mu_{Z_1,x}(b_1) = \mu_{F,x}(a') + \mu_{Z_2,x}(b_2).\]
Hence we get $ \mu_{F,Z_1}=\mu_{F,Z_2}$. By the same argument we also get $ \mu_{F_1,Z}=\mu_{F_2,Z}$
for two choices $F_1,Z$ and $F_2,Z$ as above. This implies the desired claim.
The commutativity of \eqref{recpoint.CD} is obvious from the construction and the proof of
Corollary~\ref{coro.Recpoint} is complete.
\medbreak

\begin{rem}\label{coro.Recpoint.rem}
The above argument gives the following explicit description of $\mu_x$.
For $g\in \cO_{X,x}$, let $\divXx g$ denote
the effective Cartier divisor on $X$ obtained from $\div_X(g)$ by removing its components which do not
contain $x$. Note $\divXx g=\divXx {ug}$ for $u\in \cO_{X,x}^\times$. 
Take a system of regular parameter $(\pi,f)$ of $\cO_{X,x}$ such that $\pi$ is a local parameter of $C$ at $x$. 
Then we have (cf. \eqref{Lxbasis})
\begin{equation}\label{mufpi}
 \mu_x \big(\frac{1}{f}(\alpha d\pi + \beta df) \big)= \sym{1+(\beta-\alpha)}{F,x} + \sym{1+\alpha}{F_\pi,x},
\end{equation}
where $F=\divXx f,\; F_\pi=\divXx{f+\pi}\in \ZXCC$.
\end{rem} 
\medbreak

We deduce Proposition~\ref{prop.Recpoint} from the following lemma whose proof will be given in \S\ref{keylemproofI}
(see Lemma~\ref{lem.3term2} in \S\ref{keylemmas}).

\begin{lem}\label{lem.3term}
Let $x$ be a regular closed point of $C$.
Let $F,Z_1,Z_2\in \ZXCC$ be such that $F\Cap C$ at $x$, and $F\Cap Z_i$ and $Z_i\Cap C$ at $x$ for $i=1,2$.
Let $(\pi,f)$ be a system of regular parameters such that 
\begin{equation}\label{lem.3term.eq}
F=\divXx f,\; Z_1=\divXx {u_1f+\pi},\; Z_2=\divXx {u_2f+\pi},\; C=\divXx {\pi},
\end{equation}
where $u_1,u_2\in \cO_{X,x}^\times$. For $\alpha\in \cO_{X,x}(-D)$, we have
\[
\sym{1-(u_1-u_2)\alpha}{F,x} + \sym{1-u_1\alpha}{Z_1,x} -\sym{1-u_2\alpha}{Z_2,x}\;\in 
\FXCC {D+C}.
\]
\end{lem}
\bigskip

By the assumption of Proposition~\ref{prop.Recpoint}, $F,Z_1,Z_2$ can be described as \eqref{lem.3term.eq}.
For $a,b_1,b_2\in \cO_{X,x}(-D)$, assume an equality
\[d_{F,x}(a_{|F})= d_{Z_1,x}((b_1)_{|Z_1}) +  d_{Z_2,x}((b_2)_{|Z_2})\;\in \;  \Lx\]
holds. Using \eqref{dZxexplicit} and \eqref{lem.3term.eq}, one can compute
\begin{equation}\label{lem.3term.eq2.5}
d_{Z_i,x}((b_i)_{|Z_i}) = \frac{df+ u_i^{-1} d\pi}{f}\qfor i=1,2.
\end{equation}
Thus \eqref{lem.3term.eq2.5} is equivalent to equalities
\[\overline{a} = \overline{b_1} + \overline{b_1},\; \overline{u_2b_1} + \overline{u_1b_2}=0
\;\in\; \cO_{X,x}(-D)\otimes\k(x),\]
where $\overline{z}\in\cO_{X,x}(-D)\otimes\k(x)$  is the residue class of $z\in \cO_{X,x}(-D)$.
This implies that there exists $\alpha\in \cO_{X,x}(-D)$ such that 
\[\overline{b_1}=\overline{u_1\alpha},\; \; \;\overline{b_2}=-\overline{u_2\alpha},\;\; 
\overline{a}=\overline{(u_1-u_2)\alpha}.\]
Hence the desired equality of Proposition~\ref{prop.Recpoint} follows from Lemma~\ref{lem.3term}.


\section{Reciprocity along the boundary}\label{keythmproof}

The purpose of this section is to reduce Theorem \ref{keythm}$(i)$ to Lemma \ref{lem.3term} and Lemma \ref{keylemIIfinite} below. 
In this section we always assume $k$ is finite. 
A key point is Proposition~\ref{keythm.prop} concerning reciprocity property of
the map $\mu_x$ from Corollary~\ref{coro.Recpoint} when $x$ moves along $C$.
First we prove a preliminary lemma.

\begin{lem}\label{keythmproof.lem0}
Let $X$ and $C$ be as in Theorem \ref{keythm}.
Let $k'/k$ be a finite Galois extension with $G=\Gal(k'/k)$ and put $X'=X\otimes_k k'$ and $C'=C\otimes_k k'$. 
Assume that there exists $\Xi'\in \Div(X',C')^+$ such that $C'_{sing}\subset\Xi'$ and $\Xi'$ is independent of 
$D$ as in \eqref{eqD0}  and a map defined for every $D'=D\otimes_k k'$ with $D$ as in \eqref{eqD0}:
\[ \phi_{X',D'}: H^1(C',\Omega^1_{X'}(-D'-\Xi')\otimes_{\cO_{X'}}\cO_{C'}) \to \WUkp/\hFXCCkp {D'+C'},\]
such that the following diagram commutes (cf. Remark \ref{prop.Recpoint.rem}):
\begin{equation}\label{keythm.CD2}
\xymatrix{
\underset{y|x}{\bigoplus}\;
\frac{\Omega^1_{X'}(-D')\otimes_{\cO_{X'}} {\cO_{C',y}(y)}}{\Omega^1_{X'}(-D')\otimes_{\cO_{X'}} {\cO_{C',y}}}
\ar[r]^{\hskip -20pt\sum_{y|x}\iota_y} \ar[rd]_{\hskip -10pt\sum_{y|x}\mu_y} &  
\hskip 10pt H^1(C',\Omega^1_{X'}(-D'-\Xi')\otimes_{\cO_{X'}}\cO_{C'}) \ar[d]^{\phi_{X',D'}}\\
& \hskip 30pt \WUkp/\hFXCCkp {D'+C'},\\
}
\end{equation}
where $x$ is any fixed regular closed point of $C$ not lying in the image of $\Xi'$ and $y$ ranges over the points of $C'$ lying over $x$.
Then there exists $\Xi\in \ZXCC$ such that $C_{sing}\subset\Xi$ and $\Xi$ is independent of $D$ as in \eqref{eqD0} 
and a map defined for every $D$ as in \eqref{eqD0}:
\[\phiD: H^1(C,\Omega^1_X(-D-\Xi)\otimes_{\cO_X}\cO_C)\to W(U)/\hFXCC {D+C},\]
which satisfies the condition $(i)$ of Theorem \ref{keythm}.
\end{lem}
\begin{proof}
We may replace $\Xi'$ by the sum of its Galois conjugates to assume that there exists $\Xi\in \Div(X,C)^+$ such that
$\Xi'=\Xi\otimes_k k'$. Note that \eqref{keythm.CD2} implies that $\phi_{X',D'}$ is $G$-equivariant 
since so are $\sum_{y|x}\iota_y$ and $\sum_{y|x}\mu_y$. The trace map induces an isomorphism
\[Tr_{k'/k}: H^1(C',\Omega^1_{X'}(-D'-\Xi')\otimes_{\cO_{X'}}\cO_{C'})_G \isom H^1(C,\Omega^1_X(-D-\Xi)\otimes_{\cO_X}\cO_C),\]
where $M_G$ denotes the coinvariants of a $G$-module $M$. 
We then define $\phi_{X,D}$ as the composite
\begin{multline*}
H^1(C,\Omega^1_X(-D-\Xi)\otimes_{\cO_X}\cO_C) \rmapo{(Tr_{k'/k})^{-1}} 
H^1(C',\Omega^1_{X'}(-D'-\Xi')\otimes_{\cO_{X'}}\cO_{C'})_G\\
\rmapo {\phi_{X',D'}}  \Big(\WUkp/\hFXCCkp {D'+C'}\Big)_G \rmapo {N_{k'/k}} \WU/\hFXCC {D+C},
\end{multline*}
where the last map is induced by the norm map $N_{k'/k}: \WUkp \to \WU$.
The condition of Remark \ref{prop.Recpoint.rem} for $\phi_{X,D}$ follows from \eqref{keythm.CD2} thanks to
the commutativity of the following diagrams
\[
\xymatrix{
\underset{y|x}{\bigoplus}\;
\frac{\Omega^1_{X'}(-D')\otimes_{\cO_{X'}} {\cO_{C',y}(y)}}{\Omega^1_{X'}(-D')\otimes_{\cO_{X'}} {\cO_{C',y}}}
\ar[r]^{\hskip -10pt\sum_{y|x}\mu_y} \ar[d]^{Tr_{k'/k}} & \;\; \WUL/\hFXCCkp {D'+C'}\ar[d]^{N_{k'/k}}\\
\frac{\Omega^1_X(-D)\otimes_{\cO_X} {\cO_{C,x}(x)}}{\Omega^1_X(-D)\otimes_{\cO_X} {\cO_{C,x}}} 
\ar[r]^{\hskip -20pt\mu_x} & \WU/\hFXCC {D+C},\\
}\]
\[
\xymatrix{
\underset{y|x}{\bigoplus}\;
\frac{\Omega^1_{X'}(-D')\otimes_{\cO_{X'}} {\cO_{C',y}(y)}}{\Omega^1_{X'}(-D')\otimes_{\cO_{X'}} {\cO_{C',y}}}
\ar[r]^{\hskip -20pt\sum_{y|x}\iota_y} \ar[d]^{Tr_{k'/k}} & \;\;H^1(C',\Omega^1_{X'}(-D'-\Xi')\otimes_{\cO_{X'}}\cO_{C'}) 
\ar[d]^{Tr_{k'/k}}\\
\frac{\Omega^1_X(-D)\otimes_{\cO_X} {\cO_{C,x}(x)}}{\Omega^1_X(-D)\otimes_{\cO_X} {\cO_{C,x}}} 
\ar[r]^{\hskip -20pt\iota_x} &
H^1(C,\Omega^1_X(-D-\Xi)\otimes_{\cO_X}\cO_C).\\
}\]
\end{proof}

\begin{defi}\label{pencil.def}
Let $(X,C)$ be in $\cC$ (see Definition \ref{def.cC}). 
\begin{itemize}
\item[(1)]
Let $H\subset X$ be a hyperplane section. For an integer $d>0$ 
let $\cL(d)=|dH|$ be the linear system on $X$ of hypersurface sections of degree $d$.
For $t\in \cL(d)$ let $F_t\subset X$ be the corresponding section.
We write $Gr(1,\cL(d))$ for the Grassmannian variety of lines in $\cL(d)$.
\item[(2)]
A pencil $\{F_t\}_{t\in L}$ of hypersurface sections parametrized by $L\in Gr(1,\cL(d))$,
is admissible for $(X,C)$ if $\Delta_L\cap C=\varnothing$ for the axis $\Delta_L$ of $L$ and 
$F_t\Cap C $ for almost all $t\in L$.
\end{itemize}
\end{defi}
By \cite[XVIII 6.6.1]{SGA7} and Lemma \ref{keythmproof.lem0}, we may assume by replacing $k$ by its finite extension
that for a sufficiently large $d$, there exists $L\in Gr(1,\cL(d))$ admissible for $(X,C)$.
In what follows we fix such $L\in Gr(1,\cL(d))$ and $\pi\in k(X)$ satisfying:
\begin{equation}\label{eq.pi}
\div_X(\pi)=C + G_0 -G_\infty\qwith G_0,G_\infty\in \ZXCC. 
\end{equation}
We also fix a finite set $\TL\subset L$ such that 
\begin{equation}\label{eq.Sigma}
\begin{aligned}
&F_t\Cap C \qaq F_t\cap C\cap (G_0\cup G_\infty)=\varnothing\qfor t\in L-\TL.\\
\end{aligned}
\end{equation}
We have the rational map
\[
h_L: X\cdots\to L\;;\; x\to t\text{ such that } x\in F_t.
\]
By Definition \ref{pencil.def}(2) $h_L$ is defined at any point of $C$ and it gives rise to
\[
\cO_{L,t}\hookrightarrow \cO_{X,x}\qfor t\in L\text{ and } x\in F_t\cap C.
\]

\begin{lem}\label{keythmproof.lem1}
For each $t\in L$, choose a prime element $f_t\in \cO_{L,t}$. 
\begin{itemize}
\item[(1)]
For $t\in L-\TL$ and $x\in F_t\cap C$, we have 
\begin{equation*}
\Omega^1_{X,x}=\cO_{X,x}\cdot d\pi\oplus \cO_{X,x}\cdot df_t.
\end{equation*}
\item[(2)]
There exists an effective divisor $\theta$ on $L$ independent of the choice of $f_t$ such that $|\theta|=\TL$ and that for any $t\in \TL$ and $x\in F_t\cap C$ and for any 
\[
\omega=\frac{1}{f_t}(\xi_1 d\pi + \xi_2 df_t)\in \Omega^1_X\otimes_{\cO_X} k(C)\qwith \xi_i\in 
k(C)=\underset{\lam\in I}{\prod}\;k(\Clam),
\]
we have the implication
\[
\omega\in \Omega^1_X(-\Fth) \otimes_{\cO_X}\cO_{C,x}\Rightarrow \xi_i\in \cO_{C,x}(-F_t),
\]
where 
$\Fth=\underset{t\in \TL}{\sum}\; e_t F_t$ for $\theta=\underset{t\in \TL}{\sum}\; e_t t$ with
$e_t\in \bZ_{\geq 1}$.
\end{itemize}
\end{lem}
\begin{proof}
(1) follows from the fact that $(\pi,f_t)$ is a system of regular parameters of $\cO_{X,x}$
if $t\in L-\TL$ and $x\in F_t\cap C$. To show (2), note
\[
\Omega^1_X\otimes_{\cO_X} k(C) = k(C)\cdot \frac{d\pi}{f_t}\oplus k(C)\cdot \frac{d f_t}{f_t}
\]
and put
\[
\Theta_x= \cO_{C,x} \cdot \frac{d\pi}{f_t}\oplus \cO_{C,x}\cdot \frac{d f_t}{f_t}\;\subset\; \Omega^1_X\otimes_{\cO_X} k(C).
\]
We see that $\Theta_x$ is independent of the choice $f_t$, namely 
for $f'_t=u f_t$ with $u\in \cO_{L,t}^\times$, 
\[
 \Theta_x=\cO_{C,x} \cdot \frac{d\pi}{f'_t}\oplus \cO_{C,x}\cdot \frac{d f'_t}{f'_t}.
\]
Thus (2) follows from the fact that there exists an effective divisor $\theta$ on $L$ 
such that $|\theta|= \TL$ and that for any $t\in \TL$ and $x\in F_t\cap C$, we have
\[
\Omega^1_X\otimes_{\cO_X} \cO_{C,x}(-\Fth) \subset \cO_{C,x}(-F_t) \cdot \Theta_x.
\]
\end{proof}

We fix $\theta$ as in Lemma \ref{keythmproof.lem1} and put 
\begin{equation}\label{eqXi}
\Xi=\Fth + G_\infty \in \ZXCC.
\end{equation}
Note that $\Xi$ is independent of $D$ as in \eqref{eqD0}. 
Let $B$ be a finite subset of $C$ such that 
\begin{equation}\label{eq.B}
B\cap \Xi=\varnothing \qaq h_L(B)\cap \TL=\varnothing.
\end{equation}
Note that this implies that $B$ consists of regular points of $C$ by \eqref{eq.Sigma}.
For $D$ as in \eqref{eqD0} consider the maps
\begin{equation*}\label{keythmproof.eq1}
\begin{CD}
\underset{x\in B}{\bigoplus}\; \Lx @>{\psi_{L,B}}>> H^1(C,\WX(-D-\Xi)\otimes\cO_C) \\
@VV{\sum \mu_x}V \\
\WU/\FXCC {D+C}\\
\end{CD}
\end{equation*}
where $\Lx$ is defined as \eqref{Lx} and $\psi_{L,B}$ is induced by \eqref {keythmproof.eq0}. 
Taking $B$ large enough, we assume that $\psi_{L,B}$ is surjective.
We have
\[
\Ker(\psi_{L,B})= \Image\big(H^0(C,\WX(-D-\Xi)\otimes\cO_C(B))\big) .
\]

\begin{prop}\label{keythm.prop}
Assume $p\not=2$. Take
\[
\omega\in H^0(C,\WX(-D-\Xi)\otimes\cO_C(B))
\]
and let $\omega_x\in \Lx$ be the image of $\omega$ for $x\in B$. Then we have
\begin{equation}\label{keythm.prop.eq0}
\underset{x\in B}{\sum}\;\mu_x(\omega_x)\;\in\;\hFXCC {D+C} .
\end{equation}
\end{prop}

\begin{rem}\label{keythm.prop.rem0}
It suffices to show the proposition after enlarging $B$.
\end{rem}
\medbreak

The proposition implies the existence of a map
\[
\phi_{L,B}: H^1(C,\WX(-D-\Xi)\otimes\cO_C) \to \WU/\hFXCC {D+C}
\]
such that the following diagram commutes
\begin{equation}\label{keythm.prop.eq}
\xymatrix{
\underset{x\in B}{\bigoplus}\;\Lx \ar[r]^{\hskip -40pt \psi_{L,B}}\ar[d]^{\sum \mu_x} & H^1(C,\WX(-D-\Xi)\otimes\cO_C)\ar[dl]^{\phi_{L,B}} \\
\WU/\hFXCC {D+C} \\
}
\end{equation}
This implies that $\phi_L=\phi_{L,B}$ is independent of $B$ (depending only on $L$ and $\Xi$).
Take any $x\in C\backslash \Xi$. Note $t=h_L(x)\not\in \TL$ (cf. \eqref{eqXi} and Lemma \ref{keythmproof.lem1}(2)) so that 
we can choose such $B$ that $x\in B$. 
Then \eqref{keythm.prop.eq} implies that $\phi_L=\phi_{L,B}$ satisfies \eqref{keythm.CD} for $x$, 
which completes the proof of Theorem \ref{keythm}$(i)$ by Remark~\ref{prop.Recpoint.rem}.
\bigskip

Let $\displaystyle{D=\underset{\lambda\in I}{\sum} \mlam \Clam}$ be as \eqref{eqD0}.
For $\lam\in I$ consider the map
\[h_\lam: H^0(C,\WX(-D-\Xi)\otimes\cO_C(B)) \to  \WX(-D)\otimes k(\Clam) \to 
\Omega^1_{k(\Clam)}\otimes\cO_{X}(-D).\]

\begin{claim}\label{keythm.prop.preclaim2}
After enlarging $B$, we may assume $h_\lam(\omega)\not=0$ for any $\lam\in I$.
\end{claim}
\begin{proof}
We claim that after enlarging $B$, we can find 
\[\xi_\lam\in H^0(C,\WX(-D-\Xi)\otimes\cO_C(B))\qfor \lam\in I\]
such that $h_\lam(\xi_\mu)=0$ for $\mu\in I-\{\lam\}$ and $h_\lam(\xi_\lam)\not=0$.
Admit the claim for the moment.
We may assume further that $h_\lam(\xi_\lam)\not= h_\lam(\omega)$ for any $\lam\in I$ after replacing
$\xi_\lam$ by $c\cdot \xi_\lam$ with $c\in k-\{0,1\}$ if necessary.
Then, putting 
\[\omega_1= \underset{\lam\in I}{\sum}\;\xi_\lam,\quad \omega_2=\omega-\omega_1,\]
we have $h_\lam(\omega_i)\not=0$ for all $\lam\in I$ and $i=1,2$. Noting $\omega=\omega_1+\omega_2$,
\eqref{keythm.prop.eq0} for $\omega$ follows from that for $\omega_1$ and $\omega_2$, which proves 
Claim \ref{keythm.prop.preclaim2}.
 
It remains to show the claim. By Bertini's theorem, we can take $F_1,F_2\in \cL(d_1)$ for a sufficiently large 
$d_1>0$ such that $F_i\Cap C$ for $i=1,2$.
Take $f\in k(X)^\times$ such that $\div_X(f)=F_1-F_2$ and put
\[ df\in H^0(X,\Omega^1_X(2 F_2)).\]
For a sufficiently large $d_2>>d_1$, we can take $F\in \cL(d_2)$ such that 
\[F \Cap C,\; F\supset B,\; F\cap \Xi\cap C =\varnothing,\; h_L(F\cap C)\cap \TL=\varnothing,\] 
and $F_\lam\in \cL(d_2)$ for $\lam\in I$ such that
\[ F_\lam = D + \underset{\mu\in I-\{\lam\}}{\sum} \Cmu + 2 F_2 + \Xi + G_\lam \qwith   G_\lam\in \Div(X,C)^+.\]
Then $B_F=F\cap C$ satisfies the condition \eqref{eq.B} and $B\subset B_F$.
Taking $g_\lam\in k(X)^\times$ such that $\div_X(g_\lam)=F_\lam- F$ for $\lam\in I$, we have
\[ g_\lam df \in H^0(X,\Omega^1_X(-D - \underset{\mu\in I-\{\lam\}}{\sum} \Cmu - \Xi + F)).\]
Let $\xi_\lam$ be the image of $g_\lam df$ under the composite map
\begin{multline*}
H^0(X,\Omega^1_X(-D - \underset{\mu\in I-\{\lam\}}{\sum} \Cmu - \Xi + F)) \hookrightarrow
H^0(X,\Omega^1_X(-D - \Xi + F)) \\
\to H^0(C,\Omega^1_X(-D - \Xi)\otimes\cO_C(B_F)).
\end{multline*}
Then $(\xi_\lam)_{\lam\in I}$ satisfies the claimed condition for $B_F$ instead of $B$.
\end{proof}
\medbreak

We now start the proof of Proposition~\ref{keythm.prop}.
We choose an isomorphism over $k$:
\[\iota: L\simeq \bP^1_k=\Proj(k[T_0,T_1]).\]
For a finite extension $\bF_q$ of $k$, let $L(\bF_q)$ denote the set of points $x\in L$ such that there exists an embedding $k(x)\to \bF_q$ 
(note that the notation is not the standard one that means the set of $k$-morphisms $\Spec \F_q \to L$).
Put 
\[\Liotaqo=\{t\in L(\bF_q)|\; \iota(t)\not=0,\infty\},\;\text{ where } 0=(1:0), \infty=(0:1)\in \bP^1_k.\]
By Lemma \ref{keythmproof.lem0} we may assume that $k$ is large enough so that after a coordinate transformation of $\bP^1_{k}$, we have
\begin{itemize}
\item[$(*1)$]
$h_L(B) \cup \TL \subset \Liotaqo$ for a finite extension $\bF_q$ of $k$. 
\end{itemize}
By Claim \ref{keythm.prop.preclaim2} we may assume
\begin{itemize}
\item[$(*2)$]
$h_\lam(\omega)\not=0$ for any $\lam\in I$.
\end{itemize}
For $t\in L-\TL$ and $x\in F_t\cap C$, let $\omega_{C,x}$ be the image of $\omega$ under the map 
\[
H^0(C,\WX(-D-\Xi)\otimes\cO_C(B)) \to  \WX(-D)\otimes\cO_{C,x}(B) \to \Omega^1_{C,x}(-D)\otimes \cO_C(B).
\]
By $(*2)$, $\omega_{C,x}\not=0$. 
Choose an isomorphism $s :\Omega^1_{C,x}(-D)\otimes \cO_C(B)\simeq \cO_{C,x}$ as $\cO_{C,x}$-modules.
Then the order of $s(\omega_{C,x})\in \cO_{C,x}$ is independent of the choice and is denoted by
$\ord_x(\omega_{C,x})$. The set 
\[\{x\in \underset{t\in L-\TL}{\bigcup} F_t\cap C\;|\; \ord_x(\omega_{C,x})\not=0\}\]
is finite for $\omega$ fixed. Therefore we can choose $\iota:L\simeq \bP^1_k$ (possibly after replacing $k$ by a finite extension) in such a way that the following condition holds.
\begin{itemize}
\item[$(*3)$]
$\ord_x(\omega_{C,x})=0$ for any $x\in  F_{\infty}\cap C$. 
\end{itemize}
For $t\in L-\TL$ and $x\in F_t\cap C$, if $x\not \in B$, then $\omega_x=0\in \Lx$ so that
$\mu_x(\omega_x)=0$. Therefore Proposition~\ref{keythm.prop} follows from the following claim.

\begin{claim}\label{keythm.claim1}
Under the conditions $(*1)$, $(*2)$ and $(*3)$, we have
\[
\underset{t\in \Liotaqo\backslash\TL}{\sum}\;\underset{x\in F_t\cap C}{\sum}\;\mu_x(\omega_x)\;\in\;
\hFXCC {D+C}.
\]
\end{claim}


\begin{rem}\label{keythm.claim1.rem}
It suffices to prove the claim after replacing $\bF_q$ by its finite extension.
\end{rem}

\begin{proof}
We let $0,\infty$ denote the closed points of $L$ which correspond to 
$0,\infty\in \bP^1_{\bF_p}$ by $\iota:L\simeq \bP^1_k=\Proj(k[T_0,T_1])$. Put
\begin{equation}\label{keythm.eq0}
\rho=T_0/T_1\in \bF_p(\bP^1)\qaq b= 1 - \rho^{q-1},
\end{equation}
considered as elements of $k(L)\subset k(X)$ via $\iota: L\simeq \bP^1_k$ and $h_L$. Put 
\begin{equation}\label{keythm.eq0.5}
W_b=\div_X(b+\pi) + G_\infty + (q-1) F_0. \quad\text{(cf. \eqref{eq.pi})}
\end{equation}
Note
\begin{equation}\label{keythm.eq1}
\div_{X}(b)= \underset{t\in \Liotaqo}{\sum} F_t\; - (q-1)\cdot F_0,
\end{equation}

\begin{claim}\label{keythmproof.claim3}
\begin{itemize}
\item[(1)]
$W_b\in \ZXCC$ and
$W_b\cap C\subset \underset{t\in \Liotaqo}{\bigcup} F_t\cap C.$
\item[(2)]
For $t\in \Liotaqo \backslash \TL$ and $x\in F_t\cap C$,
$W_b\Cap C$ at $x$ and $\div_{X,x}(b+\pi)$ (cf. Remark \ref{coro.Recpoint.rem})
is the irreducible component of $W_b$ containing $x$.
\item[(3)]
For $t\in \Liotaqo$ and $x\in F_t\cap C \cap W_b$, we have
\[
\cO_{W_b,x}(-F_t-G_\infty)\subset \cO_{W_b,x}(-C).
\]
\end{itemize}
\end{claim}
\begin{proof}
(1) and (2) follow immediately from \eqref{eq.pi} and \eqref{eq.Sigma} and \eqref{keythm.eq1}
except that $W_b\cap F_0\cap C=\varnothing$, which holds since $C=\div_X(\pi)$ and
$W_b=\div_X(\sigma^{q-1}-1+\pi \sigma^{q-1})$ locally at $F_0\cap C$, where
$\sigma=\rho^{-1}$ is a local parameter of $F_0$ at $F_0\cap C$.
The last fact is checked by using \eqref{eq.Sigma} and $(*1)$ and
$\displaystyle{b+\pi=\frac{\sigma^{q-1}-1+\pi \sigma^{q-1}}{\sigma^{q-1}}.}$
To show (3), let $\pi_\infty$ be a local parameter of $G_\infty$ at $x$ if $x\in G_\infty$ and
$\pi_\infty=1$ otherwise. Putting $\pi'=\pi\pi_\infty$, $\pi'\cO_{W_b,x}\subset \cO_{W_b,x}(-C)$
by \eqref{eq.pi}. Note $x\not\in F_0$ since $F_t \cap F_0\cap C=\varnothing$ for $t\in \Liotaqo$.
Hence \eqref{keythm.eq0.5} implies $W_b=\div_X(b\pi_\infty + \pi')$ locally at $x$. Thus we get
\[
\cO_{W_b,x}(-F_t-G_\infty)=b\pi_\infty \cO_{W_b,x}=\pi'\cO_{W_b,x}\subset \cO_{W_b,x}(-C).
\] 
This completes the proof of Claim \ref{keythmproof.claim3}.
\end{proof}

For $\omega$ from Proposition~\ref{keythm.prop} we write
\[
\omega=\frac{1}{b}(\alpha d\pi + \beta db)\;\text{ in }\; \Omega^1_X(-D) \otimes_{\cO_X} k(C)
\qwith \alpha,\beta\in \cO_X(-D)\otimes_{\cO_X} k(C).
\]
Recall $\omega\in H^0(C,\Omega_X(-D+\Xi)\otimes\cO_C(B))$ and 
$B\subset\underset{t\in \Liotaqo}{\bigcup} F_t\cap C$.
Noting \eqref{keythm.eq1} and 
\begin{equation}\label{keythm.eq2}
\dlog {b} = \frac{\rho^{q-2}d\rho}{1-\rho^{q-1}} = \frac{d\sigma}{\sigma(1-\sigma^{q-1})}\;\;(\sigma=\rho^{-1}),
\end{equation}
Lemma \ref{keythmproof.lem1} implies
\begin{equation}\label{keythm.eq2.5}
\begin{aligned}
& \alpha\in H^0(C,\cO_C(-D+(q-1) F_{0} - \underset{t\in \TL}{\sum} F_t -G_\infty)),\\
& \beta\in H^0(C,\cO_C(-D+(q-2) F_{\infty} - \underset{t\in \TL\cup\{0\}}{\sum} F_t- G_\infty)),
\end{aligned}
\end{equation}
where for a divisor $\Gamma$ on $X$, we write $\cO_C(\Gamma)=\cO_X(\Gamma)\otimes_{\cO_X}\cO_C$.
Since $F_0$ and $F_\infty$ are ample divisors on $X$, the restriction maps
\[
H^0(X,\cO_X(-D +(q-1) F_{0} - \underset{t\in \TL}{\sum} F_t -G_\infty))\to
H^0(C,\cO_C(-D +(q-1) F_{0} - \underset{t\in \TL}{\sum} F_t -G_\infty))),
\]
\[
H^0(X,\cO_X(-D  +(q-2) F_{\infty} - \underset{t\in \TL\cup\{0\}}{\sum} F_t-G_\infty))\to
H^0(C,\cO_C(-D  +(q-2) F_{\infty} - \underset{t\in \TL\cup\{0\}}{\sum} F_t-G_\infty))
\]
are surjective for $q$ sufficiently large (cf. Remark \ref{keythm.claim1.rem}). 
Thus we can take 
\begin{equation}\label{keythm.eq3.1}
\begin{aligned}
&\talpha \in H^0(X,\cO_X(-D +(q-1) F_{0} - \underset{t\in \TL}{\sum} F_t -G_\infty)),\\
&\tbeta \in H^0(X,\cO_X(-D  +(q-2) F_{\infty} - \underset{t\in \TL\cup\{0\}}{\sum} F_t-G_\infty)),\\
\end{aligned}
\end{equation}
such that $\omega= \tomega\otimes k(C)$ with
\[
 \tomega=\frac{1}{b}(\talpha d\pi + \tbeta db) \in \Omega^1_X(-D)\otimes_{\cO_X} \cO_{X,C},
\] 
where $\cO_{X,C}$ is the semi-local ring of $X$ at the generic points of $C$.
By Claim \ref{keythmproof.claim3}(2) and \eqref{mufpi}, 
for $t\in \Liotaqo\backslash \TL$ and $x\in F_t\cap C$, we have
\[
\mu_x(\omega_x) = \sym {1+\tbeta}{F_t,x} - \sym{1+\talpha}{F_t,x} + \sym{1+\talpha}{W_b,x}.
\]
Hence Claim~\ref{keythm.claim1} follows from the following.

\begin{claim}\label{keythm.claim2}
Under the assumption of $(*3)$ we have
\begin{equation}\label{keythmproof.eq3}
\underset{t\in \Liotaqo\backslash \TL}{\sum}\;
\sym {1+\tbeta}{F_t} \;\in \; \hFXCC {D+C} ,
\end{equation}
\begin{equation}\label{keythmproof.eq4}
\underset{t\in \Liotaqo\backslash \TL}{\sum}\;
\big(\sym{1+\talpha}{F_t} - \underset{x\in F_t\cap C}{\sum}\sym{1+\talpha}{W_b,x}\big) \;\in \; \FXCC {D+C}.
\end{equation}
\end{claim}

For the proof of the claim we need Lemma~\ref{keylemIIfinite} below. 

\begin{defi}\label{specialcurve}
Let $F\in\ZXCC$ be a reduced effective Cartier divisor such that $F\Cap C$.
For an integer $e>0$, let $\cP_{D,e}(F)_{(X,C)}= \cP_{D,e}(F)$ denote the set of $a\in H^0(X,\cO_X(-D + e F))$ 
satisfying the condition:
\begin{enumerate}
\item[$(\circledast)$]
The image $a_{F\cap C}\in \cO_{C,F\cap C}(-D + e F)$ of $a$ is a basis as an $\cO_{C,F\cap C}$-module.
\end{enumerate}
Here we note that $\cO_{C,F\cap C}$ is a semi-local ring. For $a\in \cP_{D,e}(F)$, put
\[Z_a=\div_X(1+a) + e F.\]
By Lemma \ref{classgroup.lem2}(1), $Z_a\in \ZXCC$ such that $Z_a\cap C= F\cap C$.
\end{defi}

Fix  $\pi, \piD, f\in \cO_{X,F\cap C}$ such that locally at $F\cap C$
\[C=\div_X(\pi),\; D=\div_X(\piD),\; F=\div_X(f).\]
By the assumption $(\pi,f)$ is a system of regular parameters in $\cO_{X,F\cap C}$. 
The condition $(\circledast)$ is equivalent to the condition that locally at $x\in F\cap C$, 
\begin{equation}\label{specialcurve-eq1}
Z_a=\div_X(f^e+\piD\cdot u) \qwith u\in \cO_{X,F\cap C}^\times.
\end{equation}

The proof of the following lemma will be given later (see Lemma~\ref{keylemII} in \S\ref{keylemmas}).

\begin{lem}[increasing order]\label{keylemIIfinite}
Let $Z_a$ with $a\in \cP_{D,e}(F)$ be as above and take $x\in Z_a\cap C=F\cap C$. Assume 
\[
\begin{aligned}
&H^1(X,\cO_X(-2D-C+ (e-1)F))=H^1(C,\cO_C(-2D+F))=0,\\
\end{aligned}
\leqno{(\star)}
\]
Assume further $p\not=2$. There exists a constant $c>0$ depending only on $X$ and $D$ such that for $e\geq c$, we have 
\[
\sym {1+{f^{e+1}}\cO_{Z_a,F\cap C}} {Z_a,x} \subset \hFXCC {D+C}.
\]
\end{lem}

Note that \eqref{specialcurve-eq1} implies $f^{e+1}|_{Z_a} =- u f\piD$ so that 
$\sym {1+{f^{e+1}}\cO_{Z_a,F\cap C}} {Z_a,x}$ lies in $\FXCC {D}$ without any assumption.

\begin{rem} 
By Serre's vanishing theorem the condition $(\star)$ of Lemma \ref{keylemIIfinite} is satisfied if
$F\in \cL(d)$ for $d>0$ sufficiently large (see Definition \ref{pencil.def}).
\end{rem}
\medbreak\noindent
{\it Proof of Claim~\ref{keythm.claim2}.} We first claim 
$\tbeta \in \cP_{D,q-2}(F_\infty).$
Indeed the condition $(*3)$ implies that the image of $\beta$ under the restriction map 
(cf. \eqref{keythm.eq2.5})
\[H^0(C,\cO_C(-D+(q-2) F_{\infty})) \to \cO_{C,F_\infty\cap C}(-D+(q-2) F_{\infty})\]
is a basis as an $\cO_{C,F_\infty\cap C}$-module.
Note (cf. Definition \ref{specialcurve} and \eqref{keythm.eq3.1})
\[\begin{aligned} 
&\div_X(1+\tbeta)= Z_{\tbeta} - (q-2) F_\infty , \quad Z_{\tbeta}\cap C =F_\infty \cap C,\\
& (1+{\tbeta})|_{F_t} =1 \qfor t\in \TL\cup\{0\} \qaq b|_{F_\infty}=1.
\end{aligned}\]
In view of \eqref{keythm.eq1} $\div_X(b)$ and $\div_X(1+\tbeta)$ have no common component which intersects $C$.
Hence we may apply Lemma \ref{classgroup.lem2}(2) to $1+{\tbeta}$ and $b$ to get
\begin{equation*}
\underset{t\in \Liotaqo\backslash \TL}{\sum} 
\sym {1+{\tbeta}} {F_t}\;+\;\sym {1-\rho^{q-1}} {Z_{\tbeta}}=0 \;\in \WU .
\end{equation*}
Taking $q$ sufficiently large (cf. Remark \ref{keythm.claim1.rem}), 
Lemma \ref{keylemIIfinite} (where we take $e=q-2$ and $F=F_\infty$) implies 
$\sym {1-\rho^{q-1}} {Z_{\tbeta}}\in \hFXCC {D+C} $, which proves \eqref{keythmproof.eq3}.
\medbreak

To show \eqref{keythmproof.eq4}, put
\begin{equation}\label{keythmproof.eq4.5}
Z'=\div_X(1+\talpha)+ (q-1)F_0 \in \ZXCC.
\end{equation}
By Lemma \ref{classgroup.lem2}(1) we have $Z'\cap C = F_0\cap C$.
Locally at $F_0\cap C$, 
\begin{equation}\label{keythmproof.eq5}
Z'=\div_X(\sigma^{q-1} + \piD \gamma)\;\; (\gamma\in \cO_{X,F_0\cap C}), 
\end{equation}
where $\sigma=\rho^{-1}$ and $\piD$ is a local parameter of $D$ at $F_0\cap C$.
By \eqref{keythm.eq0.5} and \eqref{keythm.eq1},
\begin{equation}\label{keythmproof.eq5.5}
\div_X(\frac{b+\pi}{b}) =W_b-G_\infty -\underset{t\in \Liotaqo}{\sum} F_t.
\end{equation}
Since $G_\infty\cap  F_0 \cap C=W_b\cap  F_0\cap C=\varnothing$ by Claim~\ref{keythmproof.claim3}(1),
$\div_X(1+\talpha)$ and $\div_X(\frac{b+\pi}{b})$ have no common component which intersects $C$.
Hence we may apply Lemma \ref{classgroup.lem2}(2) to $1+\talpha$ and $\frac{b+\pi}{b}$ to get 
\begin{multline}\label{keythmproof.eq6}
\sym {1+\talpha} {W_b}\; - \underset{t\in \Liotaqo\backslash \TL}{\sum} \sym {1+\talpha} {F_t}\; \\
 - \;\sym {\frac{b+\pi}{b}} {Z'}\;+ (q-1)\sym {\frac{b+\pi}{b}} {F_0}=0\in \WU,
\end{multline}
where we used the fact 
$(1+{\talpha})_{|G_\infty} =1$ and $(1+{\talpha})_{|F_t} =1$ for $t\in \TL$ in view of \eqref{keythm.eq3.1}.
We claim
\begin{equation}\label{keythmproof.eq7}
\sym {1+\talpha} {W_b}\; - \underset{t\in \Liotaqo\backslash \TL}{\sum}\;
\underset{x\in F_t\cap C}{\sum}\sym{1+\talpha}{W_b,x}\;\in \FXCC {D+C}.
\end{equation}
Indeed, by Claim \ref{keythmproof.claim3}(1),
\[
\sym {1+\talpha} {W_b}=\underset{t\in \Liotaqo}{\sum}\;
\underset{x\in F_t\cap C}{\sum}\sym{1+\talpha}{W_b,x}.
\]
For $t\in \TL$ and $x\in F_t\cap C$, we have $x\not\in F_0$ since $F_0\cap F_t\cap C=\varnothing$.
In view of \eqref{keythm.eq3.1} this implies $\talpha_{|W_b}\in \cO_{W_b,x}(-D-F_t-G_\infty)$.
Since $\cO_{W_b,x}(-D-F_t-G_\infty)\subset \cO_{W_b,x}(-D-C)$ by Claim \ref{keythmproof.claim3}(3),
we get $\sym{1+\talpha}{W_b,x}\in \FXCC{D+C}$.
\medbreak

By \eqref{keythmproof.eq7} we are reduced to showing that the last two terms of \eqref{keythmproof.eq6} belong to $\FXCC {D+C}$. The assertion follows from the fact
\[
\frac{b+\pi}{b} = 1 - \frac{\pi \sigma^{q-1}}{1-\sigma^{q-1}}\quad (\sigma=\rho^{-1})
\]
and that we have in view of \eqref{keythmproof.eq5}, 
\[
(\frac{b+\pi}{b})_{|Z'} = 1 + \frac{\gamma \pi \piD}{1-\sigma^{q-1}}\;\in 1+\cO_{Z',Z'\cap C}(-D-C).
\]
This reduces the proof of \eqref{keythmproof.eq4} and that of Theorem \ref{keythm} to Lemma \ref{keylemIIfinite}.

\end{proof}


\section{Compatibility with ramification theory}\label{keythmproofII}
\bigskip
\def\FACCKS#1{F^{(#1)}W^{KS}(A,C)}
\def\WUKS{W^{KS}(U)}
\def\WUxKS{W^{KS}(U_x)}
\def\FXCCKS#1{F^{(#1)}W^{KS}(X,C)}
\def\FACCxKS#1{F^{(#1)}W^{KS}(A_x,C_x)}

In this section we prove Theorem \ref{keythm2}. We need some preliminaries.
\par

\subsection{Review of class class field theory for two-dimensional local rings.}\label{reviewCFTlocalring}
Let $(A,\fm_A)$ be an excellent regular henselian two-dimensional local domain with the quotient field $K$.
Assume $F=A/\fm_A$ is finite. Let $P$ be the set of prime ideals of height one in $A$.
For $\fp\in P$ let $\Ap$ be the henselization of $A$ at $\fp$ and 
$\Kp$ (resp. $\kp$) be the quotient (resp. residue) field of $\Ap$.
Let $C$ be a reduced effective Cartier divisor on $\Spec(A)$ and put $U=\Spec(A)-C$.
Let $P_C\subset P$ be the subset of $\fp$ lying on $C$. For $\lam\in P_C$, let $\Klam$ be the quotient field of 
the henselization of $A$ at $\lam$.

For an effective Cartier divisor $D$ on $\Spec(A)$ with $|D|=C$ ($|D|$ denotes the support of $D$),
we consider the subgroup of $H^1(U)=H^1(U,\qz)$:
\[
\FH U D =\Ker\big(H^1(U) \to \underset{\lam\in P_C}{\bigoplus}\; H^1(\Klam)/\FH \Klam \mlam \big).
\]
where $\mlam\in \bZ_{>0}$ for $\lam\in P_C$ is the multiplicity of $\lam$ in $D$ (see \ref{RTlocal} for the notation).
We introduce an idele class group which controls $\piab U$:
\begin{equation}\label{CAD0}
\WUKS:=\Coker\big(\K2 K \rmapo{\partial=(\partial_\fp,\partial_\lam)} \underset{\fp\in P-P_C}{\bigoplus}\;\kp^\times\;\oplus\; \underset{\lam\in P_C}{\bigoplus}\; \K2 {\Klam}\big),
\end{equation}
where $\partial_\fp$ for $\fp\not\in P_C$ is the tame symbol and $\partial_\lam$ for $\lam\in P_C$ is the map induced by $K\to \Klam$. We put
\begin{equation}\label{CAD}
\C^{KS}(A,D)= \WUKS/ \FACCKS D,
\end{equation}
where $\FACCKS D \subset \WUKS$ is the subgroup generated by the images of 
$V^{\mlam}\K2 {\Klam}\subset \K2 {\Klam}$ for $\lam\in P_C$.
By the reciprocity law for $A$ we have a canonical map (cf. \cite[1.9]{Sa}, \cite[(2.9)]{Sa2}, \cite[Ch.I]{Sa3})
\[
\Psi_U^{KS} : \FH U D \to \Hom(\C^{KS}(A,D),\qz)
\]
such that the following diagrams are commutative for $\fp\not\in P_C$ and $\lam\in P_C$:
\begin{equation}\label{CKS-CD}
\xymatrix{ 
\FH U D \; \ar[r]^{\hskip -30pt\Psi_U^{KS}} \ar[d]  &\Hom(\C^{KS}(A,D),\qz)\ar[d] \\
H^1(\kp)  \ar[r]^{\hskip -30pt\Psi_{\kp}} &\Hom(\kp^\times,\qz), \\}
\end{equation}
\[\xymatrix{ 
\FH U D \; \ar[r]^{\hskip -30pt\Psi_U^{KS}} \ar[d]  &\Hom(\C^{KS}(A,D),\qz)\ar[d] \\
\FH {\Klam} {\mlam}  \ar[r]^{\hskip -50pt\Psi_{\Klam}}
 &\Hom(\K2 {\Klam}/V^{\mlam}\K2 {\Klam},\qz) \\
}\]
where $\Psi_{\kp}$ (resp. $\Psi_{\Klam}$) is the map \eqref{recK} for the $1$-dimensional (resp.
$2$-dimensional) local field $\kp$ (resp. $\Klam$) (cf. \eqref{recK}).

\begin{rem}\label{rem.CKS}
Let $I_D\subset A$ be the ideal defining $D$. For $\alpha\in I_D$ and $\fp\in P-P_C$, the image in $\WUKS$ of 
$1+\alpha \mod \fp\in \kp^\times$ lies in $\FACCKS D$. Indeed, let $f\in A$ be such that $\fp=(f)$ and put
$\xi=\{1+\alpha,f\}\in \K2 K$. Then one easily sees
\[
\begin{aligned}
&\partial_\lam(\xi) \in V^{\mlam}\K2 {\Klam}\qfor \lam\in P_C,\\
&\partial_\fq(\xi) =  
\left.\left\{\begin{gathered}
 1+\alpha \mod \fp \\ 
 0 \\
\end{gathered}\right.\quad
\begin{aligned}
&\text{for $\fq=\fp$}\\
&\text{for $\fq\in P-P_C-\{\fp\}$}
\end{aligned}\right.
\end{aligned}
\] 
\end{rem}

\bigskip

Now we assume $D=\Spec(A/(\pi^m))$ ($m\in \bZ_{\geq 1}$), where $\pi\in A$ is such that $\lam=(\pi)\in P$ and that
$\Blam=A/(\pi)$ is regular. Let $\fmlam=\fm_A\Blam$ be the maximal ideal of $\Blam$.
Define
\[
\nu_A: \pi^{m-1} \Omega_A^1\otimes_A \fmlam^{-1} \to \C^{KS}(A,D)
\]
as the composite
\[
 \pi^{m-1} \Omega_A^1\otimes_A \fmlam^{-1} \hookrightarrow \pi^{m-1} \Omega_{\Alam}^1\otimes_{\Alam} \klam 
\rmapo {\rho^m_{\Klam}} \K2 {\Klam}/V^{m}\K2{\Klam} \to \C^{KS}(A,D),
\]
where $\rho^m_{\Klam}$ is the map from Lemma \ref{lem.grKM}.

\begin{lem}\label{lem.nu0}
The above map induces a map
\[
\nu_A: \pi^{m-1} \Omega_A^1\otimes_A \fmlam^{-1}\otimes_{\Blam} F \to \C^{KS}(A,D).
\]
\end{lem}
\begin{proof}
Choosing $f\in A$ such that $f\mod (\pi)\in \Blam$ is a generator of $\fmlam$, 
an element $\omega\in \pi^{m-1} \Omega_A^1\otimes_A \fmlam^{-1}$ is written as
\[ \omega = \frac{\pi^{m-1}}{f} (a d\pi + b df)\quad (a,b\in A).\]
The description of $\rho^m_{\Klam}$ in Lemma \ref{lem.grKM} shows that
$\rho^m_{\Klam}(\nu_A(f\omega))$ is the image of 
\[\gamma:=\{1+\pi^m a,\pi\} + \{1+\pi^{m-1}f b,f\}\;\in K_2(K)\]
under $\partial_\lam$ in \eqref{CAD0}. One easily check that the images of $\gamma$  
under $\partial_\fp$ for $\fp\in P-\{\lam\}$ in \eqref{CAD0} vanish.
This proves the desired assertion.
\end{proof}

\begin{lem}\label{lem.nu}
Assume $p\not=2$. Let $(\pi,f)$ be a system of regular parameters of $A$.
\begin{itemize}
\item[(1)]
The image of the composite
\[
\rsw_A:\FH U m \to \FH {\Klam} m \rmapo {\rsw_{\Klam}} \frac{1}{\pi^{m}}  \Omega_{\Alam}^1\otimes_{\Alam}\klam
\]
is contained in 
$\frac{1}{\pi^{m}} \Omega_{A}^1\otimes_{A}\Blam$.
The diagram
\[\xymatrix{ 
\FH U m \; \ar[r]^{\hskip -10pt -\rsw_A} \ar[d]^{\Psi_U^{KS}} 
&\frac{1}{\pi^{m}} \Omega_{A}^1\otimes_{A}\Blam \ar[r]
&\frac{1}{\pi^{m}} \Omega_{A}^1\otimes_{A} F \ar[ld]^{\tau_A}  \\
\big(\C^{KS}(A,D)\big)^\vee   \ar[r]^{\hskip -30pt(\nu_A)^\vee}
 &\hskip 20pt \big(\pi^{m-1}\Omega_{A}^1\otimes_{A}\fmlam^{-1}\otimes_{\Blam} F\big)^\vee, \\
}\]
is commutative, where $\tau_A$ is induced by the pairing
\begin{multline*}
\frac{1}{\pi^{m}} \Omega_{A}^1\otimes_{A}\Blam \times \pi^{m-1}\Omega_{A}^1\otimes_{A}\fmlam^{-1} \to
\pi^{-1}\Omega_{A}^2\otimes_{A}\fmlam^{-1} \rmapo {Res_\lam} \fmlam^{-1}\Omega^{1}_{\Blam} \\
\rmapo {Res_{\fmlam}} F=\Blam/\fmlam \rmapo{Tr_{F/\bF_p}} \bF_p\simeq\pz.
\end{multline*}
Moreover $\tau_A$ is an isomorphism.
\item[(2)]
Putting $\fp_1=(f),\; \fp_2=(\pi+f)\in P$, we have
\[
\xi=\nu_A(\frac{1}{f}(\alpha d\pi +\beta df)) \qfor \alpha,\beta\in (\pi^{m-1})
\]
is the image in $\C^{KS}(A,D)$ of 
\[
\eta=\sym {1+\beta}{k(\fp_1)} - \sym {1+\alpha}{k(\fp_1)} + \sym {1+\alpha}{k(\fp_2)}\;\in \underset{\fp\in P-P_C}{\bigoplus}\;\kp^\times.
\]
\end{itemize}
\end{lem}
\begin{proof}
The first (resp. second) assertion of (1) follows from \cite[Prop.4.2.1]{Ma} (resp. \eqref{commutativitypairing}).
Noting that $f$ generates $\fmlam$, $\tau_A$ is induced by the pairing
\[\langle \;,\;\rangle: 
\frac{1}{\pi^{m}} \Omega_{A}^1\otimes_{A} F \;\times\; 
\frac{\pi^{m-1}}{f}\Omega_{A}^1\otimes_{A} F \to \bF_p
\]
such that for $\alpha,\beta,\gamma,\delta\in F$
\[ \langle \alpha \frac{d\pi}{\pi^{m}} + \beta \frac{d f}{\pi^{m}}\;,
\gamma\frac{\pi^{m-1}d\pi}{f} + \delta\frac{\pi^{m-1}d f}{f} \;\rangle = 
\pm Tr_{F/\bF_p}(\alpha\delta- \beta\gamma).\]
This shows that the pairing is non-degenerate so that $\tau_A$ is an isomorphism.

To show (2), note
\[ \frac{1}{f}(\alpha d\pi +\beta df) = \frac{\alpha}{\pi+ f} d(\pi+f) +\frac{\beta-\alpha}{f} df
\;\;\in\; \pi^{m-1} \Omega_A^1\otimes_A \fmlam^{-1}.\]
Hence its image under the map
\[\pi^{m-1} \Omega_A^1\otimes_A \fmlam^{-1} \hookrightarrow \pi^{m-1} \Omega_{\Alam}^1\otimes_{\Alam} \klam 
\rmapo {\rho^m_{\Klam}} \K2 {\Klam}/V^{m}\K2{\Klam} \]
is 
\[\{1+\alpha, \pi+f\} + \{1+(\beta-\alpha),f\} \equiv \{1+\alpha, \pi+f\} + \{1+\beta,f\} -  \{1+\alpha,f\}.\]
This implies the desired assertion.
\end{proof}
\bigskip

\subsection{Review of class field theory of Kato-Saito.}\label{reviewCFTKS} 
Now we come back to the global setting of \S \ref{keytheorem} where $k$ is assumed finite.
We also assume $|D|=C$.
For a closed point $x\in C$, let $A_x$ be the henselization of $\cO_{X,x}$, and $U_x$ (resp. $C_x$, resp. $D_x$) are the base change of $U=X-C$ (resp. $C$, resp. $D$) via $\Spec(A_x) \to X$.
Let $K_x$ be the fraction field of $A_x$ and $P_x$ be the set of prime ideals of height one in $A_x$.
For $\lam\in I$, let $\Klam$ be the fraction field of the henselization of $\cO_{X,\lam}$.
We introduce the idele class group for $U$:
\[
\WUKS=\Coker\big(\underset{Z\subset X}{\bigoplus}\; k(Z)^\times \oplus \underset{\lam\in I}{\bigoplus}\; \K2 {\Klam}
\; \rmapo{\partial}\; 
Z_0(U) \oplus \underset{x\in C}{\bigoplus}\; \WUxKS\big),
\]
where $\WUxKS$ is defined as in \eqref{CAD0} for $U_x$, and $Z$ ranges over integral curves on $X$ not contained in $C$, and 
$\partial$ is induced by the maps:
\[
\partial_1: k(Z)^\times\to Z_0(U),\quad \partial_2: k(Z)^\times \to \underset{x\in C}{\bigoplus}\; \WUxKS,
\]
\[
\partial_3: \K2 {\Klam}\to Z_0(U),\quad \partial_4: \K2 {\Klam} \to \underset{x\in C}{\bigoplus}\; \WUxKS,
\]
where $\partial_1$ is the divisor map, and $\partial_2$ is induced by the identification (cf. Definition \ref{classgroup.def}):
\begin{equation}\label{k(Z)infty}
k(Z)_{\infty}^\times = \underset{x\in Z\cap C}{\bigoplus}\;\underset{\fp\in P_{x,Z}}{\bigoplus} \kp^\times,
\end{equation}
where $P_{x,Z}$ denotes the set of $\fp\in P_x$ lying over $Z$, and $\partial_3$ is the zero map, and $\partial_4$ is
induced by the natural map 
\[\Klam \to \underset{x\in C}{\prod}\;\underset{\xi\in P_{x,\lam}}{\prod} K_\xi,\]
where $P_{x,\lam}$ is the set of $\xi\in P_x$ which lies over $\Clam$ (cf. \eqref{eqD1})
and $K_\xi$ is the fraction field of the henselization of $A_x$ at $\xi$. 
We then put
\[
\C^{KS}(X,D) = \WUKS /\FXCCKS D,
\]
where $\FXCCKS D\subset \WUKS$ is the subgroup generated by the image of $\FACCxKS D$ for $x\in C$.
We remark that the argument of \cite[(1.6)]{KS2} shows an isomorphism (cf. \eqref{CKSU})
\[
\C^{KS}(X,D) \isom  H^d(X_{\rm Nis} , \mathcal K^M_d(X,D) ),
\]
where $ \mathcal K^M_d(X,D)$ is the relative Milnor $K$-sheaf introduced in \cite[(1.3)]{KS2}.
By the class field theory developed in \cite[Ch.II~\S3 Th.1]{KS} and \cite[Th.9.1]{KS2}, we have a canonical isomorphism
\begin{equation}\label{recCKS}
\Psi_U^{KS}: \FH U D \isom \Hom(\C^{KS}(X,D),\qz)
\end{equation}
which fits into the commutative diagram for any closed point $x\in C$,
\begin{equation}\label{CD1}
\xymatrix{ 
\FH U D \; \ar[r]^{\Psi_U^{KS}} \ar[d]  & \C^{KS}(X,D)^\vee\ar[d] \\
\FH {U_x} {D_x}  \ar[r]^{\Psi_{U_x}^{KS}} & \C_x^{KS}(X,D)^\vee \\
}
\end{equation}
In view of \eqref{k(Z)infty}, there is a natural map
\begin{equation}\label{epsionX}
\epsilon_U: \WU =\Coker\Big(\underset{Z\subset X}{\bigoplus} \; k(Z)^\times \to 
\underset{Z\subset X}{\bigoplus}\; k(Z)_\infty^\times\;\oplus\; Z_0(U)\Big) \;\to \WUKS\;.
\end{equation}
By Remark \ref{rem.CKS} it induces a canonical map
\begin{equation*}
\epsilon_{X,D}: \WU/\FXCC {D} \to \C^{KS}(X,D).
\end{equation*}
The commutativity of the diagrams \eqref{CKS-CD} and \eqref{CD1} implies that the diagram
\begin{equation}\label{CD1b}
\xymatrix{ 
\FH U D \; \ar[r]^{\Psi_U^{KS}}  \ar[d]^{\Psi_{X,D}}  & \C^{KS}(X,D)^\vee\ar[d]^{(\epsilon_{X,D})^\vee} \\
\CXD^\vee \ar[r] & \big(\WU/\FXCC {D}\big)^\vee \\
}\end{equation}
commutes, where we recall $\CXD=\WU/\hFXCC D$ (cf. Definition~\ref{def.CXD}).
\medbreak

\subsection{Proof of Theorem \ref{keythm2}.}
By Lemma \ref{lem.nu}(1) the image of composite map
\[\FH U D \to \FH {\Klam} {\mlam} \rmapo {\rsw_{\Klam}} \Omega^1_X(D)\otimes_{\cO_X} k(\Clam)\]
is contained in 
$H^0(\Clam^o, \Omega^1_X(D)\otimes_{\cO_X} \cO_{\Clam})$,
where $\Clam^o$ denotes a dense open subset of $\Clam$ where $C$ is smooth. Since the natural map
\[ H^0(\Clam^o, \Omega^1_X(D)\otimes_{\cO_X} \cO_{\Clam})\to \underset{x\in \Clam^o}{\prod}\; 
\Omega^1_X(D)\otimes_{\cO_X} \k(x)\]
is injective, we are reduced to showing the commutativity of the diagram
\[
\xymatrix{ 
\FH U D \; \ar[r] \ar[d] & H^0(\Clam^o, \Omega^1_X(D)\otimes_{\cO_X} \cO_{\Clam}) \ar[r] 
&\Omega^1_X(D)\otimes_{\cO_X} \k(x)\\
\C(X,D)^\vee \ar[r]^{\hskip -35pt \rswM_X} & H^0(C,\Omega^1_X(D+ \Xi)\otimes_{\cO_X}\cO_C) \ar[ru] \\
}\]
for each $x\in \Clam^o$. This follows from the commutative diagram \eqref{CD1b} and 
the following commutative diagram by noting that $\tau_{A_x}$ is an isomorphism 
by Lemma \ref{lem.nu}(1):
\[
\xymatrix{ 
\FH U D \; \ar[r] \ar[d]^{\Psi_U^{KS}\hskip 30pt \boxed{I}}  & \FH {U_x} {D_x} 
\ar[r]^{-\rsw_{U_x}} \ar[d]^{\Psi_{U_x}^{KS} \hskip 30pt\boxed{III}}  
& \Omega_X^1(D)\otimes_{\cO_X}  \cO_{C,x} \ar[d] \\
 \C^{KS}(X,D)^\vee \ar[r]\ar[d]^{(\epsilon_{X,D})^\vee\hskip 20pt \boxed{II}} &  
\C_x^{KS}(X,D)^\vee\ar[d]^{(\nu_{A_x})^\vee}
& \Omega_X^1(D)\otimes_{\cO_X}  \k(x) \ar[ld]^{\hskip 30pt\tau_{A_x}} \\
 \big(\WU/\FXCC {D}\big)^\vee\; \ar[r]_{\hskip 30pt(\mu_x)^\vee} & \big(\Lam_x \big)^\vee\\
\CXD^\vee\ar[u]_{\hskip 130pt \boxed{IV}} \ar[rr]_{\rswM_X}& & 
H^0(C,\Omega^1_X(D+ \Xi)\otimes_{\cO_X}\cO_C)\ar[uu] \\
}\] 
Here 
$\Lam_x= \Omega_X^1(-D+C)\otimes_{\cO_X}\cO_{C,x}(x)\otimes_{\cO_{C,x}}\k(x)$ and 
$\rsw_{U_x}$ arises from Lemma \ref{lem.nu}(1).
The commutativity of $\boxed{I}$ comes from \eqref{CD1}, that of $\boxed{II}$ from Lemma \ref{lem.nu}(2)
and \eqref{mufpi}, that of $\boxed{III}$ from Lemma \ref{lem.nu}(1), and 
that of $\boxed{IV}$ from the definition of $\rswM_X$ and the commutativity of
\[
\xymatrix{ 
H^0(C,\Omega^1_X(D+ \Xi)\otimes_{\cO_X}\cO_C) \ar[r] \ar[d] & \Omega^1_X(D)\otimes_{\cO_X}\k(x) \ar[d]^{\tau_{A_x}}\\
H^1(C,\Omega^1_X(-D+C -\Xi)\otimes_{\cO_X}\cO_C)^\vee \ar[r]^{\hskip 60pt(\iota_x)^\vee} & \big(\Lambda_x \big)^\vee \\
}\]
where the left vertical map is induced by \eqref{duality1} and $\iota_x$ comes from \eqref{keythmproof.eq0}. 
This completes the proof of Theorem \ref{keythm2}.

\begin{rem}
In case $C$ is a simple normal crossing divisor on $X$, one can show the map \eqref{epsionX} induces a map
\[\epsilon_{X,D}: \CXD=\WU/\hFXCC D  \to \C^{KS}(X,D).\]
In view of \eqref{recCKS}, the main result Corollary \ref{CFT.coro} is equivalent to that
the map is an isomorphism.
\end{rem}



\section{Key Lemmas}\label{keylemmas}
\bigskip

In order to finish the proof of the main results of the paper, it remains to prove three key lemmas
(Lemma~\ref{movinglemfinite}, Lemma~\ref{lem.3term} and Lemma~\ref{keylemIIfinite}).
In this section we restate them over a general perfect field and give a leitfaden for the proofs,
which occupies the remaining sections \S\ref{keylemproofI} through \S\ref{keylemproofV}.
In the rest of the paper, $k$ is assumed only perfect (not necessarily finite).
\medbreak

Let the notation be as in \S\ref{classgroup}.
Let $(X,C)$ be in $\cC$ (see Definition \ref{def.cC}) and $\{\Clam\}_{\lam\in I}$ be the set of prime components of $C$.
Fix a Cartier divisor 
\begin{equation}\label{eqDkeylemmas0}
\displaystyle{D=\underset{\lambda\in I}{\sum} \mlam \Clam}
\qwith \mlam\geq 1.
\end{equation}

The first key lemma is a relation among three symbols (see Lemma~\ref{lem.3term}). 

\begin{lem}[three term relation]\label{lem.3term2}
Let $x$ be a regular closed point of $C$.
Let $F,Z_1,Z_2\in \ZXCC$ be such that $F\Cap C$ at $x$, and $F\Cap Z_i$ and $Z_i\Cap C$ at $x$ for $i=1,2$.
Let $(\pi,f)$ be a system of regular parameters such that locally at $x$, 
\begin{equation*}
F=\div_X(f),\; Z_1=\div_X(\pi- v_1f),\; Z_2=\div_X(\pi- v_2f),\; C=\div_X(\pi),
\end{equation*}
where $v_1,v_2\in \cO_{X,x}^\times$. For $\alpha\in \cO_{X,x}(-D)$, we have
\[
\sym{1 + (v_1-v_2)\alpha}{F,x} + \sym{1 + v_1\alpha}{Z_1,x} -\sym{1+v_2\alpha}{Z_2,x}\;\in 
\FXCC {D+C}.
\]
\end{lem}
\bigskip

The second key lemma refines Lemma~\ref{keylemIIfinite} to the case over a perfect field. 

\begin{lem}[increasing order]\label{keylemII}
Let $Z_a$ with $a\in \cP_{D,e}(F)$ be as Definition \ref{specialcurve} and take $x\in Z_a\cap C$. 
Let $\lam\in I$ be such that $x\in \Clam$. Assume 
\[
\begin{aligned}
&H^1(X,\cO_X(-2D-C+ (e-1)F))=H^1(C,\cO_C(-2D+F))=0,\\
\end{aligned}
\leqno{(*)}
\]
Assume $p\not=2$ and $e\geq \mlam(p^n-1)$ for a given integer $n>0$. Then we have 
\begin{equation}\label{keylemII.eq}
\sym {1+{f^{e+1}}\cO_{Z_a,F\cap C}} {Z_a,x} \subset \hFXCC {D+C} + p^n \hWU.
\end{equation}
\end{lem}

Lemma~\ref{keylemIIfinite} follows from Lemma~\ref{keylemII} and Corollary \ref{tameCFT.cor}
which implies $p^n \hWU \subset \hFXCC {D+C}$ for some $n>0$ if $k$ is finite.
\bigskip

The last key lemma concerns moving elements of $\WU$ to symbols on curves transversal to $C$.
It refines Lemma~\ref{movinglemfinite} to the case over a perfect field. 
Take any dense open subset $V\subset X$ containing the generic points of $C$. Recall Definition \ref{filtration.def0}.

\begin{lem}[moving]\label{movinglem}
Assume $p\not=2$. For any integers $n,N>0$, we have 
\[
\hFXCC D  \subset \FXCCapV D \;+\;  \hFXCC {D+N\cdot C} + p^n\hWU.
\]
\end{lem}

Lemma~\ref{movinglemfinite} follows from Lemma~\ref{movinglem} and Corollary \ref{tameCFT.cor} as above.
\bigskip

In \S\ref{keylemproofI} through \S\ref{keylemproofV} we prove the following Lemmas and implications using
Lemma \ref{classgroup.lem2}.

In \S\ref{keylemproofI} we prove Lemma \ref{keylem3D-2} and prove Lemma \ref{lem.3term2} using
Lemma \ref{keylem3D-2}.

Lemma \ref{keylemII} is a direct consequence of Lemma \ref{keylem.cor} consisting of parts
(1) and (2). 

In \S\ref{keylemproofII} we prove Lemma \ref{keylemma} and state Lemma \ref{keylemma-2} consisting of parts
(1) and (2). Using Lemma \ref{keylemma}, we prove an implication Lemma \ref{keylemma-2}(1) (resp. (2))$\;\Rightarrow\;$Lemma \ref{keylem.cor}(1) (resp. (2)). We also prove Lemma \ref{keylemma-2}(1). This completes the proof of Lemma \ref{keylem.cor}(1) and reduces Lemma \ref{keylemII} to Lemma \ref{keylemma-2}(2).

In \S\ref{keylemproofIII} we state Lemma \ref{keylem2} consisting of parts (1) and (2) and prove
an implication Lemma \ref{keylem.cor}(1) (resp. (2))$\;\Rightarrow\;$Lemma \ref{keylem2}(1) (resp. (2)). 
This completes the proof of Lemma \ref{keylem2}(1) using Lemma \ref{keylem.cor}(1) proved in \S\ref{keylemproofII}.

In \S\ref{keylemproofIV} we prove an implication Lemma \ref{keylem2}(1)$\;\Rightarrow\;$Lemma \ref{keylem3D-3}.
Lemma \ref{keylemma-2}(2) is a direct consequence of Lemma \ref{keylem3D-3}.
Thus Lemma \ref{keylem.cor}(2) and Lemma \ref{keylem2}(2) are proved by the implications shown 
in \S\ref{keylemproofII} and \S\ref{keylemproofIII}.

Finally, in \S\ref{keylemproofIV} we prove Lemma \ref{movinglem} using Lemma \ref{keylem2}(2) and
Lemma \ref{keylem3D-2}.


\section{Proof of Key Lemma I}\label{keylemproofI}

Let the notation be as in \S\ref{keylemmas}.
In this section we will prove Lemma \ref{lem.3term2}. 

\begin{lem}\label{keylem3D-2}
Take a regular closed point $x\in C$ and $Z_1,Z_2\in \ZXCC$ such that $x\in Z_1\cap Z_2$ and 
that $Z_i\Cap C$ at $x$ for $i=1,2$ (cf. Definition \ref{curve.def}).
Assume $(Z_1,Z_2)_x \geq e+1$ for an integer $e\geq 1$.
Then
\[
\sym {1+\alpha} {Z_1,x} - \sym {1+\alpha} {Z_2,x} \in \FXCC {D+eC}
\qfor \alpha\in \cO_{X,x}(-D).
\]
\end{lem}

The following lemma will be used several times in this paper.

\begin{lem}\label{preliminarylem-1}
Let $X$ be a noetherian scheme, $E$ be an effective Cartier divisor on $X$, and $A$ be an effective Cartier divisor on $E$. Let $\cF$ be a coherent $\cO_X$-module such that
\[ H^1(X,\cF\otimes\cO_X(-E))= H^1(E,\cF|_{E} \otimes\cO_E(-A))=0.\]
Then the restriction map
$r_A: H^0(X,\cF) \to H^0(A,\cF|_{A})$
is surjective.
\end{lem}
\begin{proof}
The map $r_A$ factors as
\[H^0(X,\cF) \to H^0(E,\cF|_{E}) \to H^0(A,\cF|_{A}).\]
The first (resp. second) map is surjective due to the first (resp. second) vanishing of the cohomology group.
This proves the lemma.
\end{proof}
\medbreak

Now we start the proof of Lemma \ref{keylem3D-2}.
For a later purpose we assume only $e\geq 0$ (not necessarily $e\geq 1$).
For an integer $d>0$, let $\cL(d)=|dH|$ be as Definition \ref{pencil.def}. 
Take $d$ sufficiently large so that we can choose $F$ in $\cL(d)$ satisfying the conditions: 
\begin{itemize}
\item[$(\flat 1)$]
$F\Cap C$, $x\in F$, $F\cap C\cap (Z_1\cup Z_2- x)=\varnothing$, and $F\Cap Z_i$ at $x$ for $i=1,2$.
\item[$(\flat 2)$]
Let $E=(e+1) C$ and $A\subset E$ be the part of $(e+1) F|_{E}$ supported at $x$. 
Then 
\[H^1(X,\cO_X(-D-2E+(e+1)F))=H^1(E,\cO_X(-D-E+(e+1)F)\otimes\cO_E(-A)) =0.\]
\end{itemize}
Then take $d'$ sufficiently large (compared with $d$ chosen above)
so that we can choose integral hypersurface section $H'$ in $\cL(d')$ satisfying
\begin{itemize}
\item[$(\clubsuit 1)$]
$(F\cap C)-x\subset H'$ and $x\not\in H'$.
\item[$(\clubsuit 2)$]
$H'\Cap C$ at $(F\cap C)-x$ and $H'\Cap F$ at $(F\cap C)-x$. 
\end{itemize}
For $i=1,2$ we put 
\begin{equation}\label{keylem3D-2.claim4.eq-1}
F_i=Z_i+H' \; \in \ZXCC.
\end{equation}
Choose $\pi,\piD,f\in \cO_{X,F\cap C}$ such that locally at $F\cap C$,
\[
C=\div_X(\pi),\quad D=\div_X(\piD),\quad F=\div_X(f)
\]
so that $(\pi,f)$ is a system of regular parameters in $\cO_{X,F\cap C}$.
Let $\fm_{F\cap C}=(f,\pi)$ denote the radical of $\cO_{X,F\cap C}$.



\begin{claim}\label{keylem3D-2.claim4}
There exist $u_1,u_2\in \cO_{X,F\cap C}^\times$ such that $F_i = \div_X(\pi- u_i f)$ locally at $F\cap C$
and that 
\begin{equation}\label{keylem3D-2.claim4.eq0}
u_1-u_2\in (\pi-u_2f,f^e)\subset \fm_{F\cap C},\;\;
u_1-u_2\in \fm_y \qfor y\in (F\cap C)-x,
\end{equation}
where $\fm_y$ is the maximal ideal of $\cO_{X,y}$
(the second condition of \eqref{keylem3D-2.claim4.eq0} is contained in the first except the case $e=0$).
\end{claim}
\begin{proof}
Let $g_i\in \cO_{X,F\cap C}$ be a local equation of $F_i$ at $F\cap C$.
By ($\clubsuit 1$) we have $F\cap C\subset F_i$ so that $g_i\in \fm_{F\cap C}$ and we can write
$g_i=a_i\pi+b_i f$ for some $a_i,b_i\in \cO_{X,F\cap C}$. By ($\flat 1$) and ($\clubsuit 2$), 
$F_i\Cap F$ and $F_i\Cap C$ at any $y\in F\cap C$. 
It implies $a_i,b_i\in \cO_{X,F\cap C}^\times$ proving the first part of the claim.
By \eqref{keylem3D-2.claim4.eq-1} we have
\begin{equation}\label{keylem3D-2.claim4.eq1}
\cO_{X,F\cap C}/(\pi-u_2 f)=\cO_{F_2,F\cap C}= \cO_{Z_2,x}\times \cO_{H',F\cap C-x}.
\end{equation}
By $(\flat 1)$, locally at $F\cap C -x$. we have $H'=F_1=F_2$ so that
\[(\pi- u_1 f)-(\pi- u_2 f)=(u_1-u_2)f \equiv 0\in \cO_{H',F\cap C-x}.\]
By $(\clubsuit 2)$, $f$ is not a zero divisor in $\cO_{H',F\cap C-x}$ so we get  
\begin{equation}\label{keylem3D-2.claim4.eq1.2}
(u_1-u_2)|_{H'}=0\in \cO_{H',F\cap C-x}.
\end{equation}
This implies the second condition of \eqref{keylem3D-2.claim4.eq0}.
By the assumption of Lemma \ref{keylem3D-2} 
\begin{equation*}\label{keylem3D-2.claim4.eq1.5}
(Z_1,Z_2)_x = {\rm length}_{\cO_{Z_2,x}}\big(\cO_{Z_2,x}/((u_1-u_2)f)\big) \geq e+ 1,
\end{equation*}
which implies $(u_1-u_2)|_{Z_2}\in (f^e)\cO_{Z_2,x}$ noting that
$f$ generates the maximal ideal of $\cO_{Z_2,x}$ since $Z_2\Cap F$ at $x$.
In view of \eqref{keylem3D-2.claim4.eq1} and \eqref{keylem3D-2.claim4.eq1.2}, we get
\[
(u_1-u_2)|_{F_2}\in (f^e)\cO_{F_2,F\cap C},
\]
which proves the first condition of \eqref{keylem3D-2.claim4.eq0}. 
\end{proof}

\begin{claim}\label{keylem3D-2.claim0}
Let the notation be as Claim \ref{keylem3D-2.claim4}.
Take any $\vartheta\in \cO_{X,x}$. For $e\geq 1$ 
\[
\sym {1+u_1^{e+1} \vartheta\piD} {Z_1,x} - \sym {1+u_2^{e+1} \vartheta\piD} {Z_2,x} \in \FXCC {D+eC}.
\]
For $e=0$ we have
\[
\sym {1+(u_1-u_2) \vartheta\piD} {F,x} + \sym {1+u_1 \vartheta\piD} {Z_1,x} - \sym {1+u_2\vartheta\piD} {Z_2,x} \in \FXCC {D+C}.
\]
\end{claim}

By Claim \ref{keylem3D-2.claim4}, we have
\[(u_1-u_2)_{|Z_i}\in (f^e)\cO_{Z_i,x}=\cO_{Z_i,x}(-eC) \qfor i=1,2.\]
Hence the first assertion of Claim \ref{keylem3D-2.claim0} implies Lemma \ref{keylem3D-2} 
noting $u_1,u_2\in \cO_{X,x}^\times$.
\medbreak

We now prove Claim \ref{keylem3D-2.claim0}.
Putting
\[\cF=\cO_{X}(-D-(e+1)C+(e+1)F),\]
we have an isomorphism
\[
\hskip -100pt
\mu: 
\cF\otimes \frac{\cO_{X,x}}{(\pi^{e+1},f^{e+1})} 
\simeq \cO_{X,x}/(\pi^{e+1},f^{e+1})
\]
\[
c\cdot \piD(\frac{\pi}{f})^{e+1} \to c \mod (\pi^{e+1},f^{e+1})
\quad (c\in \cO_{X,x}).
\]
By the assumption $(\flat 2)$ Lemma \ref{preliminarylem-1} implies that 
the natural map
\[
\iota: H^0(X,\cF) \to \cF\otimes \cO_{X,x}/(\pi^{e+1},f^{e+1})
\]
is surjective so that we can find
\begin{equation}\label{keylem3D-2.eq0}
a\in H^0(X,\cO_{X}(-D-(e+1)C+(e+1)F))
\end{equation}
such that
\[
\mu(\iota(a))\equiv \vartheta\; \in\; \cO_{X,x}/(\pi^{e+1},f^{e+1}).
\]
It implies 
\begin{equation}\label{keylem3D-2.eq1}
a=\gamma \piD(\frac{\pi}{f})^{e+1}\qwith \gamma\in \cO_{X,F\cap C}
\text{  such that } \gamma-\vartheta\in (\pi^{e+1},f^{e+1})\cO_{X,x}
\end{equation}
Put
\begin{equation}\label{keylem3D-2.eq1.5}
Z=\div_X(1+a) + (e+1) F .
\end{equation}
By Lemma \ref{classgroup.lem2}(1) $Z\in \ZXCC$ and $Z\cap C=F\cap C$.
\eqref{keylem3D-2.eq1} implies  
\begin{equation}\label{keylem3D-2.eq2}
Z=\div_X(f^{e+1}+\piD\pi^{e+1} \cdot \gamma)\;\;\text{locally at $F\cap C$}.
\end{equation}
Put
$\displaystyle{b =\frac{\pi - u_1 f}{\pi - u_2 f}\;\in k(X)^\times.}$
By Claim \ref{keylem3D-2.claim4} we have
\[\div_X(b) =F_1-F_2 +G = Z_1-Z_2 + G,\]
where $G$ is a divisor with $G\cap F\cap C=\varnothing.$
Noting  $Z\cap C=F\cap C$, we may apply Lemma \ref{classgroup.lem2}(3) to $1+a$ and $b$ to get 
\begin{equation}\label{keylem3D-2.eq3}
\sym {1+a} {Z_1} -\sym {1+a} {Z_2} - \sym {b} {Z} + (e+1)\sym {b} {F} \;\in \FXCC {D+(e+1)C} .
\end{equation}
Since $F\cap C\cap (Z_1\cup Z_2 -x )=\varnothing$ by $(\flat 1)$, we get for $i=1,2$ 
\[a|_{Z_i}\in  \cO_{Z_i,y}(-D-(e+1)C)) \qfor y\in (Z_i\cap C) -x.\]
Noting $\pi= u_i f$ in $\cO_{Z_i,x}$ (cf. Claim \ref{keylem3D-2.claim4}) and that
$Z_i\Cap F$ and $Z_i\Cap C$ at $x$, \eqref{keylem3D-2.eq1} implies
\[
a|_{Z_i} =   u_i^{e+1}\gamma\piD
\equiv   u_i^{e+1}\vartheta\piD\; \in \frac{\cO_{Z_i,x}(-D)}{\cO_{Z_i,x}(-D-(e+1)C)}
\quad (i=1,2).
\]
Hence \eqref{keylem3D-2.eq3} implies
\[
 \sym {1+u_1^{e+1}\vartheta\piD} {Z_1,x} - \sym {1+u_2^{e+1}\vartheta\piD} {Z_2,x} 
 - \sym {b} {Z} + (e+1)\sym {b} {F} \;\in  \FXCC {D+(e+1)C}.
\]
Hence Claim \ref{keylem3D-2.claim0} follows from the following claims:
\begin{equation}\label{keylem3D-2.eq4}
 \sym {b} {Z}-(e+1) \sym {b} {F}  \;\in \;  \FXCC {D+eC} \qfor e\geq 1,
\end{equation}
\begin{equation}\label{keylem3D-2.eq4.5}
\sym {b} {Z} - \sym {b} {F} - \sym {1+(u_1-u_2) \vartheta\piD} {F,x} \;\in \;  \FXCC {D+C} \qfor e=0,
\end{equation}


\begin{claim}\label{keylem3D-2.claim3}
There exists $b'\in \cO_{X,F\cap C}^\times$ such that $b'|_{Z}=b|_{Z}\in \cO_{Z,F\cap C}$ and that 
\begin{equation}\label{keylem3D-2.claim3.eq1}
(b'/b)|_{F} \in 1+ \cO_{F,F\cap C}(-D-eC) \qfor e\geq1,
\end{equation}
and that for $e=0$, 
\begin{equation}\label{keylem3D-2.claim3.eq2}
(b'/b)|_{F} \in 1+ \cO_{F,y}(-D-C) \qfor y\in (F\cap C)-x,
\end{equation}
\begin{equation}\label{keylem3D-2.claim3.eq3}
(b'/b)|_{F}  \equiv 1 + (u_1-u_2)\vartheta\piD \in \frac{\cO_{F,x}(-D)}{\cO_{F,x}(-D-C)}.
\end{equation}
\end{claim}

First we finish the proof of Claim \ref{keylem3D-2.claim0} (and hence that of Lemma \ref{keylem3D-2})
assuming the above claim. Since $b'\in \cO_{X,F\cap C}^\times$, we may apply Lemma \ref{classgroup.lem2}(4) 
to $1+a$ and $b'$ to get 
\[\sym {b'} {Z} -(e+1) \sym {b'} {F} \;\in \;  \FXCC {D+(e+1)C}.\]
We have $\sym {b'} {Z}= \sym {b} {Z}$ since $b'|_{Z}=b|_{Z}\in \cO_{Z,F\cap C}$.
If $e\geq 1$, \eqref{keylem3D-2.claim3.eq1} implies
\[\sym {b'} {F} -\sym {b} {F} \;\in \;  \FXCC {D+eC},\]
which implies \eqref{keylem3D-2.eq4}. 
If $e=0$, \eqref{keylem3D-2.claim3.eq2} implies
$\sym {b'} {F,y} -\sym {b} {F,y} \;\in \;  \FXCC {D+C}$ for $y\in (F\cap C) -x$, and 
\eqref{keylem3D-2.claim3.eq3} implies
\[\sym {b'} {F,x} -\sym {b} {F,x} - \sym {1 + (u_1-u_2) \vartheta\piD} {F,x}  \;\in \;  \FXCC {D+C},\]
which implies \eqref{keylem3D-2.eq4.5}. This completes the proof of Claim \ref{keylem3D-2.claim0}.

\medbreak\noindent
{\it Proof of Claim \ref{keylem3D-2.claim3}:}\:
For $e\geq 1$ Claim \ref{keylem3D-2.claim4} implies 
\begin{equation}\label{keylem3D-2.claim3.eq4}
u_2-u_1=\eta(\pi-u_2f) + \beta f^e \qwith \eta,\beta\in \cO_{X,F\cap C}.
\end{equation}
For $e=0$ we put $\beta=u_2-u_1$ and $\eta=0$. Then
\[
b=\frac{\pi - u_1 f}{\pi - u_2 f}= 1+ \frac{(u_2-u_1)f}{\pi - u_2 f}= 1+\eta f + \frac{\beta f^{e+1}}{\pi - u_2 f}.
\]
\eqref{keylem3D-2.eq2} implies $f^{e+1}|_{Z}= -\gamma\piD\pi^{e+1}\in \cO_{Z,F\cap C}$ so that 
\[
\big(\frac{\beta f^{e+1}}{\pi - u_2 f}\big)|_{Z}=
\frac{-\beta \gamma\piD\pi^{e+1}\delta}{\pi^{e+1} - u_2^{e+1} f^{e+1}}=\theta|_{Z}\;\in \cO_{Z,F\cap C},
\]
where
\[
\delta=\frac{\pi^{e+1} - u_2^{e+1} f^{e+1}}{\pi - u_2 f},\;
\theta=\frac{-\beta \gamma\piD\delta}{1+\gamma u_2^{e+1} \piD} \;\in \cO_{X,F\cap C}.
\]
Thus, taking $b'=1+ \eta f+ \theta\in \cO_{X,F\cap C}$, we get $b|_{Z}=b'|_{Z}\in \cO_{Z,F\cap C}$.
On the other hand, we have $b|_{F}=1$ so that $(b'/b)|_{F}= 1+ \theta\in \cO_{F,F\cap C}$.
We have $\delta|_{F}=\pi^e$ so that $\theta|_{F}\in \piD\pi^{e}\cO_{F,F\cap C}$, which proves \eqref{keylem3D-2.claim3.eq1}.
Assume $e=0$. We have $\delta|_{F}=1$. Hence
\begin{equation}\label{keylem3D-2.claim3.eq6}
1+\theta|_{F}=1+\frac{(u_1-u_2) \gamma\piD }{1+\gamma u_2\piD}\; \in \cO_{F,F\cap C}.
\end{equation}
By \eqref{keylem3D-2.eq1} and Claim \ref{keylem3D-2.claim4}, we have
\[\gamma|_{F}-\vartheta|_{F}\in \pi\cO_{F,x},\qaq
(u_1-u_2)|_{F}\in \pi\cO_{F,y} \qfor y\in (F\cap C)-x.\]
Hence \eqref{keylem3D-2.claim3.eq6} implies \eqref{keylem3D-2.claim3.eq2} and \eqref{keylem3D-2.claim3.eq3},
and Claim \ref{keylem3D-2.claim3} is proved.
\bigskip

Finally we prove Lemma \ref{lem.3term2}. 
In case $(Z_1,Z_2)_x\geq 2$,  we have $v_1-v_2\in \fm_x$ so that
$(v_1-v_2)|_{F} \in \cO_{F,x}(-C)$ and $(v_1-v_2)|_{Z_i} \in \cO_{Z_i,x}(-C)$ for $i=1,2$.
Hence Lemma \ref{lem.3term2} follows from Lemma \ref{keylem3D-2}.
We assume $Z_1\Cap Z_2$ at $x$. 
By Lemma \ref{keylem3D-2}, we may replace $F$ by any curve which is regular at $x$ and tangent to $F$ at $x$.
Hence we may take $F$ as in the beginning of the proof of Lemma \ref{keylem3D-2}.
Then take $H'$ and $F_i$ for $i=1,2$ as before.
Then Claim \ref{keylem3D-2.claim4} and ($\clubsuit 1$) imply $Z_i = \div_X(\pi- u_i f)$ locally at $x$ for $i=1,2$.
Hence we get $u_i-v_i\in \fm_x$ so that $(u_i-v_i)|_{F} \in \cO_{F,x}(-C)$.
Thus Lemma \ref{lem.3term2} follows from the second assertion of Claim \ref{keylem3D-2.claim0}.


\section{Proof of Key Lemma II}\label{keylemproofII}

Let $(X,C)$ be in $\cC$ (see Definition \ref{def.cC}) and $\{\Clam\}_{\lam\in I}$ be the set of irreducible
components of $C$. Fix a Cartier divisor 
\begin{equation*}
\displaystyle{D=\underset{\lambda\in I}{\sum} \mlam \Clam}
\qwith \mlam\geq 1.
\end{equation*}

\begin{lem}\label{keylem.cor}
Let $Z_a$ with $a\in \cP_{D,e}(F)$ be as Definition \ref{specialcurve} and take $x\in Z_a\cap C$ and $\lam\in I$ such that $x\in \Clam$. Assume 
\[\begin{aligned}
&H^1(X,\cO_X(-2D-C+ (e-1)F))=H^1(C,\cO_C(-2D+F))=0,\\
\end{aligned}\]
\begin{itemize}
\item[(1)]
Assuming $e\geq \mlam$, we have 
\[
\sym {1+{f^{e+1}}\cO_{Z_a,F\cap C}} {Z_a,x} \subset \tFXCC {2D-\Clam} + \FXCC {D+C}.
\]
\item[(2)]
Assuming $(p,2\mlam)=1$ and $e\geq \mlam(p^n-1)$ for a given integer $n>0$, we have 
\[
\sym {1+{f^{e+1}}\cO_{Z_a,F\cap C}} {Z_a,x} \subset \hFXCC {D+C} + p^n \hWU.
\]
\end{itemize}
\end{lem}

Lemma~\ref{keylemII} is a consequence of Lemma~\ref{keylem.cor} by the following remark.

\begin{rem}\label{keylem.cor.rem}
Assuming $p\not=2$, Lemma~\ref{keylem.cor} implies the following: 
\begin{equation}\label{keylem.re.eq}
\sym {1+{f^{e+1}}\cO_{Z_a,F\cap C}} {Z_a,x} \subset \hFXCC {D+C} + p^n \hWU 
\end{equation}
for $e\geq \mlam(p^n-1)$. Indeed, if $\mlam\geq 2$, then $2D-\Clam\geq D+C$ and Lemma~\ref{keylem.cor}(1) implies
\eqref{keylem.re.eq} for $e\geq \mlam$.
If $\mlam= 1$, then Lemma~\ref{keylem.cor}(2) implies \eqref{keylem.re.eq} for $e\geq \mlam(p^n-1)$.
\end{rem}

In this section we prove Lemma~\ref{keylem.cor}(1).
The proof of Lemma~\ref{keylem.cor}(2) will be complete in \S\ref{keylemproofIV}.

\begin{lem}\label{keylemma}
Let the notation be as in Definition \ref{specialcurve} and the paragraph after it.
Fix an integer $1\leq \epsilon\leq e-1$.
Assume 
\begin{equation}\label{keylem.eq0}
H^1(X,\cO_X(-2D-C+\epsilon F))=H^1(C,\cO_C(-2D  + F))=0.
\end{equation}
For $a\in \cP_{D,e}(F)$, we have
\[
\sym {1+\frac{\piD^2}{f^\ep}\cO_{Z_a,F\cap C}}{Z_a}
\subset 
\sym {1+\frac{\piD^2\pi}{f^\ep}\cO_{Z_a,F\cap C}} {Z_a} +
\sym {1+\frac{\piD^2}{f}\cO_{Z_a,F\cap C}} {Z_a} + \FXCC{2D}.
\]
\end{lem}

Recall $\pi, \piD, f\in \cO_{X,F\cap C}$ are such that locally at $F\cap C$
\[
C=\div_X(\pi),\; D=\div_X(\piD),\; F=\div_X(f)
\]
so that $(\pi,f)$ is a system of regular parameters in $\cO_{X,F\cap C}$. 

\begin{rem}\label{keylemma.rem}
By \eqref{specialcurve-eq1},
$\displaystyle{\frac{\piD^2}{f^{e-1}} \equiv u^{-2}f^{e+1}\in \cO_{Z_a,F\cap C}}$ so
$\displaystyle{\frac{\piD^2}{f^\ep}\in f^{e+1}\cO_{Z_a,F\cap C}}$.
\end{rem}
\def\ta{a'}

\begin{proof} 
We fix $a\in \cP_{D,e}(F)$ and write $Z=Z_a$. 
Take $1+\frac{\piD^2}{f^\ep}\alpha$ with $\alpha\in \cO_{Z,F\cap C}$.
By \eqref{keylem.eq0}, letting $A=(\ep-1)F|_{C} \subset C$, 
Lemma \ref{preliminarylem-1} implies that the natural map
\[
H^0(X,\cO_{X}(-2D+\ep F)) \to H^0(A,\cO_{X}(-2D+\ep F)|_{A})
\]
is surjective. The target group is isomorphic to 
\[
\cO_{X}(-2D+\ep F)\otimes \cO_{X,F\cap C}/(f^{\ep-1},\pi)
\simeq \cO_{X}(-2D+\ep F)\otimes \cO_{Z,F\cap C}/(f^{\ep-1},\pi),
\]
where the isomorphism holds since 
$\cO_{X,F\cap C}(-Z)\subset (f^e,\pi) \subset (f^{\ep-1},\pi)$ by \eqref{specialcurve-eq1} and the assumption $\ep\leq e-1$. Hence there exist
\[\ta\in H^0(X,\cO_{X}(-2D+\ep F))\]
and $\beta,\gamma\in \cO_{Z,F\cap C}$ such that
\[
\ta|_{Z} =\frac{\piD^2}{f^\ep}\alpha + \frac{\piD^2}{f}\beta + \frac{\piD^2\pi}{f^\ep}\gamma
\;\in \cO_{Z,F\cap C}.
\]
This implies
\begin{equation}\label{keylem.eq3}
(1+\ta)|_{Z} =(1+\frac{\piD^2}{f^\ep}\alpha)(1 +  \frac{\piD^2}{f}\beta')
(1 + \frac{\piD^2\pi}{f^\ep}\gamma') \;\in \cO_{Z,F\cap C}^\times,
\end{equation}
where $\beta'=\beta v^{-1}$ with
\[
v=1+\frac{\piD^2}{f^\ep}\alpha = 1-\piD{f^{e-\ep}}u^{-1}\alpha\;\in \cO_{Z,F\cap C}^\times
\quad (\text{cf. Remark \ref{keylemma.rem}}) 
\]
and $\gamma'=\gamma v^{-1} (1 +  \frac{\piD^2}{f}\beta')^{-1}\in \cO_{Z,F\cap C}$. Put
\[
Z'=\div_X(1+\ta) + \ep F.
\]
By Lemma \ref{classgroup.lem2}(1), $Z'\in \ZXCC$ with $Z'\cap C= F\cap C$. Locally at $F\cap C$, 
\begin{equation}\label{keylem.eq4}
\ta=c\frac{\piD^2}{f^\ep}\qaq Z'=\div_X(f^{\ep}+\piD^2 \cdot c) \qwith c\in \cO_{X,F\cap C}.
\end{equation}

By \eqref{specialcurve-eq1}, letting $b=f^e+\piD\cdot u\;\in k(X)^\times$, we have  
\begin{equation}\label{keylem.eq1}
\div_X(b)= Z + G\qfor G\in \Div(X)\; \text{ such that } |G|\cap F\cap C=\varnothing.
\end{equation}
Noting $Z'\cap C= F\cap C$, we can apply Lemma \ref{classgroup.lem2}(3) to $1+\ta$ and $b$ to get
\[
\sym {1+\ta} {Z}- \ep\sym {b} {F} + \sym {b} {Z'}\in \FXCC {2D}.
\]
Hence \eqref{keylem.eq3} implies 
\[
\sym{1+\frac{\piD^2}{f^\ep}\alpha}{Z} + 
\sym{1 +  \frac{\piD^2}{f}\beta'}{Z} +
\sym{1 + \frac{\piD^2\pi}{f^\ep}\gamma'}{Z}
- \ep\sym {b} {F} + \sym {b} {Z'} \in \FXCC{2D}.
\]
Thus the proof of Lemma \ref{keylemma} is reduced to showing 
\begin{equation}\label{keylem.eq5}
-\ep\sym {b} {F} + \sym {b} {Z'} \in \FXCC{2D}.
\end{equation}

\begin{claim}\label{keylem.claim2}
There exists $b'\in \piD\cdot \cO_{X,F\cap C}^\times$ such that 
\begin{equation}\label{keylem.claim2.eq}
b|_{F\cup Z'}=b'|_{F\cup Z'}\in \cO_{F\cup Z',F\cap C}.
\end{equation}
\end{claim}
\begin{proof}
We have 
\begin{equation*}\label{keylem.eq5.1}
b= f^e + \piD u =u\piD(1-\piD f^{e-\ep}c u^{-1}) + f^{e-\ep}(f^{\ep}+ \piD^2 c)\;\in \;\cO_{X,F\cap C}.
\end{equation*}
The second term is divisible by a local equation of $F\cup Z'$ around $F\cap C$ due to \eqref{keylem.eq4}
and the assumption $\ep\leq e-1$. Thus it suffices to take $b'= u\piD(1-\piD f^{e-\ep}c u^{-1})$.
\end{proof}

By Lemma \ref{classgroup.lem2}(4) applied to $1+\ta$ and $b'$, we have
\[
-\ep\sym {b'} {F} + \sym {b'} {Z'} \in \FXCC{2D}.\]
which implies \eqref{keylem.eq5} since
$\sym {b} {F} =\sym {b'} {F}$ and $\sym {b} {Z'}=\sym {b'} {Z'}$ thanks to \eqref{keylem.claim2.eq}
and the fact $Z'\cap C=F\cap C$.
This completes the proof of Lemma \ref{keylemma}.
\end{proof}




\medbreak
Let the assumption be as in Lemma \ref{keylem.cor}. By Lemma \ref{keylemma} and Remark \ref{keylemma.rem},
\[
\begin{aligned}
\sym {1+{f^{e+1}}\cO_{Z_a,F\cap C}} {Z_a} 
&\subset \sym {1+\frac{\piD^2}{f^{e-1}}\cO_{Z_a,F\cap C}} {Z_a}\\
&\subset \sym {1 + \frac{\piD^2\pi}{f^{e-1}}\cO_{Z_a,F\cap C}}{Z_a} 
+ \sym {1 + \frac{\piD^2}{f}\cO_{Z_a,F\cap C}}{Z_a} +\FXCC {2D}\\
&\subset \sym {1 + \frac{\piD^2}{f}\cO_{Z_a,F\cap C}}{Z_a} +\FXCC {D+C}
\end{aligned}
\]
where the last inclusion holds since 
$\piD^2\pi/f^{e-1}\equiv -u^{-1}\piD\pi \in \cO_{Z_a,F\cap C}$ by \eqref{specialcurve-eq1}.
Hence Lemma \ref{keylem.cor} follows from the following.

\begin{lem}\label{keylemma-2}
Fix a regular closed point $x\in C$ and let $\lam\in I$ be such that $x\in \Clam$.
Let $F\in \ZXCC$ be such that $x\in F$ and $F\Cap C$ at $x$.
Take a local parameter $f$ (resp. $\pi$) of $F$ (resp. $C$) at $x$ so that 
$(\pi,f)$ is a system of regular parameters at $x$. 
Let $Z\in \ZXCC$ be such that locally at $x$,
\begin{equation}\label{specialcurve-eq2}
Z=\div_X(c\cdot f^e+\pi^m) \qwith c\in \cO_{X,x},
\end{equation}
where $m>0$ is an integer. Consider the object $(\tX,\tC)$ of $\cB_X$, where 
$g:\tX\to X$ is the blowup at $x$ and $\tC=g^*C$ (cf. Definition \ref{def.cB}).
Let $Z'\subset \tX$ be the proper transform of $Z$.
\begin{itemize}
\item[(1)]
If $e\geq m$, we have
\[
\frac{\piD^2}{f} \in \cO_{Z',Z'\cap \tC}(-g^*(2D-\Clam)).
\]
In particular we have
\[
\sym {1 + \frac{\piD^2}{f}\cO_{Z,x}}{Z,x} \in \tFXCC {2D-\Clam}.
\]
\item[(2)]
Assuming $(p,2\mlam)=1$ and $e\geq m(p^n-1)$ for an integer $n>0$, we have
\[
\sym {1 + \frac{\piD^2}{f}\cO_{Z,x}}{Z,x} \in \hFXCC {2D} + p^n \hWU.
\]
\end{itemize}
\end{lem}

\bigskip

In this section we prove Lemma \ref{keylemma-2}(1). The proof of Lemma \ref{keylemma-2}(2) will be given in 
\S\ref{keylemproofIV}. Note that Lemma \ref{keylem.cor}(1) follows from Lemma \ref{keylemma-2}(1). 

Let $\Clam'$ be the proper transform of $\Clam$ and $E=g^{-1}(x)$ be the exceptional divisor 
so that $g^*\Clam=\Clam'+E$. Let $F'$ be the proper transform of $F$ in $\tX$. 
We claim $Z'\cap F'\cap E=\varnothing$. Indeed $f/\pi$ (resp. $\pi$) is a local parameter of $F'$ (resp. $E$) at $F'\cap E$.
By \eqref{specialcurve-eq2} and the assumption $e\geq m$, $Z'$ is defined locally at $F'\cap E$ by
$c(f/\pi)^{e}\pi^{e-m} + 1$, which implies the claim.
Noting that $f$ is a local parameter of $E$ at any $y\in E-(F'\cap E)$, we have
\begin{equation*}\label{keylemma-2.eq1}
\frac{\piD^2}{f}\in \cO_{Z',Z'\cap E}(-2g^*D+E)\subset \cO_{Z',Z'\cap E}(-g^*(2D-\Clam))
\end{equation*}
noting $-2g^*D+E =-2g^*D+g^*\Clam-\Clam' =-g^*(2D-\Clam) -\Clam'$.
Lemma \ref{keylemma-2}(1) follows from this.


\section{Proof of Key Lemma III}\label{keylemproofIII}

Let $(X,C)$ be in $\cC$ (see Definition \ref{def.cC}) and $\{\Clam\}_{\lam\in I}$ be the set of irreducible
components of $C$. Fix a Cartier divisor 
\begin{equation*}
\displaystyle{D=\underset{\lambda\in I}{\sum} \mlam \Clam}
\qwith \mlam\geq 1.
\end{equation*}
The following lemma is a preliminary for the proof of 
Lemma \ref{movinglem} whose proof will be completed in \S\ref{keylemproofIV}.
In this section we prove only Lemma \ref{keylem2}(1),
which is necessary for the proof of Lemma \ref{keylem3D-3}.
Lemma \ref{keylemma-2}(2) and hence Lemma \ref{keylem.cor}(2) will be deduced from Lemma \ref{keylem3D-3}.
Lemma \ref{keylem2}(2) will be then deduced from Lemma \ref{keylem.cor}(2). 

\begin{lem}\label{keylem2}
Take a reduced $Z\in \ZXCC$ and $x\in Z\cap C$ and a dense open subset $V\subset X$ containing all generic points of $C$.
\begin{itemize}
\item[(1)]
We have (cf. Definition \ref{filtration.def0}(4))
\[
\sym {1+\cO_{Z,x}(-D)}{Z,x} \subset \FXCCapV D \;+\;  \tFXCC {2D-C}  + \FXCC {D+C},
\]
\item[(2)]
Assume $p\not=2$. For any integer $n>0$, we have
\[
\sym {1+\cO_{Z,x}(-D)}{Z,x} \subset \FXCCapV D \;+\;  \hFXCC {D+C} + p^n\hWU.
\]
\end{itemize}
\end{lem}

\begin{rem}\label{keylem2.rem}
\begin{itemize}
\item[(1)]
Letting (cf. \eqref{eqDkeylemmas0})
\begin{equation}\label{eqDkeylemmas}
C_0= \underset{\lam\in J}{\sum} \Clam \quad\text{with  }\;
J=\{\lam \in I\;|\; \mlam\geq 2\},
\end{equation}
Lemma \ref{keylem2}(1) implies
\[
\FXCC D \subset \FXCCapV D \;+\;  \tFXCC {D+C_0},
\]
noting $\FXCC {D+C}\subset \FXCC {D+C_0}\subset \tFXCC {D+C_0}$.
\item[(2)]
Assuming $p\not=2$, Lemma \ref{keylem2}(2) implies
\[
\FXCC D  \subset \FXCCapV D \;+\;  \hFXCC {D+C} + p^n\hWU.
\]
\end{itemize}
\end{rem}
\medbreak\noindent
\begin{proof}[Proof of Lemma \ref{keylem2}]
Take an integer $d>0$ large enough that the linear system $|d H-Z|$ on $X$
($H\subset X$ is a hyperplane section) is sufficiently very ample.
By Bertini's theorem there are $F_0,F_\infty \;\in \cL(d)=|dH|$ satisfying the following:
\begin{itemize}
\item[$(\sharp 1)$]
$H^1(C,\cO_C(-2D+F))=0$ for any $F\in \cL(d)$ (cf. Remark \ref{keylem.cor.rem}),
\item[$(\sharp 2)$]
$F_0=  Z + G$ for $G\in \ZXCC$ such that $G\Cap C$, $G\cap C\subset V$, $G\cap C \cap Z=\varnothing$,
\item[$(\sharp 3)$]
$F_\infty\Cap C$ and $F_\infty\cap C\subset V$, and 
$F_\infty \cap C\cap F_0 =\varnothing$.
\end{itemize}
Let $L\in Gr(1,\cL(d))$ be the line passing through $F_0$ and $F_\infty$ and consider the pencil 
$\{F_t\}_{t\in L}$. By $(\sharp 3)$ the following conditions hold:
\begin{itemize}
\item[$(\sharp 4)$]
$F_t \cap F_{t'}\cap C =\varnothing$ for $t\not= t'\in L$, 
\item[$(\sharp 5)$]
there exists a finite subset $\Sigma\subset L$ such that $F_t\Cap C$ and $F_t\cap C\subset V$ for $t\in L-\Sigma$.
\end{itemize}
Choose an identification $L\simeq \bP^1=\Proj(k[T_0,T_1])$ such that
$F_0=F_t$ with $t=(1:0)\in \bP^1$ and $F_\infty=F_t$ with $t=(0:1)\in \bP^1$. Put
\[W: =Z+ F_\infty\in \ZXCC.\]
Take $q=p^N$ with $N$ sufficiently large and a finite subset 
$S\subset \bP^1(\bF_q)\backslash (\{0,1,\infty\}\cup\Sigma)$ with $s=\deg_L(S)$ large enough that the following conditions hold for $F\in \cL(d)$:
\begin{itemize}
\item[$(\spadesuit 1)$]
$H^1(X,\cO_X(-2D-C+(q-2)F))=0$.
\item[$(\spadesuit 2)$]
$H^1(X,\cO_X(-D-W +(s-1) F))=H^1(W,\cO_W(-2D + (s-1) F))=0$.
\end{itemize}
Put 
\[
\Theta=- \underset{t\in \Lambda}{\sum} F_t+(q-2) F_\infty, \quad
\Lambda=\bP^1(\bF_q)-(S\cup\{0,\infty\}).
\]
We have $\cO_X(\Theta)\simeq \cO_X((s-1)F_\infty))$ so that $(\spadesuit 2)$ implies
\[H^1(X,\cO_X(-D + \Theta-W))=H^1(W,\cO_W(-D + \Theta)\otimes\cO_W(-D))=0.\]
Thus Lemma \ref{preliminarylem-1} implies that the natural map
\[r_W: H^0(X,\cO_X(-D - \underset{t\in \Lambda}{\sum} F_t +(q-2) F_\infty))
\to \frac{\cO_{W,W\cap C}(-D+\Theta)}{\cO_{W,W\cap C}(-2D + \Theta)}\]
is surjective. $(\sharp 2)$, $(\sharp 3)$ and $(\sharp 4)$ imply
\[Z\cap C\cap |\Theta|=\varnothing \qaq
\cO_{F_\infty,F_\infty\cap C}(\Theta) 
=\cO_{F_\infty,F_\infty\cap C}((q-2) F_\infty) .\]
Hence the target of $r_W$ is equal to
\[\frac{\cO_{Z,Z\cap C}(-D)}{\cO_{Z,Z\cap C}(-2D)}\;\oplus\;
\frac{\cO_{F_\infty,F_\infty\cap C}(-D+(q-2) F_\infty)}
{\cO_{F_\infty,F_\infty\cap C}(-2D+(q-2) F_\infty)}.\]
Since $F_\infty\Cap C$ we have an isomorphism
\begin{equation*}\label{keylem3.eq3}
\nu: \frac{\cO_{F_\infty,F_\infty\cap C}(-D+(q-2) F_\infty))}
{\cO_{F_\infty,F_\infty\cap C}(-D-C+(q-2) F_\infty))} 
\simeq \underset{x\in F_\infty\cap C}{\bigoplus} k(x).
\end{equation*}
Consider the composite map
\[
\mu: H^0(X,\cO_X(-D  +\Theta)) \to 
\cO_{F_\infty,F_\infty\cap C}(-D+(q-2) F_\infty))
\rmapo{\nu} \underset{x\in F_\infty\cap C}{\bigoplus} k(x).
\]
By the surjectivity of $r_W$, for given $\alpha\in \cO_{Z,x}(-D)$, one can find
\[
a\in H^0(X,\cO_X(-D - \underset{t\in \Lambda}{\sum} F_t +(q-2) F_\infty))
\] 
satisfying the following conditions:
\begin{itemize}
\item[$(\clubsuit 1)$]
$a|_{Z}= \alpha \mod \cO_{Z,x}(-2D)$ and $a|_{Z}\in \cO_{Z,y}(-2D)$ for all $y\in (Z\cap C)-x$,
\item[$(\clubsuit 2)$]
$\mu(a) = (1,\dots,1)$.
\end{itemize}
$(\clubsuit 2)$ implies $a\in \cP_{D,q-2}(F_\infty)$ and we put (cf. Definition \ref{specialcurve})
\[
Z_a=\div_X(1+a) + (q-2) F_\infty.
\]
Take rational functions 
\[
\rho=\frac{T_0}{T_1-T_0} \qaq b= 1 - \rho^{q-1}\text{  on } L=\bP^1.
\]
We have
\[\div_X(b)= \underset{t\in \bA^1(\bF_q)-\{0,1\}}{\sum} F_t + (1-q)F_1 + F_0 \qwith F_0=Z+G.\]
By $(\sharp 2)$, $(\sharp 3)$ and $(\sharp 4)$, no component of $\div_X(b)$ passes through $F_\infty\cap C$.
Noting 
\[b|_{F_\infty}=1\qaq
(1+a)|_{F_t}=1 \qfor t\in \Lam=\bP^1(\bF_q)-(S\cup \{0,\infty\}),\] 
Lemma \ref{classgroup.lem2}(1) implies
\begin{equation}\label{proof-keylem2.eq1}
W(X,C)\ni 0= \sym {1+a} {Z} + \sym {1+a} {G} + 
\underset{t\in S}{\sum} \sym {1+a} {F_t} + \sym {1-\rho^{q-1}} {Z_a}.
\end{equation}
$(\clubsuit 1)$ implies 
\[
\sym {1+a} {Z,x} -\sym {1+\alpha}{Z,x} \in \FXCC {2D},
\]
\[
\sym {1+a} {Z,y} \in \FXCC {2D}\qfor y\in (Z\cap C)-x.
\]
By $(\sharp 2)$ and $(\sharp 4)$, $F_\infty\cap G \cap C=\varnothing$ and $F_\infty\cap F_t\cap C=\varnothing$
for $t\in S$ so that
\[
a\in \cO_{X,G\cap C}(-D) \qaq
a\in \cO_{X,F_t\cap  C}(-D) \qfor t\in S.
\]
Finally, using $(\sharp 1)$ and $(\clubsuit 1)$, Lemma \ref{keylem.cor}(1) with $e=q-2$ implies 
\[\sym {1-\rho^{q-1}} {Z_a}\in \tFXCC {2D-C} + \FXCC {D+C} .\]
Recalling $G\Cap C\subset V$ and $F_t\Cap C\subset V$ for $t\in S$, this proves Lemma \ref{keylem2}(1)
since the third term of \eqref{proof-keylem2.eq1} lies in $\FXCCapV D$.
Lemma \ref{keylem2}(2) would follow from Lemma \ref{keylem.cor}(2) (cf. Remark \ref{keylem.cor.rem})
whose proof will be complete in \S\ref{keylemproofIV}.
\end{proof}





\section{Proof of Key Lemma IV}\label{keylemproofIV}
\bigskip

In this section we finish the proof of Lemma \ref{keylemma-2}(2).
Note that this completes the proof of Lemmas \ref{keylem.cor} and \ref{keylem2} and hence Lemma \ref{keylemII}.
Lemma \ref{keylemma-2}(2) follows from the following lemma
where we take $2D$ in Lemma \ref{keylemma-2} for $D$ in Lemma \ref{keylem3D-3}.
Let the notation be as in the beginning of \S\ref{keylemproofIII}. Recall 
\begin{equation*}
\displaystyle{D=\underset{\lambda\in I}{\sum} \mlam \Clam}
\qwith \mlam\geq 1.
\end{equation*}

\begin{lem}\label{keylem3D-3}
Fix a regular closed point $x\in C$ and an integer $n>0$. 
Let $F\in \ZXCC$ be such that $x\in F$ and $F\Cap C$ at $x$ and $\lam \in I$ be such that $x\in \Clam$.
Take a local parameter $f$ (resp. $\pi$) of $F$ (resp. $C$) at $x$ so that 
$(\pi,f)$ is a system of regular parameters at $x$. 
Let $Z\in \ZXCC$ be such that locally at $x$,
\begin{equation}\label{keylem3D-3.eq0}
Z=\div_X(u\cdot f^b+ \pi^a) \qwith u\in \cO_{X,x},\;a,b\in \bZ_{>0}.
\end{equation}
Assume 
\begin{equation}\label{keylem3D-3.eq0b}
\text{$(p,\mlam)=1$ and $b\geq a(p^n-1)$.}
\end{equation}
Then we have (cf. Definition \ref{filtration.def0})
\[
 \sym {1+\cO_{Z,x}(-D+F)} {Z,x} \in \hFXCC {D} + p^n\hWU. 
\]
\end{lem}


For the proof of Lemma \ref{keylem3D-3}, we need a preliminary lemma.

\begin{lem}\label{keylem2-cor}
For a fixed integer $n>0$, put
\[
C_P= \underset{\lam\in J_P}{\sum} \Clam \qwith  \;
J_P=\{\lam \in I\;|\; \frac{\mlam}{p^n}\in \bZ\}.
\]
For $Z\in \ZXCC$ and $x\in Z\cap C$, we have
\[
\sym {1+\cO_{Z,x}(-D)}{Z,x} \subset  \tFXCC {D+C_P} + p^n \FXCC C.
\]
\end{lem}
\begin{proof}
By Remark \ref{keylem2.rem}(1) we have (noting $C_P\subset C_0$)
\[
\sym {1+\cO_{Z,x}(-D)}{Z,x} \subset  \tFXCC {D+C_P} +  \FXCCap D .
\]
Hence the lemma follows from the following fact:
for integral $F\in \ZXCC$ such that $F\Cap C$ at $y\in F\cap C$, we have
\[
1+\cO_{F,y}(-D)\subset \big(1+\cO_{F,y}(-D-C_P)\big) \cdot \big(1+\cO_{F,y}(-C)\big)^{p^n}.
\]
Indeed, letting $\lam\in I$ be such that $y\in \Clam$, we have an isomorphism
\[
\frac{1+\cO_{F,y}(-D)}{1+\cO_{F,y}(-D-C_P)} \simeq 
 \frac{1+\cO_{F,y}(-\mlam\Clam)}{1+\cO_{F,y}(-(\mlam+1)\Clam)}\\ 
\]
if $\lam\in J_P$, and otherwise the group on the left hand side vanishes.
The desired fact is checked by using $p^n|\mlam$ if $\lam\in J_P$ and the perfectness of $\k(y)$
(use the equality $(1+a)^{p^n}=1+a^{p^n}$).
\end{proof}

\noindent
{\it Proof of Lemma \ref{keylem3D-3}}\;\;
Taking a local parameter $\piD$ of $D$ at $x$, we want to show
\begin{equation}\label{keylem3D-3.eq1}
 \sym {1+\alpha\frac{\piD}{f}} {Z,x} \in \hFXCC {D}+ p^n\hWU\qfor \alpha\in \cO_{X,x}. 
\end{equation}
For integers $m\geq 0$, we inductively define 
\[
(X_m,C'_m,E_m,x_m)
\]
as follows. For $m=0$, $X_0=X$, $C'_0=C$, $E_0=\varnothing$, $x_0=x$.
For $m=1$, let $g_1:X_1=Bl_x(X) \to X$ be the blowup at $x$, $C'_1\subset X_1$ be 
the proper transform of $C$, $E_1=g_1^{-1}(x)$ be the exceptional divisor, and $x_1=C'_1\cap E_1$.
Assuming $(X_{m-1},C'_{m-1},E_{m-1},x_{m-1})$ defined, let $g_m:X_m=Bl_{x_{m-1}}(X_{m-1}) \to X_{m-1}$,
$C'_m\subset X_m$ be the proper transform of $C'_{m-1}$, $E_m=g_m^{-1}(x_{m-1})$, and $x_m=C'_m\cap E_m$. Let 
\[
\phi_m=g_m\circ \cdots\circ g_1 : X_m\to X,
\]
be the composite map, $C_m=\phi_m^{-1}(C)_{\rm red}$, and $E_{i,m}\subset X_m$ be 
the proper transform of $E_{i}\subset X_i$ for $1\leq i\leq m-1$. 
We also define $E_{0,1}$ as the proper transform of $F$.
Note that $(X_m,C_m)$ is an object of $\chBX$ but not in $\cB_X$ (cf. Definition \ref{def.cB}).
We easily see that the following facts hold for $m\geq 1$:
\begin{itemize}
\item[$(*1)$]
$(\pi/f^m,f)$ is a system of regular parameters of $X_m$ at $x_m$.
\item[$(*2)$]
$f$ is a local parameter of $E_m\subset X_m$ at any point $y\in E_m\backslash E_{m-1,m}$.
\item[$(*3)$]
Let $Z_m\subset X_m$ be the proper transform of $Z$ in \eqref{keylem3D-3.eq0}.
If $b\geq am$, $Z_m$ is defined locally around $E_m\backslash E_{m-1,m}$ by
$(\pi/f^m)^a+ u f^{b-am}$. We have
\[
b\geq am \;\Longleftrightarrow\; Z_m\cap \phi_m^{-1}(x) \subset E_m\backslash E_{m-1,m},
\]
\[
b\geq am+1 \;\Longleftrightarrow\; Z_m\cap \phi_m^{-1}(x) = x_m.
\]
\item[$(*4)$]
$\phi_m^*C=C'_m + mE_m +\underset{1\leq i\leq m-1}{\sum}i E_{i,m}$,
\end{itemize}
By the assumption $(p,\mlam)=1$, there is a (unique) integer $m$ such that 
$1\leq m\leq p^n-1$ and $p^n|m  \mlam-1$. Take $y\in Z_m\cap \phi_m^{-1}(x)$. 
By \eqref{keylem3D-3.eq0b} we have $b\geq a(p^n-1)\geq am$ 
and $(*2)$, $(*3)$ and $(*4)$ imply
\[
\frac{\piD}{f}\in \cO_{Z_m,y}(-\phi_m^*D+E_m)\qaq
\phi_m^*D-E_m\geq 0
\]
and hence
\[
\sym {1+\alpha\frac{\piD}{f}} {Z_m,y}\in F^{(\phi_m^*D-E_m)}W(X_m,C_m)
\qfor \alpha\in \cO_{X,x}.
\]
By $(*4)$ the multiplicity of $E_m$ in $\phi_m^*D-E_m$ is $m\mlam-1$.
Since $p^n|m\mlam-1$, Lemma \ref{keylem2-cor} with 
Remark \ref{filtration.rem} implies 
\[
\sym {1+\alpha\frac{\piD}{f}} {Z_m,y}\in \tFB^{(\phi_m^*D)}W(X_m,C_m)\;+\; p^nF^{(C_m)}W(X_m,C_m).
\]
Noting $F^{(C_m)}W(X_m,C_m)\subset \hWU$, this implies \eqref{keylem3D-3.eq1} and completes the proof of Lemma \ref{keylem3D-3}.


\section{Proof of Key Lemma V}\label{keylemproofV}
\bigskip

Let the notation be as in \S\ref{keylemmas}.
In this section we prove Lemma \ref{movinglem}. 

\begin{lem}\label{movinglem2-lem1}
Let $g:\tX\to X$ be the blowup at a closed point of $C$ and $\tC=g^{-1}(C)_{\rm red}$. 
Then, for any integer $N>0$, we have 
\begin{equation}\label{movinglem2-lem1.eq0}
\FXtCCap {g^*D}\subset \FXCC D \;+\; \FtXCC{g^*(D+N\cdot  C)}.
\end{equation}
\end{lem}
\begin{proof}
Put $\tD=g^*D\in \ZXCCt$. It suffices to show 
\[
\sym{1+\cO_{F,y}(-\tD)}{F,y} \subset \FXCC D +\FtXCC {g^*(D+N\cdot C)}
\]
for any reduced $F\in \ZXCCt$ such that $F\Cap \tC$ and for any $y\in F\cap \tC$. 

\begin{claim}\label{movinglem2-lem1.claim1} Let $x=g(y)$. We may assume $\k(x)=\k(y)$.
\end{claim}
Indeed, take a finite Galois extension $k'$ of $k$ and consider the diagram
\[
\xymatrix{
\tX'  \ar[d]_{g'} \ar[r]^{\tilde{\phi}} &  \tX \ar[d]^{g} \\
X' \ar[r]_{\phi} &  X
}
\]
where $X'=X\otimes_k k'$ and $\tX'=\tX\otimes_k k'$.
Let $C'=\phi^*C$ and $\tC'=\tilde{\phi}^*\tC$. Note that these are reduced since $\phi$ and $\tilde{\phi}$ 
are \'etale. Take a point $y'\in \tX'$ lying over $y$ and let $x'=g'(y')$.
Taking $k'$ large enough, we may assume $\k(x')=\k(y')$. 
Set $F'=\tilde{\phi}^*F \in \ZXCCtd$. Then $F'$ is reduced and $F'\Cap \tC'$. Thus we may assume 
\begin{equation}\label{movinglem2-lem1.eq1}
\sym{1+\cO_{F',y'}(-{g'}^*D')}{F',y'} \subset \FXdCC {D'}  + F^{{g'}^*(D'+ N \cdot C')}W(\tX',\tC'),
\end{equation}
where $D'=\phi^*D$. Note
\begin{equation}\label{movinglem2-lem1.eq2}
{g'}^*(D'+ N \cdot C')= {g'}^*\phi^*(D+ N \cdot C)={\tilde{\phi}^*g^*(D+ N \cdot C)}.
\end{equation}
We have a commutative diagram
\[
\xymatrix{
1+\cO_{F',y'}(-\tilde{\phi}^*\tD) \ar[r]^{N_{F'/F}} \ar[d]^{\sym{\;,\;}{F',y'}} &
1+\cO_{F,y}(-\tD)  \ar[d]^{\sym{\;,\;}{F,y}} \\
F^{(\tilde{\phi}^*\tD)}W(\tX',\tC')   \ar[r]^{N_{\tilde{\phi}}} 
&  F^{(\tD)}W(\tX,\tC)   \\
F^{(\phi^*D)}W(X',C')  \ar[r]^{N_\phi}  \ar[u]_{\hookrightarrow}&  
\FXCC {D}\ar[u]_{\hookrightarrow} 
}
\]
where $N_\phi$ (resp. $N_{\tilde{\phi}}$) is induced by the norm map \eqref{eq.Wnorm} for $\phi$ 
(resp. $\tilde{\phi}$), and $N_{F'/F}$ is the norm map induced by $F'\to F$.
Since $F'\to F$ is \'etale, $N_{F'/F}$ is surjective after replacing $\cO_{F',y'}$ and $\cO_{F,y}$ by their
henselizations. Noting Remark \ref{filtration.rem}, \eqref{movinglem2-lem1.eq0} follows from 
\eqref{movinglem2-lem1.eq1} and \eqref{movinglem2-lem1.eq2} and 
\[N_{\tilde{\phi}}\big(F^{\tilde{\phi}^*g^*(D+ N \cdot C)}W(\tX',\tC') \big) \subset\FtXCC{g^*(D+N\cdot  C)}.\]
This completes the proof of Claim \ref{movinglem2-lem1.claim1}.
\medbreak

Now we prove Lemma \ref{movinglem2-lem1} assuming $\k(x)=\k(y)$. 
By Lemma \ref{keylem4-c} below, for any integer $m>0$, there exists $G\in \ZXCC$ such that $G$ is regular at $x$ and $(G',F)_y\geq m+1$ where $G'$ is the proper transform of $G$ in $\tX$. Lemma \ref{keylem3D-2} implies
\[
 \sym{1+\cO_{F,y}(-\tD)}{F,y}\subset \sym{1+\cO_{G',y}(-\tD)}{G',y}+ \FtXCC {\tD+m\tC}.
\]
We have $m\tC\geq N\cdot g^*C$ for $m$ sufficiently large so that
\[
 \sym{1+\cO_{F,y}(-\tD)}{F,y}\subset \sym{1+\cO_{G',y}(-\tD)}{G',y}+ 
\FtXCC {g^*(D+N\cdot C)}.
\]
Since $G$ is regular at $x=g(y)$, $G'\to G$ is an isomorphism at $y$. 
Hence we have
\[\sym{1+\cO_{G',y}(-g^*D)}{G',y}=\sym{1+\cO_{G,x}(-D)}{G,x}\text{  in } W(U).\]
This completes the proof of Lemma \ref{movinglem2-lem1}.
\end{proof}

\begin{lem}\label{keylem4-c}
Let $X$ be a smooth surface over $k$ and $x\in X$ be a closed point and $g:X' \to X$ be the blowup at $x$
with $E=g^{-1}(x)$. Let $x'$ be a closed point of $E$ such that $\k(x)=\k(x')$.
Let $F\subset X'$ be an integral curve such that $x'\in F$ and $F\Cap E$ at $x'$.
Then, for any integer $m>0$, there exists a curve $G\subset X$ such that $G$ is regular at $x$ and $(G',F)_{x'}\geq m+1$ where $G'$ is the proper transform of $G$ in $X'$.
\end{lem}
\begin{proof}
We can take a system $(s,t)$ of regular parameters of $\cO_{X,x}$ such that letting $E_0=\div_X(s)$ and 
$E_0'\subset X'$ be the proper transform of $E_0$, $F\cap E_0'\cap E=\varnothing$.
Then $E-(E_0'\cap E) = \Spec(\k(x)[t/s])$ and there is $c\in \k(x)$ such that
$x'$ is given by the ideal $(t/s-c)$. 
For $h\in \widehat{\cO_{X,x}}\cong \k(x)[[s,t]]$ write 
\[ h = \sum_{i=0}^\infty h_i\qwith h_i=0\;\text{ or }\; h_i\in \k(x)[s,t]\;\text{ homogeneous of degree $i$.}\]
Then we define $in(h)= h_{i_h}$ with $i_h=\max\{i\; |\;  h_i\not=0\}$.
By \cite[Ch.VIII\S7 Bilinear Lemma]{ZS} we have the following.

\begin{claim}\label{bilinerlemma}
Assume $in(h)=QR$ for non-constant homogeneous $Q,R\in\k(x)[s,t]$ which have no common prime factor. 
Then there exist $q,r\in \k(x)[[s,t]]$ such that $h=qr$, $Q=in(q)$ and $R=in(r)$.
\end{claim}

Let $f\in \cO_{X,x}$ be a local equation of $g(F)\subset X$. Putting $d=i_f$, we can write
\[ in(f) = \underset{1\leq j\leq d}{\sum}\; a_j t^j s^{d-j} \qwith a_j\in \k(x).\]
Then we have 
\[E\times_{X'} F = \Spec\; \k(x)[z]/(\phi(z))\qwith
\phi(z) = \underset{1\leq j\leq d}{\sum}\; a_j z^j.\] 
Since $F\Cap E$ at $x'$, this implies a decomposition
$in(f) = (t-cs) H$ where $H\in \k(x)[t,s]$ is not divisible by $(t-cs)$.
By Claim \ref{bilinerlemma} there exist $g,f'\in \k(x)[[s,t]]$ such that $f=g f'$, $in(g)=(t-cs)$ and $in(f')=H$. 
Then, taking $G=\divXx{g'}$ (cf. Remark \ref{coro.Recpoint.rem})
for $g'\in \cO_{X,x}$ such that $g'\equiv g \mod (s,t)^{m+2}$ in $\widehat{\cO_{X,x}}$,
one easily see that $G$ satisfies the desired property of Lemma \ref{keylem4-c}.
\end{proof}
\medskip

Put (cf. \eqref{eqDkeylemmas0})
\begin{equation}\label{keylem4D.eq1}
C_0= \underset{\lam\in J}{\sum} \Clam \quad\text{with  }\;
J=\{\lam \in I\;|\; \mlam\geq 2\}.
\end{equation}
Take a dense open subset $V\subset X$ containing all generic points of $C$.

\begin{lem}\label{movinglem2-lem2}
\begin{itemize}
\item[(1)]
Let $g:\tX\to X$ be a map in $\cB_X$ (cf. Definition \ref{def.cB}).
Note that $\tC=g^*(C)$ is reduced. Then
\begin{equation}\label{movinglem2-lem2.eq1}
\FtXCC {g^*D} \subset \FXCC D \;+\; \tFXCC{D+C_0}.
\end{equation}
\item[(2)]
For any integer $N>0$, we have (cf.\ Definition \ref{filtration.def0}(4))
\[
\tFXCC {D}\subset \FXCCapV D \;+\; \tFXCC{D+N\cdot C_0}.
\]
\end{itemize}
\end{lem}
\begin{proof}
To show \eqref{movinglem2-lem2.eq1}, we first assume $g$ is the blowup at a regular closed point $x$ of $C$.
By Remark \ref{keylem2.rem}(1) applied to $(\tX,\tC)$, we have
\[
\FtXCC {g^*D} \subset \FXtCCap {g^*D} \;+\; \tFtXCC{g^*D+\tC_0},
\]
where $\tC_0$ is defined for $g^*D$ as \eqref{keylem4D.eq1}.
Clearly $\tC_0=g^*C_0$ so that
\[
\tFtXCC{g^*D+\tC_0}=\tFtXCC{g^*(D+C_0)}=\tFXCC{D+C_0}, 
\]
where the second equality holds by \eqref{eq.tFXCC}. 
Hence \eqref{movinglem2-lem2.eq1} follows from Lemma \ref{movinglem2-lem1}.

Now we prove \eqref{movinglem2-lem2.eq1} in general case by induction on the number of 
blown-up points. Decompose $g$ as
\[
g:\tX \rmapo{\psi} X' \rmapo{\phi} X\qwith D'=\phi^*D,\; \tD=\psi^*D'=g^*D,
\]
where $\phi$ is in $\cB_X$  and $\psi$ is the blowup at a regular closed point of $C'=\phi^*C$. 
By the induction hypothesis applied for $\phi$, we have
\begin{equation}\label{movinglem2-lem2.eq2} 
\FXCCd {D'}\subset \FXCC {D} \;+\; \tFXCC{D+C_0},
\end{equation}
By applying to $\psi$ what we have shown, we get
\begin{equation}\label{movinglem2-lem2.eq3} 
\FtXCC {\tD}\subset \FXdCC {D'} \;+\; \tFXdCC{D'+C'_0},
\end{equation}
where $C'_0$ is defined for $D'$ as \eqref{keylem4D.eq1}. Noting that $\phi$ is in $\cB_X$, we see 
$\phi^*C_0 = C'_0$ and 
\[
\tFXdCC{D'+C'_0} = \tFXdCC{\phi^*(D+C_0)} = \tFXCC{D+C_0}.
\]
Thus \eqref{movinglem2-lem2.eq1}  follows from \eqref{movinglem2-lem2.eq2}  and \eqref{movinglem2-lem2.eq3} .
\medbreak

Next we show Lemma \ref{movinglem2-lem2}(2). By induction on $N$, we may assume $N=1$. 
By Remark \ref{keylem2.rem}(1) we have
\begin{equation*}\label{keylem4-1.eq1}
\FXCC D \subset \FXCCapV D \;+\; \tFXCC{D+C_0}.
\end{equation*}
Hence it suffices to show 
\[
\tFXCC {D}\subset \FXCC D \;+\; \tFXCC{D+ C_0}.
\]
This is a direct consequence of \eqref{movinglem2-lem2.eq1} and Definition \ref{filtration.def}.
\end{proof}

\begin{lem}\label{movinglem2-lem4}
\begin{itemize}
\item[(1)]
Let $g:\tX\to X$ be a map in $\chBX$ (cf. Definition \ref{def.cB}) and $\tC=g^{-1}(C)_{\rm red}$. Then, for any integer $N>0$,  
\begin{equation}\label{movinglem2-lem4.eq0}
\FtXCC {g^*D}\subset \FXCC D \;+\; \hFXCC{D+N\cdot  C_0}.
\end{equation}
\item[(2)]
For any integer $N>0$, 
we have (cf.\ Definition \ref{filtration.def0}(4))
\[
\hFXCC {D}\subset \FXCCapV D \;+\; \hFXCC{D+N\cdot C_0}.
\]
where $V$ is any open subset of $X$ such that $V\cap C$ is dense in $C$.
\end{itemize}
\end{lem}
\begin{proof}
(2) is shown by the same argument as the proof of Lemma \ref{movinglem2-lem2}(2)
using \eqref{movinglem2-lem4.eq0} instead of \eqref{movinglem2-lem2.eq1}.
We show \eqref{movinglem2-lem4.eq0}. Write $\tD=g^*D\in \ZXCCt$.
First we assume $g$ is the blowup at a closed point $x\in X$. Fix $N$ in the lemma.
By Lemma \ref{movinglem2-lem2}(2) applied to $(\tX,\tC)$ and $\tD$, for any integer $M>0$ we have
\begin{equation}\label{movinglem2-lem4.eq1}
\FtXCC {\tD} \subset \FXtCCap {\tD} \;+\; \tFtXCC {\tD+M\cdot\tC_0},
\end{equation}
where $\tC_0$ is defined for $\tD$ as \eqref{keylem4D.eq1}. 
We have $|g^{-1}(C_0)| \subset |\tC_0|$ so that $M\cdot\tC_0\geq N\cdot g^*C_0$ for 
$M$ large enough. Then
\begin{equation}\label{movinglem2-lem4.eq2}
\tFtXCC {\tD+M\cdot\tC_0} \subset \tFtXCC {g^*(D+N\cdot C_0)} \subset \hFXCC {D+N\cdot C_0}.
\end{equation}
Thus \eqref{movinglem2-lem4.eq0} follows from \eqref{movinglem2-lem4.eq1}, \eqref{movinglem2-lem4.eq2}
and Lemma \ref{movinglem2-lem1}.

Now we prove Lemma \ref{movinglem2-lem4} in general case by induction on the number of blown-up points.
The reduction argument is similar to that of the proof of Lemma \ref{movinglem2-lem2}(1). Decompose $g$ as
\[
g:\tX \rmapo{\psi} X' \rmapo{\phi} X\qwith D'=\phi^*D,\; \tD=\psi^*D'=g^*D,
\]
where $\phi$ is in $\chBX$ and $\psi$ is the blowup at a closed point of $\phi^{-1}(C)$. 
By the induction hypothesis applied for $\phi$, we have
\begin{equation}\label{keylem4-2.claim1.eq1}
\FXCCd {D'}\subset \FXCC {D} \;+\; \hFXCC{D+N\cdot C_0},
\end{equation}
where $C'=\phi^{-1}(C)_{red}$. By applying to $\psi$ what we have shown, we get
\begin{equation}\label{keylem4-2.claim1.eq2}
\FtXCC {\tD}\subset \FXCCd {D'} \;+\; \hFXCCd{D'+M\cdot C'_0}
\end{equation}
for any integer $M>0$, where $C'_0$ is defined for $D'$ as \eqref{keylem4D.eq1}. 
We have $|\phi^{-1}(C_0)| \subset |C'_0|$ so that $M\cdot C'_0\geq N\cdot \phi^*C_0$ for 
$M$ large enough. Hence 
\[
\hFXCCd{D'+M\cdot C'_0}\subset \hFXCCd{\phi^*(D+N\cdot C_0)}=\hFXCC{D+N\cdot C_0}.
\]
Thus \eqref{movinglem2-lem4.eq0} follows from \eqref{keylem4-2.claim1.eq1} and \eqref{keylem4-2.claim1.eq2}.
\end{proof}

\begin{rem}\label{movinglem2-lem5.rem}
If $D\geq 2C$, then $C_0=C$ so Lemma \ref{movinglem2-lem4}(2) implies
\[
\hFXCC D  \subset \FXCCapV D \;+\;  \hFXCC {D+N\cdot C}\quad\text{for any } N>0.
\]
\end{rem}

\begin{lem}\label{movinglem2-lem6}
Let the assumption be as in Lemma \ref{movinglem2-lem4}(1). Assume $p\not=2$. 
For any integers $N,n>0$ we have 
\begin{equation}\label{movinglem2-lem65.eq1}
\FtXCC {g^*D}\subset \FXCC D \;+\; \hFXCC{D+N\cdot  C} + p^n\hWU.
\end{equation}
\end{lem}
\begin{proof}
Put $\tD=g^*D$. By the same induction argument as the proof of Lemma \ref{movinglem2-lem4}(1), 
we are reduced to the case where $g$ is the blowup at a closed point $x\in C$.
By Remark \ref{keylem2.rem}(2) applied to $\tX,\tC,\tD$, we have
\[
\FtXCC {\tD} \subset \FXtCCap {\tD} \;+\; \hFXtCC {\tD+\tC} + p^n\hWU.
\]
Noting $\tD+\tC\geq 2\tC$, Remark \ref{movinglem2-lem5.rem} applied to $\tX,\tC,\tD+\tC$ implies
\[
\hFXtCC {\tD+\tC} \subset \FXtCCap {\tD} \;+\; \hFXtCC {\tD+M\cdot\tC}\quad\text{for any } M>0.
\]
Taking $M$ large enough so that $M\cdot\tC\geq N\cdot g^*C$, we have
\[
\hFXtCC {\tD+M\cdot\tC}\subset \hFXtCC {g^*(D+N\cdot C)}=\hFXCC {D+N\cdot C}.
\]
Thus we get
\[
\FtXCC {\tD} \subset \FXtCCap {\tD} \;+\; \hFXCC {D+N\cdot C} + p^n\hWU.
\]
Now \eqref{movinglem2-lem65.eq1} follows from Lemma \ref{movinglem2-lem1}.
\end{proof}

Finally we prove Lemma \ref{movinglem}. By Remark \ref{keylem2.rem}(2) we have 
\[
\FXCC D  \subset \FXCCapV D \;+\;  \hFXCC {D+C} + p^n\hWU.
\]
Noting $D+C\geq 2C$, Remark \ref{movinglem2-lem5.rem} implies
\[
\hFXCC {D+C}  \subset \FXCCapV D \;+\;  \hFXCC {D+N\cdot C}.
\]
Thus it suffices to show
\[
\hFXCC D  \subset \FXCC D \;+\;  \hFXCC {D+N\cdot C} + p^n\hWU,
\]
which follows from Lemma \ref{movinglem2-lem6} and Lemma \ref{filtration.lem}.

\section {Appendix}\label{appendix}

In this section we give a sketch of a proof of Lemma \ref{lemRes}.
For a ring $R$ over $\bF_p$ and for an integer $q\geq 0$, put 
\[S\widehat{C}K_q(R) = \projlim n \Ker\big( K^{sym}_q(R[T]/(T^n)) \to K_q(R)\big), \]
where $K^{sym}_q(R[T]/(T^n))$ is the subgroup of $K_q(R[T]/(T^n))$ generated by symbols.
For an integer $n\geq 0$ put 
\begin{equation}\label{filSCK}
Fil^n S\widehat{C}K_q(R)= \Ker\big( S\widehat{C}K_q(R) \to K^{sym}_q(R[T]/(T^{n+1}))\big). 
\end{equation}
Let $T\widehat{C}K_q(R)$ be the $p$-typical part of $S\widehat{C}K_q(R)$ as defined in \cite[Ch.II \S7]{B} (see also \cite[p.636]{Ka1}).
Letting $I_p$ be the set of positive integers not divisible by $p$, 
we have the projection to the $p$-typical part:
\begin{equation}\label{typicalpart}
\tau_R=\underset{n\in I_p}{\sum} -\frac{\mu(n)}{n}V_nF_n: S\widehat{C}K_{q+1}(R) \to T\widehat{C}K_{q+1}(R)
\end{equation}
(see \cite[Ch.II \S1 (2.1)]{B} for $F_n$ and $V_n$). 
Recall the Artin-Hasse exponential
\[E(T):=\underset{n\in I_p}{\prod}\; (1-T)^{-\mu(n)/n}\;\in \; T\widehat{C}K_1(R)\]
(Note $S\widehat{C}K_1(R)= 1+ T\cdot R[[T]]$ and $E(T)= \tau_R(1-T)$).
There is an isomorphism
\[ \psi: W(R) \isom T\widehat{C}K_1(R)\;;\; (a_0,a_1,a_2,\cdots) \mapsto \underset{n\geq 0}{\prod} E(a_nT^{p^n}).\]
It induces a pairing
\begin{equation}\label{TCKpairing}
W(R) \times K^{sym}_q(R) \to T\widehat{C}K_{q+1}(R)\;;\; (x,y_1,\dots,y_q) \mapsto (\psi(x),y_1,\dots,y_q).
\end{equation}

For a discrete valuation field $L$ with residue field $E$, let 
\begin{equation}\label{resLE}
Res_{L/E}:S\widehat{C}K_{q+1}(L) \rmapo{Res_{L/E}} S\widehat{C}K_{q}(E)
\end{equation}
be the residue map from \cite[\S2.5 Lemma 2.5]{Ka1}. 
For an $N$-dimensional local field $K$ as in \eqref{recK}, we define
\[Res_{K/\bF_p}: S\widehat{C}K_{N+1}(K) \to S\widehat{C}K_1(\bF_p)=(1+T\bF_p[[T]])^\times \]
as the composite $Tr_{k_0/\bF_p}\circ Res_{k_1/k_0}\circ Res_{k_2/k_1}\circ \cdots Res_{K/k_{N-1}}$, where
$Tr_{k_0/\bF_p}$ is the trace map $S\widehat{C}K_1(k_0)\to S\widehat{C}K_1(\bF_p)$. It induces natural maps (see $(a)$ below)
\[Res_{K/\bF_p}:T\widehat{C}K_{N+1}(K) \to T\widehat{C}K_1(\bF_p)=W(\bF_p), \]
\[Res^s_{K/\bF_p}:T\widehat{C}K_{N+1}(K) \rmapo{Res_{K/\bF_p}}  W(\bF_p)\to W_s(\bF_p)\quad(s\in \bZ_{\geq 1}). \]
By \cite[\S3.1 Th.2 and the argument on p.662]{Ka1}, the $p$-part of $\Psi_K$ is induced by \eqref{ASW.eq} and a pairing
\[ \pairII: W_s(K) \times \KM N K \to  W_s(\bF_p)\simeq  \bZ/p^s\bZ\]
which arises from the pairing (cf. \eqref{TCKpairing})
\[  W(K) \times \KM N K \to T\widehat{C}K_{N+1}(K) \rmapo{Res^s_{K/\bF_p}} W_s(\bF_p).\]
Thus the commutativity of \eqref{commutativitypairing} follows from that of the following diagram
\begin{equation*}
\begin{CD} 
\fil m W_s(K) & \times& \;\KM N K/V^m \KM N K & @>{\pairII}>> & \bZ/p^s\bZ \\
@VV{-F^sd}V  @AA{\rho^m_K}A && @AA{p^{s-1}}A \\
\fmK^{-m} \Omega_{\cO_K}^1\otimes_{\cO_K} E\; &\times& \;\fmK^{m-1}\Omega_{\cO_K}^{N-1}\otimes_{\cO_K} E 
& @>{\pairI}>> & \pz \;,\\
\end{CD}
\end{equation*}
where $F^sd$ is the map in \eqref{Fsd}. The latter commutativity follows from
\begin{equation}\label{eq.res}
\begin{aligned}
& Res^s_{K/\bF_p}\{E(c T^{p^{s-1-i}}), 1+ab_1\cdots b_{N-1}, b_1,\dots,b_{N-1}\}= -E(hT^{p^{s-1}})\\
&\qwith \; h= Res^{\Omega}_{K/\bF_p}\big(c^{p^i-1} dc\wedge adb_1\wedge \cdots\wedge db_{N-1} \big).\\
\end{aligned}
\end{equation}
for $a,b_1\dots,b_{N-1}\in \cO_K$, $c\in K$ and $i\in \bZ$ such that $v_K(a)\geq m-1$ and that
$p^i v_K(c)\geq -(m-1)$ with $0\leq i\leq s-1$ or $p^i v_K(c)\geq -m$ with $0\leq i\leq \min\{s,\ord_p(m)\}-1$.
The formula is checked by using the explicit computation of the residue map in \cite[\S2.2 and \S2.4]{Ka1}.
\medbreak

Here we give a sketch of the computation assuming:
\begin{enumerate}
\item[$(*)$] 
$K$ is the fraction field of the henselization of $F\{t,\pi\}$ at the prime ideal $(\pi)$, where
$F$ is a finite field and $F\{t,\pi\}$ is the henselization of $F[t,\pi]$ at $(t,\pi)$
(this is the only case we use in this paper).
\end{enumerate}
In this case we have $N=2$ and $k_1=E:=F\{t\}[1/t]$ and $k_0=F$. We identify $E$ with a subfield of $K$.
The formula \eqref{eq.res} is then reduced to the following formulas
\begin{equation}\label{eq2.res}
\begin{aligned}
& Res^s_{K/\F_p}\{E(c T^{p^{s-1-i}}), 1+a\pi^{m-1}t, t\} = -E(f(c,a) T^{p^{s-1}}) \\
& \qwith f(c,a)= Res^{\Omega}_{K/\F_p}\big(c^{p^i-1} dc\wedge a\pi^{m-1} dt\big)\in F,
\end{aligned}
\end{equation}
\begin{equation}\label{eq3.res}
\begin{aligned}
&Res^s_{K/\F_p}\{E(c T^{p^{s-1-i}}), 1+b\pi^m, \pi\}= -E(g(c,b) T^{p^{s-1}}) \\
& \qwith g(c,b)= Res^{\Omega}_{K/\F_p}\big(c^{p^i-1} dc\wedge b\pi^{m} d\pi\big),
\end{aligned}
\end{equation}
for $a,b\in E$ and $c=\gamma \pi^{-e}$ with $\gamma\in E$ and $e,i\in\bZ$ such that
$ep^i\leq m-1$ with $0\leq i\leq s-1$ or $e p^i \leq m$ with $0\leq i\leq \min\{s,\ord_p(m)\}-1$.
Noting
\[ c^{p^i-1} dc = \frac{\gamma^{p^i}}{\pi^{ep^i}}(\frac{d\gamma}{\gamma}-e \frac{d\pi}{\pi})\]
and that $ep^i\leq m$ and $p|e$ if $ep^i= m$, we compute
\begin{equation}\label{eq3.5.res}
 Res^{\Omega}_{K/\bF_p}\big(c^{p^i-1} dc\wedge a\pi^{m-1} dt\big)  = 
 \left.\left\{\begin{gathered}
 -e Res^{\Omega}_{E/\bF_p}(a\gamma^{p^i} dt ) \\ 
 0 \\
\end{gathered}\right.\quad
\begin{aligned}
&\text{if $m=e p^i+1$,}\\
&\text{otherwise,}
\end{aligned}\right.
\end{equation}
\begin{equation}\label{eq4.res}
 Res^{\Omega}_{K/\bF_p}\big(c^{p^i-1} dc\wedge b\pi^{m} d\pi\big)  = 
 \left.\left\{\begin{gathered}
 Res^{\Omega}_{E/\bF_p}(b\gamma^{p^i-1} d\gamma ) \\ 
 0 \\
\end{gathered}\right.\quad
\begin{aligned}
&\text{if $m=e p^i$,}\\
&\text{otherwise.}
\end{aligned}\right.
\end{equation}
In what follows we only prove \eqref{eq3.res} (\eqref{eq2.res} is proved similarly).
We will use the following facts (cf. \cite[\S2.2 and \S2.4]{Ka1}):
Let $L$ be a discrete valuation field with residue field $E$ and $Res_{L/E}$ be as in \eqref{resLE}.
\begin{itemize}
\item[$(a)$]
$Res_{L/E}$ preserves the filtration from \eqref{filSCK} and commutes with $F_n$ and $V_n$.
In particular $Res_{L/E}\circ\tau_L=\tau_E\circ Res_{L/E}$ (cf. \eqref{typicalpart}).
\item[$(b)$]
We have 
$S\widehat{C}K_{q+1}(\cO_L) \subset \Ker(Res_{L/E}).$
\item[$(c)$]
For $\alpha\in S\widehat{C}K_{q}(\cO_L)$ with its image $\overline{\alpha}\in S\widehat{C}K_{q}(E)$, we have
\[ Res_{L/E}(\{\alpha,\pi\}) = \overline{\alpha}.\]
\item[$(d)$]
Endow $S\widehat{C}K_{q}(L)$ with topology by the filtration \eqref{filSCK}.
For $s\in \Z_{>0}$ let $F^s T\widehat{C}K_{q}(L)$ be
the closed subgroup of $S\widehat{C}K_{q}(L)$ topologically generated by
\[ \{E(aT^{p^n}),r_1,\dots,r_{q-1}\}\;\text{ and }\;
\{E(aT^{p^n}),r_1,\dots,r_{q-2},T\} \qfor n\geq s,\]
where $a\in L$, $r_1,\dots,r_{q-1}\in L^\times$ (for the second element, see $(e)$ below). Put 
\[ T\Phi_sK_{q}(L)=F^s T\widehat{C}K_{q}(L)/F^{s+1} T\widehat{C}K_{q}(L).\]
Then there is a well-defined map 
\[ \phi_L^{q,s}: \Omega^{q-1}_L \to T\Phi_s K_{q}(L) \;;\;
a \dlog {r_1}\wedge \cdots \dlog {r_{q-1}} \mapsto \{E(aT^{p^s}),r_1,\dots,r_{q-1}\}.\]
We have $Res_{L/E}\big(F^s T\widehat{C}K_{q}(L)\big) \subset F^s T\widehat{C}K_{q}(E)$ and 
the diagram
\[\xymatrix{
\Omega^{q-1}_L \ar[r]^{\hskip -20pt \phi_L^{q,s}}\ar[d]_{Res^{\Omega}_{L/E}} 
& T\Phi_s K_{q}(L) \ar[d]^{Res_{L/E}}  \\
\Omega^{q-2}_E \ar[r]^{\hskip -20pt\phi_E^{q-1,s}} &  T\Phi_s K_{q-1}(E)\;. \\
}\]
is commutative (see \cite[\S2 Prop.3]{Ka1}).
\item[$(e)$]
For a ring $R$ over $\bF_p$ let $S(R)=\underset{q\geq 0}{\bigoplus} S_q(R)$ be the sub-graded ring of
\[ \underset{q\geq 0}{\bigoplus}\; K_q^{sym}(R[[T]][T^{-1}])\]
generated by the image of $\underset{q\geq 0}{\bigoplus}\; K_q^{sym}(R[[T]])$ and $T\in K_1(R[[T]][T^{-1}])$.
For $n\in \bZ_{\geq 1}$ let $S^{(n)}(R)=\underset{q\geq 0}{\bigoplus} S^{(n)}_q(R)$ be the graded ideal of $S(R)$ 
generated by $1+T^{n}R[[T]]\in S_1(R)$. Then there exists a natural map (see \cite[\S2.2 Lemma 4]{Ka1})
\[ \projlim n S^{(1)}_q(R)/S^{(n)}_q(R) \to S\widehat{C}K_{q}(R), \]
which is an isomorphism if $R$ is regular having a $p$-basis over $\bF_p$ (see \cite[\S2.2 Cor.1]{Ka1}).
In particular the symbol $\{x,T\}\in S\widehat{C}K_{q}(R)$ with $x\in S\widehat{C}K_{q-1}(R)$ makes a sense. 
We have $\{\tau_R(x),T\} = \tau_R\{x,T\}$ by \cite[Ch.II \S6 Prop.(1.1)(ii)]{B}.
\item[$(f)$]
For an $N$-dimensional local field $K$, \cite[(3) on p.679]{Ka1} and $(a)$ imply
\[\tau_{\F_p} Res_{K/\bF_p}\big(\{S\widehat{C}K_{N}(K),T\} \big)=0.\]
\end{itemize}

\begin{lem}\label{lem1.residue}
Let $K$ be as in $(*)$.
Take integers $d,j,e,m\geq 0$ such that $j|d$ and $je\leq m$.
Fix $\gamma,b\in E$ and put $\xi=1-\gamma\pi^{-e} T^{d/j}$ and 
define $\theta_k\in 1+T\cO_K[[T]]$ for $0\leq k\leq j$ inductively by
\[ \theta_0=1+b\pi^m,\; \theta_{k}= 1+(-1)^k(\theta_0\theta_1\cdots\theta_{k-1})^{-1} b\pi^{m-ke}\gamma^k T^{kd/j}.\]
Then we have 
\[ \{\xi, 1 + b\pi^m,\pi\} \equiv (-1)^j\{(-1)^{1+e} \gamma^{-1}\xi,\theta_j,\pi\}
\;\mod \Ker(\tau_{\F_p}Res_{K/\bF_p})\subset S\widehat{C}K_3(K) .\]
\end{lem}
\begin{proof}
For $A,B\in K[[T]]$ with $1+B\not=0$ we have a formula
\begin{equation}\label{eq1.lem1.residue}
  \{1+TA,1+B \} = -\{-B(1+TA), 1 + (1+B)^{-1} TAB\}\;\text{in } K_2(K[[T]][T^{-1}]).
\end{equation}
From this and $(e)$ we deduce
\[ \{\xi, 1 + b\pi^m,\pi\} = -\{-b\pi^m\xi,\theta_1,\pi\}= -\{(-1)^{m+1} b\xi,\theta_1,\pi\}\;\text{ in } S\widehat{C}K_3(K).\]
If $m-e>0$, $(b)$ and $(c)$ imply $\{(-1)^{m+1} b,\theta_1,\pi\}\in \Ker(Res_{K/\bF_p})$ so that 
\[ \{\xi, 1 + b\pi^m,\pi\} \equiv -\{\xi,\theta_1,\pi\}\;\mod \Ker(Res_{K/\bF_p}).\]
Repeat the same argument by noting $m-ke>0$ for $k<j$ and the facts 
\begin{equation*}\label{eq2.lem1.residue}
\begin{aligned}
& \{a,\theta_{k},\pi\}\;\in \;
\Ker(Res_{K/\bF_p})\;\;\text{for $ k <j$ and $a\in \cO_K[[T]]^\times$},\\
&\{T,\theta_{k},\pi\}\;\in \;\Ker(\tau_{\F_p}Res_{K/\bF_p})
\;\;\text{for $k \leq j$},\\
\end{aligned}
\end{equation*}
which follows from $(b)$, $(c)$ and $(f)$, one obtains 
\begin{equation*}
\begin{aligned}
\{\xi, 1 + b\pi^m,\pi\} & \equiv (-1)^{j-1}\{\xi,\theta_{j-1},\pi\}\;\mod 
\Ker(\tau_{\F_p}Res_{K/\bF_p})\\
& = (-1)^{j}\{(1-\theta_{j-1})\xi,\theta_{j},\pi\}\\
& \equiv (-1)^{j}\{(-1)^j b\gamma^{j-1}\pi^{m-(j-1)e}\xi,\theta_{j},\pi\}\;\mod 
\Ker(\tau_{\F_p}Res_{K/\bF_p})\\
& = (-1)^{j}\{(-1)^j (1-\theta_{j})^{-1}b\gamma^{j-1}\pi^{m-(j-1)e}\xi,\theta_{j},\pi\}\\
&\equiv (-1)^{j}\{-\pi^e\gamma^{-1}\xi,\theta_j,\pi\} \;\mod 
\Ker(\tau_{\F_p}Res_{K/\bF_p})\\
& = (-1)^{j}\{(-1)^{1+e}\gamma^{-1}\xi,\theta_j,\pi\}.
\end{aligned}
\end{equation*}
This completes the proof of Lemma \ref{lem1.residue}.
\end{proof}

Now we prove \eqref{eq3.res}. We take $d=p^{s-1}$ and $j=p^i$ in Lemma \ref{lem1.residue}. 
By the lemma and $(a)$, putting $\ep=(-1)^{1+e}$,
we are reduced to showing
\begin{equation}\label{eq7.res}
\tau_{\F_p}Res^s_{K /\bF_p}\{\ep \gamma^{-1}\xi,\theta_j,\pi\}= E(h T^{p^{s-1}}),
\end{equation}
\begin{equation}\label{eq.h}
 h =  \left.\left\{\begin{gathered}
 Res^{\Omega}_{E/\bF_p}(b\gamma^{p^i} \dlog{\gamma} ) \\ 
 0 \\
\end{gathered}\right.\quad
\begin{aligned}
&\text{if $m=e p^i$,}\\
&\text{if $m > e p^i$}
\end{aligned}\right.
\end{equation}
(note $(-1)^j=-1$ by the assumption $p\not=2$).
we have
\[\{\xi,\theta_j,\pi\}\in \{S\widehat{C}K_2(K),T\} + Fil^{p^{s-1}}S\widehat{C}K_3(K)
\;\;\text{ (cf. \eqref{filSCK})}.\]
Indeed, writing $A=-\gamma\pi^{-e}$ and $B=(-1)^j(\theta_0\theta_1\cdots\theta_{j-1})^{-1} b\pi^{m-ke}\gamma^j$, 
\[ \{\xi,\theta_j\}= \{1+AT^{d/j},1+BT^d\}=-\{-BT^d(1+AT^{d/j}), 1+ (1+BT^d)^{-1}ABT^{d+d/j}\}\]
by \eqref{eq1.lem1.residue}. In view of $(f)$ we get
\[\tau_{\bF_p} Res_{K/\bF_p}\{\xi,\theta_j,\pi\}\;\in \;
\tau_{\bF_p}\big(Fil^{p^{s-1}}S\widehat{C}K_1(\bF_p)\big) = \Ker\big(W(\bF_p) \to W_s(\bF_p)\big),\]
where the equality follows from \cite[Ch.I \S1 Remark 4.3 (ii)]{B}.
Recalling 
\[\theta_j= 1-(\theta_0\theta_1\cdots\theta_{j-1})^{-1} 
b\pi^{m-je}\gamma^{p^i} T^{p^{s-1}}, \]
$(a)$ and $(c)$ imply (recall $j=p^i$)
\begin{equation*}\label{eq5.res}
 Res_{K /E}\{\ep\gamma^{-1},\theta_j,\pi\} = 
\left.\left\{\begin{gathered}
 \{1- b\gamma^{p^i} T^{p^{s-1}},\ep\gamma\}, \\ 
 0.\\
\end{gathered}\right.\quad
\begin{aligned}
&\text{if $m=e p^i$,}\\
&\text{otherwise.}
\end{aligned}\right.
\end{equation*}
Hence \eqref{eq7.res} follows from  
\begin{equation*}\label{eq6.res}
\tau_{\F_p}Res^s_{E/\bF_p}\{1- b\gamma^{p^i} T^{p^{s-1}},\ep\gamma\}=
Res^s_{E/\bF_p}\{E(b\gamma^{p^i} T^{p^{s-1}}),\ep\gamma\}= E(h T^{p^{s-1}}),
\end{equation*}
where $h$ is from \eqref{eq.h}. The first (resp. second) equality follows from $(a)$ 
(resp. $(d)$). This completes the proof.

\bigskip


\bigskip



\begin{thebibliography}{CGGHL}

 \bibitem[B]{B}
Bloch, S., 
 {\it Algebraic $K$-theory and crystalline cohomology}.
 Inst.\ Hautes Etudes Sci.\ Publ.\ Math.\ {\bf 63} (1977), 187--268.


 \bibitem[BK]{BK}
 Bloch, S., Kato, K.,
 {\it $p$-adic \'etale cohomology}.
 Inst.\ Hautes Etudes Sci.\ Publ.\ Math.\ {\bf 63} (1986), 107--152.



\bibitem[dJ]{dJ}
J. de Jong, 
{\it Smoothness, semi-stability and alterations}, 
Inst.\ Hautes Etudes Sci.\ Publ.\ Math.\  {\bf 83} (1996), 51--93.







\bibitem[D]{D2}
P. Deligne, 
{\it   Finitude de l'extension de $\mathbb{Q}$ engendr\'ee par des 
traces de Frobenius, en caract\'eristique  finie},   
Volume dedicated to the memory of I.~M.~Gelfand,  Moscow Math. J. {\bf 12} (2012) no. 3.



\bibitem[Dr]{Dr} V. Drinfeld, 
{\it On a conjecture of Deligne},   
Volume dedicated to the memory of I.~M.~Gelfand,
Moscow Math. J. {\bf 12} (2012) no. 3.



\bibitem[EK]{EK}
H. Esnault, M. Kerz,
{\it A finiteness theorem for Galois representations of function
fields over finite fields (after Deligne)},
Acta Mathematica Vietnamica {\bf 37}, Number 4 (2012), p. 351--362.



\bibitem[ESV]{ESV}
Esnault, H., Srinivas, V., Viehweg, E. 
{\it The universal regular quotient of the Chow group of points on projective varieties}, 
Invent. math. {\bf 135} (1999), 595--664.







\bibitem[KS]{KS}
K. Kato and S. Saito,
{\it Global class field theory of arithmetic schemes},
Contemporary Math.\ {\bf 55} (1986), 255--331.









\bibitem[Ka1]{Ka1}
K. Kato, 
{\it A generalization of local class field theory by using $K$-groups, II},
J. Fac. Sci.Univ. Tokyo, Sect.IA Math. {\bf 27} (1980), 603--683.

   
\bibitem[Ka2]{Ka}
K. Kato, 
{\it Swan conductors for characters of degree one in the imperfect residue field case},
Comtemp. Math., {\bf 83} 101--131 (1989).



\bibitem[Ka3]{Ka3}
K. Kato, 
{\it Class field theory, $D$-module, and ramification on higher dimensional schemes, Part I},
Amer. J. Math. {\bf 116} (1994), 757--784.

   

\bibitem[KS1]{KS}
K. Kato and S. Saito, 
{\it Two dimensional class field theory}, 
in: Galois groups and Their Representations, 
	Advanced Studies in Pure Math. {\bf 2} (1983), 103--152.

\bibitem[KS2]{KS2}
	K. Kato and S. Saito, 
    {\t Global class field theory of arithmetic schemes}, 
	Contemporary Math. {\bf 55} (1986), 255--331.

\bibitem[KSY]{KSY}
 B. Kahn, S. Saito, T. Yamazaki (with appendices by Kay Rulling) 
{\it Reciprocity sheaves, I},
preprint,  ArXiv:1402.4201

\bibitem[KeS]{KeSLef}
M. Kerz, S. Saito, 
{\it Lefschetz theorem for abelian fundamental group with modulus}, 
Algebra and Number Theory {\bf 8} (2014), 689--702.



\bibitem[KeSc]{KeSc}
M. Kerz and A. Schmidt,    
{\it Covering data and higher dimensional global class field theory},
J. of Number Theory {\bf 129} (2009), 2569--2599.



  


\bibitem[LW]{LevWei}
M. Levine, C. Weibel,
{\it Zero cycles and complete intersections on singular varieties},
J. Reine Angew. Math. {\bf 359} (1985), 106--120. 





\bibitem[Ma]{Ma}
S. Matsuda, 
{\it On the Swan conductors in positive characteristic},
Amer. J. Math., {\bf 119} 705--739 (1997).





\bibitem[Ra]{Ras}
W. Raskind, 
{\it Abelian class field theory of arithmetic schemes}, 
$K$-theory and algebraic geometry (Santa Barbara, CA, 1992), 85--187, 
Proc. Sympos. Pure Math., 58, Part 1, Amer. Math. Soc., Providence, RI, 1995. 





\bibitem[Ru]{Rus}
H. Russell, 
{\it Albanese varieties with modulus over a perfect field},
preprint, ArXiv 0902.2533.



\bibitem[Sa1]{Sa}
S. Saito, 
{\it Class field theory for two dimensional local rings},
in: Galois Representations and Arithmetic Geometry, 
Advanced Studies in Pure Math. {\bf 12} (1987), 343-373.

\bibitem[Sa2]{Sa2}
S. Saito, 
{\it Arithmetic on  two dimensional local rings},
Invent. Math. {\bf 85} (1986), 379-414.

\bibitem[Sa3]{Sa3}
S. Saito, 
{\it Class field theory for curves over local fields},
Journal of Number Theory. {\bf 21} (1985), 44-80.


\bibitem[TSa]{TSa}
T. Saito, 
{\it Wild Ramification and the Cotangent Bundle},
preprint, ArXiv:1301.4632


\bibitem[Sc]{Sc}
A. Schmidt, 
{\it Some applications of Wiesend's higher dimensional class field theory},
Math. Zeit. {\bf 256} (2007), 731--736 

\bibitem[SS]{SS}
A. Schmidt, M. Spiess, 
{\it Singular homology and class field theory of varieties over finite fields}, 
J. Reine Angew. Math. {\bf 527} (2000), 13--36. 

		
\bibitem[SaSa]{SaSa}
S. Saito, K. Sato,
{\it A finiteness theorem for zero-cycles over p-adic fields},  
Annals of Math. {\bf 172} (2010), 593--639.

\bibitem[Se1]{Se}
J.-P. Serre,
Corps Locaux,
Hermann, Paris 1968.



\bibitem[Se2]{Se2} 
J.-P. Serre, 
{\em Zeta and $L$ functions},  
Arithmetical Algebraic Geometry (Proc. Conf. Purdue Univ., 1963)  pp. 82--92 Harper \& Row, New York 1965.

\bibitem[SV]{SV}
A. Suslin, V. Voevodsky,
{\it Singular homology of abstract algebraic varieties},
Invent. Math. {\bf 123} (1996), 61--94.

\bibitem[Wi]{Wi}
G. Wiesend, 
{\it Class field theory for arithmetic schemes},
Math. Z. {\bf 256} (2007), no. 4, 717--729. 

  




\bibitem[SGA4]{SGA4}
M. Artin, A. Grothendieck, J.-L. Verdier, 
{\it  Th\'eorie des topos et cohomologie \'etale des sch\'emas}, 
Lecture Notes in Mathematics {\bf 269}, {\bf 270}, {\bf 305}, Springer-Verlag.



\bibitem[SGA7II]{SGA7}
P. Deligne and N. Katz,
{\it Groupes de Monodromie en G\'eom\'etrie Alg\'ebrique},
Lecture Notes in Math. {\bf 340}, Springer-Verlag.

\bibitem[ZS]{ZS}
O. Zariski and P. Samuel,
{\it Commutative Algebra II}
GTM {\bf 29}, Springer-Verlag.

\end{thebibliography}
\end{document}